\documentclass[11pt]{amsart}

\usepackage[text={475pt,660pt},centering]{geometry}

\usepackage{color}
\usepackage{esint,amssymb}
\usepackage{graphicx}
\usepackage{MnSymbol}
\usepackage{amsmath, mathtools}
\usepackage[colorlinks=true, pdfstartview=FitV, linkcolor=blue, citecolor=blue, urlcolor=blue,pagebackref=false]{hyperref}
\usepackage{microtype}

\usepackage{bm}
\usepackage{scalerel} 
\usepackage{dsfont}
\usepackage{mathrsfs}
\usepackage[font={footnotesize}]{caption}
\DeclareUnicodeCharacter{00A0}{~}

\definecolor{darkgreen}{rgb}{0,0.5,0}
\definecolor{darkblue}{rgb}{0,0,0.7}
\definecolor{darkred}{rgb}{0.9,0.1,0.1}

\makeatletter
\newtheorem*{rep@theorem}{\rep@title}
\newcommand{\newreptheorem}[2]{%
\newenvironment{rep#1}[1]{%
 \def\rep@title{#2 \ref{##1}}%
 \begin{rep@theorem}}%
 {\end{rep@theorem}}}
\makeatother

\newtheorem{theorem}{Theorem}
\newreptheorem{theorem}{Theorem}
\newreptheorem{proposition}{Proposition}
\newtheorem{proposition}{Proposition}
\newtheorem{lemma}[proposition]{Lemma}

\theoremstyle{remark}

\theoremstyle{definition}
\newtheorem{definition}[proposition]{Definition}
\newtheorem{remark}[proposition]{Remark}
\newtheorem{conjecture}[proposition]{Conjecture}

\numberwithin{equation}{section}
\numberwithin{proposition}{section}

\newcommand{\N}{\mathbb{N}}
\newcommand{\R}{\mathbb{R}}

\newcommand{\E}{\mathbb{E}}
\renewcommand{\P}{\mathbb{P}}
\newcommand{\F}{\mathcal{F}}
\newcommand{\Zd}{\mathbb{Z}^d}
\newcommand{\Rd}{{\mathbb{R}^d}}

\newcommand{\Ze}{\mathcal{Z}}
\newcommand{\di}{\mathrm{d}}

\newcommand{\ep}{\varepsilon}

\renewcommand{\a}{\mathbf{a}}
\newcommand{\ahom}{{\overbracket[1pt][-1pt]{\a}}}  
\newcommand{\ahominv}{{ \overbracket[1pt][-1pt]{\mathrm{inv} \, \a}}}

\renewcommand{\subset}{\subseteq}

\newcommand{\cu}{\square}

\renewcommand{\fint}{\strokedint}

\DeclareMathOperator{\dist}{dist}

\DeclareMathOperator{\var}{var}

\DeclareMathOperator{\supp}{supp}

\DeclareMathOperator{\size}{size}

\DeclareMathOperator{\argmin}{argmin}

\renewcommand{\bar}{\overline}
\renewcommand{\tilde}{\widetilde}

\newcommand{\indc}{\mathds{1}}

\renewcommand{\O}{\mathcal{O}}
\newcommand{\A}{\mathcal{A}}

\newcommand{\avsum}{\sum}

\begin{document}

\sloppy

\title{Quantitative Homogenization of Differential Forms}

\begin{abstract}
We develop a quantitative theory of stochastic homogenization in a general framework involving differential forms. Inspired by recent progress in the uniformly elliptic setting, the analysis relies on the study of certain sub- and superadditive quantities. We establish an algebraic rate of convergence for these quantities and an algebraic error estimate for the homogenization of the Dirichlet problem. Most of the ideas needed in this article come from two distinct theories, the theory of quantitative stochastic homogenization, and the generalization of the main results of functional analysis and of the regularity theory of second-order elliptic equations to the setting of differential forms.
\end{abstract}

\author[P. Dario]{Paul Dario}
\address[P. Dario]{School of Mathematical Sciences, Tel Aviv University, Ramat Aviv, Tel Aviv 69978, Israel}
\email{pauldario@mail.tau.ac.il}

\date{\today}

\maketitle

\setcounter{tocdepth}{1}
\tableofcontents

\section{Introduction}
The classical theory of stochastic homogenization focuses on the study of the second-order elliptic equation
\begin{equation} \label{ehom.01}
\nabla \cdot  \left( \a (x) \nabla u \right) = 0,
\end{equation}
where the environment $\a$ is a random, rapidly oscillating, uniformly elliptic coefficient field. The standard qualitative result states that, under appropriate assumptions of ergodicity and stationarity on the law of the environment, a solution $u_r$ of the equation~\eqref{ehom.01} in a ball converges almost surely as the radius of the ball tends to infinity to a deterministic solution $\bar{u}_r$ of the equation
\begin{equation} \label{ehom.02}
\nabla \cdot  \left( \ahom  \nabla \bar{u}_r \right) = 0,
\end{equation} 
where $\ahom$ is a constant, symmetric, definite-positive matrix called the homogenized environment, see~\cite{K1, PV1, Y1, JKO}. Developing a quantitative theory of stochastic homogenization drew a lot of attention in the recent years to the point that the theory is now well-understood, see e.g.~\cite{armstrong2017quantitative, AKM2, AM, AS, GNO, GO1, GO2,  GO5}.

The purpose of this article is to extend the theory to a more general setting of degenerate systems of equations involving differential forms. To introduce the problem, we fix a dimension $d \geq 2$ and an integer $k \in \{ 0 , \ldots, d \}$. We let $\Lambda^{k}\left( \Rd\right)$ be the set of $k$-alternating multilinear maps. These spaces are equipped with an exterior product, denoted by the symbol $\wedge$, which maps the space $\Lambda^{k_1}\left( \Rd\right) \times \Lambda^{k_2}\left( \Rd\right)$ to the space $\Lambda^{k_1 + k_2}\left( \Rd\right)$, and with a canonical scalar product for which the collection $\left( \di x_{i_1} \wedge \cdots \wedge \di x_{i_k} \right)_{1 \leq i_1 < \cdots < i_k \leq d}$ is an orthonormal basis.
Given a domain $U \subseteq \Rd$, a $k$-differential form $u$ on $U$ is a mapping from $U$ to $\Lambda^{k}\left( \Rd\right)$ which can be decomposed along the canonical basis of $\Lambda^{k}\left( \Rd\right)$ according to the formula $u := \sum_{1 \leq i_1 < \cdots < i_k \leq d } \, u_{i_1, \cdots, i_k} \, \di x_{i_1} \wedge \cdots \wedge \di x_{i_k}.$

There is an important notion of derivative associated with differential forms: the exterior derivative, denoted by d, which sends $k$-forms to $(k+1)$-forms, and is defined by the formula
\begin{equation*}
\di u := \sum_{i = 1}^{d}  \sum_{1 \leq i_1 < \cdots < i_k \leq d }  \frac{\partial u_{ i_1, \cdots ,  i_k}}{\partial x_i} \di x_i \wedge \di x_{i_1} \wedge \cdots \wedge  \di x_{i_k}.
\end{equation*}
We denote by $\mathcal{L}\left(\Lambda^k \left( \Rd \right), \Lambda^{d-k} \left( \Rd \right) \right)$ the set of linear maps from $\Lambda^k \left( \Rd \right)$ to $\Lambda^{d-k} \left( \Rd \right)$ and define an environment $\a$ to be a measurable map defined on $\Rd$, valued in $\mathcal{L}\left(\Lambda^k \left( \Rd \right), \Lambda^{d-k} \left( \Rd \right) \right)$, and
satisfying the following uniform ellipticity and symmetry assumptions: for each $x \in \Rd$ and each pair $p_1,p_2 \in \Lambda^r(\Rd)$,
\begin{equation} \label{symandelllassumption}
\lambda |p_1|^2  \leq \left| p_1 \wedge \a(x) p_1 \right| \leq \frac 1\lambda |p_1|^2 \hspace{5mm} \mbox{and} \hspace{5mm}
 p_1 \wedge \a(x) p_2 =  p_2 \wedge \a(x) p_1 ,
\end{equation}
where we use the symbol $|\cdot|$ to denote the canonical norm on the Euclidean vector spaces $\Lambda^k \left( \Rd \right)$ and~$\Lambda^d \left( \Rd \right).$ We assume that the environment is random and satisfies the assumptions of stationarity and finite range dependence stated in~\eqref{stationarityassumption} and~\eqref{independenceassumption}, denote by $\P$ the law of the environment $\a$ and by $\E$ the expectation with respect to the probability distribution $\P$ (see Section~\ref{section:NARH}). The objective of this article is then to study the large-scale behavior of the solutions of the equation
\begin{equation} \label{eq:formintroell}
\di \left( \a \di u \right) = 0 ~\mbox{in}~ U,
\end{equation}
which are the critical points associated with the functional
\begin{equation} \label{chapitre0.ehom.03}
\mathcal{J}(U,u) := \int_U  \di u \wedge \a \di u.
\end{equation} 
Let us note that this setting strictly contains the framework of uniformly elliptic equations~\eqref{ehom.01}. Indeed the case $k=0$ corresponds to the uniformly elliptic equation~\eqref{ehom.01}. On the other hand, this formalism is part of the more general framework of the elliptic systems of partial differential equations: since the dimensions of the space $\Lambda^k \left( \Rd \right)$ is finite, the equation $\di \left( \a \di u \right) = 0$ can be written as an elliptic system. However, the system~\eqref{eq:formintroell} is \emph{not} uniformly elliptic.

A motivation to study these systems comes from the specific case when $r = 1$ and the underlying space is $4$-dimensional: in this setting the system of equations in~\eqref{chapitre0.ehom.03} has the same structure as \emph{Maxwell's equations} (see e.g.\ \cite[Section~1.2]{levy_survey}), with yet a fundamental difference: here we assume $\a(x)$ to be Riemannian, that is, elliptic in the sense of~\eqref{symandelllassumption}, while for Maxwell's equations the underlying geometric structure is Lorentzian. Replacing a Lorentzian geometry by a Riemannian one, a procedure sometimes referred to as ``Wick's rotation'', is very common in constructive quantum field theory, see e.g.\ \cite[Section~6.1(ii)]{gj-book}. While the objects we study here are minimizers of a random Lagrangian, we believe that the techniques developed in this article will be equally informative for the study of the Gibbs measures associated with such Lagrangians.

The main result of this article is to prove a quantitative homogenization theorem for differential forms. Following the method developed for uniformly elliptic equations  in~\cite{armstrong2017quantitative}, we first introduce a subadditive energy $J$ defined by the formula, for every bounded domain $U \subset \Rd$ and every pair~$(p, q) \in \Lambda^{r}(\Rd) \times \Lambda^{(d-r)}(\Rd)$,
\begin{equation}
\label{e.def.J}
J(U ,p ,q) :=\sup_{v \in \A(U) } \fint_U \left( -\frac 12 \di v \wedge \a \di v -  p \wedge \a \di v +   \di v \wedge q \right),
\end{equation}
where the notation $\A(U)$ denotes the set of solutions of the system $\di \left( \a \di u \right) = 0$ over the set $U$ (see~\eqref{defsolution} below).

The energy $J$ is nonnegative and satisfies a subadditivity property with respect to the domain~$U$ stated in Proposition~\ref{basicpropJ}. In particular, if we let $\cu_n$ be the triadic cube centered at $0$ of sidelength $3^n$ (see Section~\ref{section2}), then the sequence
$
n \mapsto  \E \left[  J(\cu_n ,p ,q) \right]
$
decreases and converges as $n$ tends to infinity. Our first result identifies the limit in terms of an homogenized tensor $\ahom$, provides an algebraic rate of convergence and quantifies the stochastic integrability. We first introduce a notation to measure the latter parameter.

\begin{definition} \label{def.stochint1.1}
For each exponent $s > 0$, each constant $K > 0$, and each non-negative random variable $X$, we write
$
X \leq \O_s (K)$ if and only if $\E \left[ \exp \left( X^s/K^s \right) \right] \leq 2.
$
\end{definition}
We are now ready to state the first theorem of this article.

\begin{theorem}[Quantitative convergence of the energy functional $J$] \label{maintheorem}
Let $r$ be an integer in $\{ 1 , \ldots, d \}$. There exist an exponent $\alpha(d, \lambda) > 0$, a constant $C(d, \lambda) < \infty$ and a linear mapping $\ahom \in \mathcal{L} \left(\Lambda^{r}(\Rd), \Lambda^{(d-r)}(\Rd) \right)$ such that, for every integer $n \in \N$,
\begin{equation} \label{statementmaintheorem}
\sup_{(p, q) \in B_1\Lambda^{r}(\Rd) \times B_1\Lambda^{d-r}(\Rd)}  \left| J(\cu_n ,p , q) - \frac 12 p \wedge \ahom p -  \frac 12 \ahom^{-1} q \wedge q + \star \left( p \wedge q \right) \right| \leq \O_2 \left( C 3^{-n \alpha}\right),
\end{equation}
where $B_1\Lambda^{r}(\Rd)$ denotes the ball of radius one of the Euclidean space $\Lambda^{r}(\Rd)$, the symbol $\star$ denotes the Hodge star operator, defined in~\eqref{def:hodgestar}.
\end{theorem}

Once this result is established, we deduce the quantitative homogenization theorem which is stated below.

\begin{theorem}[Homogenization theorem] \label{homogenizationtheorem}
Let $U$ be a bounded smooth domain of $\Rd$, $r$ be an integer in $\{ 1 ,\ldots , d \} $, $\varepsilon$ be a real number in $(0, 1]$ and $f $ be a form whose coefficients are in the Sobolev space~$H^2(U)$. Then there exist two solutions $u^\varepsilon , u$ of the Dirichlet boundary value problems
\begin{align} \label{eq:(1.10)}
    & \left\{ \begin{aligned}
        \di \left( \a \left( \frac{\cdot}{\varepsilon}\right) \di u^\varepsilon \right) = 0 & \hspace{5mm} \mbox{in} ~U,\\
       \mathbf{t} u^\varepsilon = \mathbf{t} f & \hspace{5mm}  \mbox{on} ~\partial U,
    \end{aligned} \right. &&
    \begin{aligned}
\mbox{and}
    \end{aligned}&&
     \left\{ \begin{aligned}[l]
 	\di \left( \ahom \di u \right) = 0 & \hspace{5mm} \mbox{in} ~U, \phantom{ \left( \frac{\cdot}{\varepsilon}\right)} \\
       \mathbf{t} u = \mathbf{t} f & \hspace{5mm} \mbox{on} ~\partial U,
    \end{aligned}\right.
\end{align}
where the notation $\mathbf{t}u$ denotes the tangential trace of the form $u$ over the boundary $\partial U$ defined in Section~\ref{section2.1}, an exponent $\alpha := \alpha (d , \lambda) > 0$ and a constant $C := C(d , \lambda , U) < \infty$ such that 
\begin{equation*}
\left\| u^{\varepsilon} - u \right\|_{L^2\Lambda^r(U)} +  \left\| \di u^{\varepsilon} - \di u \right\|_{H^{-1}\Lambda^r(U)} \leq\O_2 \left( C  \left\| \di f \right\|_{H^1\Lambda^{r+1}  (U)}  \varepsilon^\alpha \right),
\end{equation*}
where the $L^2\Lambda^r(U)$, $H^{-1}\Lambda^r(U)$ and $H^{1}\Lambda^r(U)$-norms are the Sobolev norms for differential forms defined in Subsections~\ref{sec1.6.1} and~\ref{sectionH-1norms}.
\end{theorem}

The third main result of this article investigates the relation between homogenization and the natural duality structure of differential forms and shows a commutation property between these two structures. More precisely,  if one considers an environment $\a \in \Omega$ sending $r$-forms to $(d-r)$-forms and satisfying the assumptions of ellipticity and symmetry~\eqref{symandelllassumption}, then the inverse environment $\a^{-1}$ sends $(d-r)$-forms to $r$-forms and satisfies the same properties. This allows to define the dual energy, for any pair $(q, p) \in \Lambda^{d-r}(\Rd) \times  \Lambda^r(\Rd)$,
\begin{equation*}
J_{\mathrm{inv}}(\cu_m ,q ,p) := \sup_{u \in \A^{\mathrm{inv}}\left(\cu_m \right) } \fint_{\cu_m} \left( -\frac 12 \a^{-1} \di u \wedge \di u - \a^{-1} \di u  \wedge p +   q \wedge  \di u  \right),
\end{equation*}
where $\A^{\mathrm{inv}}\left(\cu_m \right)  $ is the set of solutions of the dual system $\di \left( \a^{-1} \di u \right) = 0$ over the set $U$ defined in~\eqref{def.Ainv}. In Section~\ref{section7}, we prove that the conclusion of Theorem~\ref{maintheorem} holds for the energy functional $J_{\mathrm{inv}}$: there exists a linear map $\ahominv$ which sends $(d-r)$-forms to $r$-forms such that, for each pair $(q , p) \in \Lambda^{(d-r)}(\Rd) \times  \Lambda^r(\Rd),$
\begin{equation*}
\sup_{(p, q) \in  B_1\Lambda^{d-r}(\Rd) \times B_1\Lambda^{r}(\Rd) }  \left| J_{\mathrm{inv}} \left( \cu_n , q , p \right) -  \frac 12  p \wedge \ahominv^{-1} p -  \frac 12 \ahominv q \wedge q  - \star \left( p \wedge q \right) \right| \leq \O_2 \left( C 3^{- \alpha n}\right) .
\end{equation*}
Theorem~\ref{dualityprop} shows that the two linear maps $\ahom$ and $\ahominv$ are related by the following identity.

\begin{theorem}[Duality]  \label{dualityprop}
The homogenized linear maps $\ahom$ and $\ahominv$ satisfy
$
\ahominv =  \left( \ahom \right)^{-1}.
$
\end{theorem}

This duality structure is behind certain exact formulas for the homogenized matrix which are known to hold in dimension~$d=2$ (see~\cite[Chapter 1]{JKO}). We note that similar results were obtained independently by Serre~\cite{serre2017periodic} in the case of periodic coefficients. 

\subsection{Related results}
A qualitative theory of stochastic homogenization was first developed by Kozlov~\cite{K1}, Papanicolaou and Varadhan~\cite{PV1} in the uniformly elliptic setting under an assumption of ergodicity on the environment and the first quantitative results are due to Yurinski\u\i~\cite{Y1} (see also the monograph~\cite{JKO}). More recently, many important progress have been obtained in the development of a quantitative theory of stochastic with the works of Gloria and Otto~\cite{GO1, GO2, GO3} and Gloria, Neukamm, Otto~\cite{GNO, GNO2} who proved a number of optimal estimates under the assumption that the law of the environment satisfies some form of spectral gap inequality. Another approach based upon the study of sub- and superadditive quantities  was later initiated by Armstrong and Smart~\cite{AS} and developed by Armstrong, Kuusi and Mourrat in~\cite{AKM1, AKM2}. We refer to the monograph~\cite{armstrong2017quantitative} for a summary of this approach.

The results presented in this article extend the theory of homogenization to a system of equations which is elliptic but not uniformly elliptic. The question of the extension of the theory to non-uniformly elliptic environments has been an active subject of research over the past decades, partly due to its connection with the random conductance model (see~\cite{Ku10, biskupsurvey} for a survey on this model) and we review below some results in degenerate stochastic homogenization. In~\cite{AN19}, Andres and Neukamm established a Berry-Esseen theorem and decay estimates on the semigroup associated to a degenerate random conductance model. In~\cite{neukamm2017stochastique}, Neukamm, Sch\"affner and Schl\"omerkemper studied homogenization of nonconvex energy functionals with degenerate growth under moments condition. In~\cite{bella2018liouville}, Bella, Fehrman and Otto studied degenerate elliptic systems of equations and established, under some moments condition on the law of the coefficient field, a first-order Liouville theorem and a large scale $C^{1 , \alpha}$-regularity theory. In~\cite{flegel2017homogenization},  Flegel, Heida and Slowik studied homogenization of discrete degenerate elliptic equations on the discrete lattice $\Zd$ under moment conditions on the law of the conductances and allowing jumps of arbitrarily length. In~\cite{GHV18},  Giunti, H\"ofer and Vel\'azquez studied homogenization of the Poisson equation in a randomly perforated domain (see~\cite{GH19} for an extension of the result to the Stokes equation). In~\cite{GM18}, Giunti and Mourrat studied the degenerate random conductance model and found some sufficient and necessary conditions for the relaxation of the environment seen by the particle to be diffusive and obtained as a corollary moment bounds on the corrector. In~\cite{LNO}, Lamacz, Neukamm and Otto studied homogenization on a percolation model modified such that all the bonds in a given direction are declared open. In~\cite{AD2}, Armstrong and the first author established a quantitative homogenization theorem and a large scale regularity theory on the infinite cluster in supercritical Bernoulli bond percolation.

To the best of our knowledge, the only results about homogenization of differential forms were obtained by Serre in~\cite{serre2017periodic} in the periodic setting.

\subsection{Further outlook and conjecture} The rates of convergence obtained for the energy functional~$J$ stated in Theorem~\ref{maintheorem} and in the homogenization theorem, Theorem~\ref{homogenizationtheorem}, are suboptimal. We expect that the strategy developed in~\cite{AKM1, AKM2} should apply to improve the value of the exponent. More specifically, we expect the following optimal rate to hold.

\begin{conjecture}
Let $r$ be an integer in $\{ 1 , \ldots, d \}$. There exist a constant $C(d, \lambda) < \infty$ and a linear mapping $\ahom \in \mathcal{L} \left(\Lambda^{r}(\Rd), \Lambda^{(d-r)}(\Rd) \right)$, which is symmetric and satisfies the ellipticity condition~\eqref{ellipticityassumption}, such that, for every $n \in \N$,
\begin{equation*}
\sup_{(p, q) \in B_1\Lambda^{r}(\Rd) \times B_1\Lambda^{d-r}(\Rd)}  \left| J(\cu_n ,p , q) - \frac 12 p \wedge \ahom p -  \frac 12 \ahom^{-1} q \wedge q + \star \left( p \wedge q \right) \right| \leq \O_1 \left( C 3^{-n}\right).
\end{equation*}
\end{conjecture}

An improvement in the rate of convergence of the subadditive energy $J$ can be transferred into an improvement in the rate of convergence in Theorem~\ref{homogenizationtheorem}, with minor and mostly notational changes of the arguments.

\subsection{Outline of the proof}\label{section1.3} Following the ideas developed by Armstrong, Mourrat and Kuusi in~\cite[Chapter 2]{armstrong2017quantitative} in the case of uniformly elliptic equations, the strategy to obtain an algebraic rate of convergence for the energy $J$ relies on the following argument. We first reduce the problem and show that to obtain the estimate~\eqref{statementmaintheorem}, it is sufficient to study the deterministic quantity $\E \left[ J \left( \cu_n , p , \ahom p \right) \right]$ and to prove that it decays algebraically fast. To this end, we prove an estimate of the form, for each $p \in \Lambda^r(\Rd)$,
\begin{equation} \label{key.estJell}
\E \left[ J \left( \cu_{n+1} , p , \ahom p \right) \right] \leq C \left( \E\left[ J\left( \cu_n , p , \ahom p \right) \right] - \E\left[ J\left( \cu_{n+1} , p , \ahom p \right) \right] \right),
\end{equation}
which implies the inequality
\begin{equation*}
\E \left[ J \left( \cu_{n+1} , p , \ahom p \right) \right] \leq \frac{C}{C+1} \E\left[ J\left( \cu_n , p , \ahom p \right) \right].
\end{equation*}
Iterating the previous display provides a quantitative rate of convergence for the expectation of the subadditive quantity $J$.
The proof of the estimate~\eqref{key.estJell} is thus the key ingredient in the proof of Theorem~\ref{maintheorem} and is carried out in Section~\ref{section5}. The proof relies on a series of lemmas (Lemmas~\ref{l.descendre},~\ref{l.flatness} and~\ref{l.flatness2}) where various quantities are estimated in terms of the subadditivity defect $\tau_n := \left(  \E\left[ J\left( \cu_n , p , \ahom p \right) \right] - \E\left[ J\left( \cu_{n+1} , p , \ahom p \right) \right] \right)$.
Once the quantitative convergence of the expectation is established, one deduces the quantitative convergence of the random variable $J$ with exponential moments as stated in Theorem~\ref{maintheorem} by using the subbaditivity argument combined with a concentration inequality.

In Section~\ref{section6}, we use the results of Section~\ref{section5}, and more specifically the quantitative sublinearity of the corrector and the flux stated in Proposition~\ref{weakcvggradientandflux} to deduce Theorem~\ref{homogenizationtheorem} by using a two-scale expansion (see~\eqref{ad0}).

While the strategy comes from~\cite[Chapter 2]{armstrong2017quantitative}, a number of technical difficulties arise due to the specific structure of the problem; in particular the tools of functionnal analysis useful in the development of the theory of stochastic homogenization need to be adapted to the setting of differential forms. This is achieved by using results  of Mitrea, Mitrea, Monniaux~\cite{MMM08}, Mitrea, Mitrea, Shaw~\cite{MMS08} and the monograph of Schwarz~\cite{Sch95}.

We complete this section by giving the outline of the proof of Theorem~\ref{dualityprop}. It relies on the two following observations: the map $J_{\mathrm{inv}} $ is related to the map $J$ by the formula $J\left( \cu_m , q , p \right) = J_{\mathrm{inv}} \left( \cu_n , p , q \right)$, and the assumptions of Theorem~\ref{maintheorem} are satisfied by the random environment $\a^{-1}$, which implies that there exists a linear map $\ahominv$ such that $J_{\mathrm{inv}}\left(\cu_m , q , \ahominv q  \right) \leq \O_2 \left( C3^{-\alpha m} \right)$.
A combination of these two observations yields the identity $\ahom^{-1} =  \ahominv$.

\subsection{Organization of the paper.} The rest of this article is organized as follows. In Section~\ref{section3}, we state without proof some properties pertaining to differential forms. In Section~\ref{section4}, we generalize some inequalities known for functions to the setting of differential forms. In Sections~\ref{section5} and~\ref{section6} are devoted to the proofs of Theorem~\ref{maintheorem} and~\ref{homogenizationtheorem} respectively. In Section~\ref{section7}, we study the duality structure between $r$-forms and $(d-r)$-forms and prove Theorem~\ref{dualityprop}. Appendix~\ref{appA} is devoted to the proof of some regularity estimates.

\subsection{Convention for constants, exponents} In this article, the symbols $c$ and $C$ denote positive constants which may vary from line to line. These constants may depend only on the dimension $d$ and the ellipticity $\lambda$. Similarly we use the symbols $\alpha, \, \beta$ to denote positive exponents which may vary from line to line and depend only on $d$, $\lambda$. We use the symbol $C$ for large constants (whose value is expected to belong to $[1, \infty)$) and $c$ for small constants (whose value is expected to be in $(0,1]$). The values of the exponents $\alpha, \, \beta$ are always expected to be small. When the constants and exponents depend on other parameters, we write it explicitly and use the notation $C := C(d , U , \lambda)$ to mean that the constant $C$ depends on the parameters $d ,U$ and $\lambda$.

\subsection{Notation and assumptions} \label{section2}
We unfortunately must introduce quite a bit of notation which are collected in this section. The reader is encouraged to skim and consult as a reference.

\subsubsection{General notations and definitions} \label{sec1.6.1}
We consider the space $\Rd$ in dimension $d \geq 2$, equipped with the standard Euclidean norm denoted by $|\cdot|$. We denote by $e_1 , \ldots , e_d$ the canonical basis of $\Rd$. A cube of $\Rd$, generally denoted by the symbol $\cu$, is a set of the form $\cu := z + (-R/2,R/2)^d;$ we denote by $\size(\cu) := R$ the sidelength of the cube. We will frequently use the cube centered at $0$ and of sidelength $3^m$, for which we introduce a specific notation: for $m \in \N$, we write
$
\cu_m :=  \left( - 3^m/2 ,  3^m/2 \right)^d.
$
We refer to the cubes of the form $
z +\cu_m , \, m \in \N, z \in 3^m \Zd$ as triadic cubes. Given two sets $U, V \subseteq \Rd$, we denote by $\dist(U,V) := \inf_{x \in U , y \in V} |x - y|$.

If $U$ is an open subset of $\Rd$, we denote its Lebesgue measure by $|U|$. The normalized integral of a function $u : U \rightarrow \R$ is denoted by
\begin{equation*}
\fint_U u(x) \, \di x := \frac{1}{|U|} \int_U u(x) \, \di x.
\end{equation*}
We denote by $L^p(U)$, for $1 \leq p \leq \infty$, and $H^s(U)$, for $s \in \R$, the standard Lebesgue and Sobolev spaces on the open set $U$.

For $0 \leq r \leq d$, we denote by $\Lambda^{r}(\Rd)$ the space of $r$-linear forms. It is an Euclidean space of dimension~$\binom{d}{r}$, a canonical basis is given by the family $\left(\di x_{i_1} \wedge \ldots \wedge  \di x_{i_r} \right)_{1 \leq i_1 < \ldots < i_r \leq d}.$ To ease the notation, we denote by
\begin{equation} \label{notationconventiondxI}
\di x_I := \di x_{i_1} \wedge \ldots \wedge  \di x_{i_r}, \mbox{ for } I = \left\{  i_1 , \ldots, i_r \right\} \subseteq  \left\{ 1 , \ldots ,d \right\}.
\end{equation}
We denote by $(\cdot , \cdot )$ the scalar product on $\Lambda^r(\Rd)$, i.e., for any pair $\alpha, \beta \in \Lambda^r \left( \Rd \right)$
\begin{equation} \label{def.scproductform}
\left( \sum_{|I|=r} \alpha_I \di x_I,  \sum_{|I|=r} \beta_I \di x_I \right) =  \sum_{|I|=r} \alpha_I \beta_I. 
\end{equation}
We use the notation, for $\alpha \in \Lambda^r(\Rd)$,
$
|\alpha| = \sqrt{(\alpha , \alpha)}.
$
We denote by $B_1 \Lambda^r (\Rd)$ the unit ball of $\Lambda^r(\Rd)$.

Given $U$ an open subset of $\Rd$, a differential form is a measurable map defined in $U$, valued in $\Lambda^{r}(\Rd)$, which can be decomposed $u : = \sum_{|I| = r} u_I \di x_I$.

\smallskip

Throughout the article, we need to assume some regularity on the differential form $u$. To this end, we introduce the following spaces:
\begin{itemize}
\item The space of smooth (resp. smooth and compactly supported) differential forms on $U$, denoted by $C^\infty \Lambda^r \left( U \right)$ ( resp. $C^\infty_c \Lambda^r \left( U \right)$), i.e.,
\begin{equation*}
\begin{aligned}
C^\infty \Lambda^r \left( U \right) := \left\{ u = \sum_{|I| = r} u_I (x) \di x_I~:~ \forall I , ~ u_I \in C^\infty \left( U \right)  \right\}, \\
C^\infty_c \Lambda^r (U) := \left\{ u = \sum_{|I| = r} u_I (x) \di x_I~:~ \forall I , ~u_I \in C_c^\infty(U)  \right\},
\end{aligned}
\end{equation*}
where the notations $C^\infty \left( U \right)$ (resp. $C^\infty_c \left( U \right)$) refers to the set of smooth (resp. smooth and compactly supported) functions defined on $U$ and valued in $\R$. 
We denote by $\mathcal{D}_r (U)$ the space of $r$-currents, i.e., the space of formal sums
$
\sum_{|I|= r} u_I \di x_I,
$
where, for each subset $I \subseteq \{ 1 , \ldots, d \}$ of cardinality $r$, $u_I$ is a distribution on $\Omega$. It is equivalently defined as the topological dual of the space $C^\infty_c \Lambda^r \left( U \right)$;

\smallskip

\item For $1 \leq p \leq \infty$, we let $L^p \Lambda^r (U)$ be the set of $L^p$-differential forms on $U$, i.e.,
\begin{equation*}
L^p \Lambda^r (U) := \left\{ u = \sum_{|I| = r} u_I (x) \di x_I~:~ \forall I , ~u_I \in L^p(U) \right\},
\end{equation*}
equipped with the norm, for $1 \leq p < \infty$,
\begin{equation*}
\| u \|_{L^p \Lambda^r (U)}^p := \sum_{|I| = r} \int_U \left| u_I (x) \right|^p \, dx \hspace{3mm} \mbox{and} \hspace{3mm} \| u \|_{L^\infty \Lambda^r (U)} := \sum_{|I| = r} \mathrm{ess \, sup}_U \, u_I,
\end{equation*}
where the notation ess sup refers to the essential supremum of the mapping $u$.
For $1 \leq p < \infty$, we introduce the normalized $L^p$-norm
\begin{equation*}
\| u \|_{\underline{L}^p \Lambda^r (U)}^p := \sum_{|I| = r} \fint_U \left|u_I (x) \right|^p \, \di x.
\end{equation*}
We also equip the space $L^2 \Lambda^r (U)$ with the scalar product defined by the formula $\left\langle u , v \right\rangle_{U} := \sum_{|I| = r} \int_U u_I(x) v_I(x) \, dx$;
\smallskip
\item For a regularity exponent $k \in \N$, we define the set of $H^k$-differential forms on $U$, denoted by $H^k \Lambda^r (U)$, i.e.,
\begin{equation*}
H^k \Lambda^r (U) := \left\{ u = \sum_{|I| = r} u_I (x) \di x_I~:~ \forall I , ~u_I \in H^k(U)  \right\},
\end{equation*}
equipped with the scalar product~$\left\langle u , v \right\rangle_{H^k  \Lambda^r (U)} := \sum_{|I| = r} \left\langle u_I, v_I \right\rangle_{H^k(U)}$, where the notation $\left\langle \cdot, \cdot \right\rangle_{H^k(U)}$ is used to denote the scalar product of the Hilbert space $H^k(U)$.
\end{itemize}
We denote the $i$th-partial derivative of a form $u$ by $\partial_{i} u$, it is understood in the sense of currents according to the formula
$
\partial_{i} u = \sum_{|I| = r} \partial_{i} u_I \di x_I, 
$
where $\partial_{i} u_I$ is understood in the sense of distributions. The gradient of $u$, denoted by $\nabla u := \left( \partial_{1} u, \ldots, \partial_{d} u \right)$, is a vector-valued differential form. Higher derivatives, which are also vector-valued forms, are denoted by, for $l \geq 1$,
$
\nabla^l u := \left( \partial_{i_1}\ldots  \partial_{i_l} u \right)_{i_1, \ldots, i_l \in \{ 1, \ldots , d\} }.
$

Given a $m$-form $\alpha$ and an $r$-form $\omega$, we consider the exterior product $\alpha \wedge \omega$ which is an $(m + r)$-form and satisfies the following property
$
\alpha \wedge \omega = \left( -1 \right)^{mr}  \omega \wedge \alpha.
$
If $m+r > d$, we set $\omega\wedge \alpha = 0$.

\smallskip

We define the exterior derivative which maps $C^\infty \Lambda^r\left(U\right)$ to $C^\infty \Lambda^{r+1}\left(U\right)$ according to the formula
\begin{equation*}
\di u = \sum_{|I|= r} \sum_{k \notin I} \frac{\partial u_I}{\partial x_k} \di x_k \wedge \di x_I,
\end{equation*}
and extend the definition of this operator to currents. In particular, if $u$ is a differential form of degree $d$, then $\di u = 0$. This operator satisfies the following properties
\begin{equation} \label{propd}
\di \circ \di = 0 \hspace{5mm} \mbox{ and } \hspace{5mm} \di ( u \wedge v ) = (\di u) \wedge v + (-1)^r u \wedge (\di v).
\end{equation}
Given a smooth form $u := \sum_{|I|= r} u_I \di x_I  \in C^\infty \Lambda^{r} \left( U \right)$, an open set $V \subseteq \Rd$ and a smooth map $\Phi = \left( \Phi_1 , \ldots, \Phi_d \right) : V \rightarrow U $, we define the pullback of $u$ by $\Phi$ to be the smooth form defined by the formula
\begin{equation*}
\Phi^* u :=   \left\{ \begin{aligned}
V & \rightarrow  \Lambda^{r}(\Rd), \\
 x  & \mapsto \sum_{I= \left\{ i_1 , \ldots , i_r \right\}} u_{I} \left( \Phi(x) \right)  \di \Phi_{i_1} (x)  \wedge \cdots \wedge \di \Phi_{i_r} (x),
 \end{aligned}  \right.
\end{equation*}
where $\di \Phi_i(x)$ denotes the differential of the map $\Phi_i$ evaluated at the point $x$. The pullback satisfies the following properties, given a $r$-form $u$ and a $m$-form $v$,
\begin{equation} \label{pullbackdist}
\Phi^*  \di u = \di \Phi^* u ~ \mbox{and} ~ \Phi^*  \left( u \wedge v \right) = \Phi^*  u \wedge\Phi^*  v.
\end{equation}
Given another open set $W \subseteq \Rd$ and another smooth map $\Psi : W \rightarrow V$, we have the composition rule
$
\Psi^* \left(\Phi^* u \right) = \left( \Phi \circ \Psi \right)^* v.
$
Moreover, if we assume that $\Phi$ is a smooth diffeomorphism from $V$ to $U$ such that all the derivatives of $\Phi$ are bounded then, for each integer $k \in \N$, the mapping $\Phi^*$ sends the space $H^k \Lambda^r (U)$ to the space $H^k \Lambda^r (V)$ and we have the estimate
\begin{equation} \label{phistarhscont}
\left\| \Phi^* u \right\|_{H^k \Lambda^r (V)} \leq C \left\| u \right\|_{H^k \Lambda^r (U)},
\end{equation}
for some constant $C$ depending on the dimension $d$, the integer $k$ and the function $\Phi.$

There is a canonical bijection between the spaces $\Lambda^{r}(\Rd)$ and $\Lambda^{(d-r)}(\Rd)$ given by the Hodge star operator. It is denoted by the symbol~$\star$, sends the space $ \Lambda^{r}(\Rd) $ to the space $ \Lambda^{(d-r)}(\Rd)$ and satisfies the property, for each $\alpha, \beta \in \Lambda^{r}(\Rd)$,
\begin{equation*}
\alpha \wedge \star \beta = (\alpha , \beta) \di x_1 \wedge \cdots \wedge \di x_d.
\end{equation*}
It is defined on the canonical basis by the formula
\begin{equation} \label{def:hodgestar}
\star \left(\di x_{i_1} \wedge \cdots \wedge \di x_{i_r} \right) : = \di x_{i_{r+1}} \wedge \cdots \wedge \di x_{i_d},
\end{equation}
where $(i_1,\ldots , i_d)$ is an even permutation of $\left\{ 1 , \ldots , d \right\}$. An important property of this operator is the following, for each form $\alpha \in \Lambda^{r}(\Rd)$,
\begin{equation} \label{invstar}
\star \star \alpha = (-1)^{r(d-r)} \alpha.
\end{equation}
Let $u = u_{\{1,\ldots, d\}} \di x_1 \wedge \cdots \wedge \di x_d$ be a $d-$form defined on~$U$. If the function $u_{\{1,\ldots, d\}}$ belongs to the space $L^1(U)$, we say that $u$ is integrable and define
\begin{equation} \label{def.int.int}
\int_U u := \int_U u_{\{ 1,\ldots, d\}} (x) \, \di x.
\end{equation}
In particular, the scalar product on $L^2 \Lambda^r (U)$ can be rewritten, for each pair of forms $u, v \in L^2 \Lambda^r(U)$,
\begin{equation*}
\left\langle u, v \right\rangle_U = \int_U u \wedge \star v.
\end{equation*}
 Additionally, if $\Phi$ is a smooth positively oriented diffeomorphism mapping $V$ to $U$, then the change of variables formula reads, for each integrable $d$-form $u$,
\begin{equation} \label{changevar}
\int_{V} \Phi^* u = \int_{U}  u.
\end{equation}

We then define the normal and tangential components of a form. We consider a smooth bounded domain $U \subset \Rd$, denote by~$\nu$ the outward normal of $\partial U$ and fix a smooth $r$-form $u \in C^\infty \Lambda^{r}(\Rd)$. For each point $x \in \partial U$, we define $\mathbf{n}u (x) \in \Lambda^r \left( \Rd\right) $, the normal component of $u(x)$, to be the orthogonal projection of $u(x)$ with respect to the scalar product $(\cdot,\cdot)$ defined in~\eqref{def.scproductform} on the kernel of the mapping

\begin{equation} \label{defnormalcomponent}
\di \nu (x) \wedge \, \cdot ~ : \left\{ \begin{aligned}
\Lambda^{r}(\Rd) & \rightarrow  \Lambda^{r+1}(\Rd), \\
 v  & \rightarrow \di \nu (x) \wedge v.
\end{aligned}  \right.
\end{equation}
The tangential component of $u(x)$, denoted by $\mathbf{t} u(x)$, is given by the formula
\begin{equation} \label{tangentialcomponent}
\mathbf{t} u(x) = u(x) - \mathbf{n}u(x).
\end{equation}
For a smooth form $u \in C^\infty \Lambda^{d-1} (U)$, using the previous notation, we know that there exists a smooth function $v : \partial U \rightarrow \R$ such that, for each $x \in \partial U$,
$
\mathbf{t} u(x) = v(x) \di e_1^x \wedge \cdots \wedge \di e_{d-1}^x ,
$
where the vectors $e_1^x, \ldots , e_{d-1}^x  \in \Rd$ are such that $(e_1^x,\ldots, e_{d-1}^x, \nu(x))$ is an orthonormal basis positively oriented of $\Rd$.
With this notation, we define the integral of $u$ on $\partial U$ by the formula
\begin{equation} \label{defboundaryint}
\int_{\partial U} u = \int_{\partial U} v(x) \di \mathcal{H}^{d-1} (x),
\end{equation}
where $\mathcal{H}^{d-1}$ is the Hausdorff measure of dimension $(d-1)$ on $\Rd$.

The two definitions of integrals~\eqref{def.int.int} and~\eqref{defboundaryint} are related by the Stokes'~formula: for each smooth bounded domain $U \subset \Rd$ and each form $u \in C^\infty\Lambda^{d-1}(U)$, 
\begin{equation} \label{Stokes}
\int_{\partial U} u = \int_{ U} \di u.
\end{equation}

We define the codifferential $\delta$, to be the formal adjoint of the exterior derivative $\di$ with respect to the scalar product $\left\langle \cdot, \cdot \right\rangle_{U}$, i.e., the operator which satisfies for each pair $(u , v) \in C^\infty_c \Lambda^{r-1}(U) \times C^\infty_c \Lambda^r(U)$,
\begin{equation*}
\left\langle \di u, v \right\rangle_{U} = \left\langle u, \delta v \right\rangle_{U}.
\end{equation*}
This operator can be explicitly computed using the second equality in~\eqref{propd}, the equality~\eqref{invstar}, and the Stokes' formula~\eqref{Stokes}. We obtain
\begin{equation} \label{explicitformuladelta}
\delta = (-1)^{(r-1)d + 1} \star \di \star.
\end{equation}
For each open subset $U \subseteq \Rd$, and each integer $r \in \{ 0 , \ldots , d-1 \}$, we define the space $H^1_{\di} \Lambda^r(U)$ to be the set of forms $u$ in the space $L^2\Lambda^r(U)$ such that the exterior derivative $\di u$ belongs to the space $L^2\Lambda^{r+1}(U)$, i.e.,
\begin{equation*}
H^1_\di \Lambda^r (U) := \left\{ u \in L^2\Lambda^r(U) ~:~ \exists f \in L^2\Lambda^{r+1}(U), \forall v \in C^\infty_c \Lambda^{d-r-1} (U) ,\int_U \left( u \wedge \di v + (-1)^r f \wedge v \right) = 0  \right\}.
\end{equation*}
Given a form $u \in H^1_\di \Lambda^r(U)$, we denote by $\di u$ the unique form in the space $L^2\Lambda^{r+1}(U)$ which satisfies, for every $v \in C^\infty_c \Lambda^{d-r-1} (U)$,
\begin{equation} \label{IPPdi}
\int_U \left( u \wedge \di v + (-1)^r \di u \wedge v \right) = 0.
\end{equation}
This space is a Hilbert space equipped with the norm
\begin{equation*}
\| u \|_{ H^1_\di \Lambda^r(U)}^2 = \left\langle u , u \right\rangle_U + \left\langle \di u , \di u \right\rangle_U.
\end{equation*}
In the case $r=d$, we have $\di u = 0$ for each form $u \in L^2 \Lambda^d (U)$. This implies the equality $ H^1_\di \Lambda^d(U) = L^2 \Lambda^d (U)$. We denote by $H^1_{\di , 0} \Lambda^r(U)$ the closure of $C_c^\infty \Lambda^r (U)$ in $H^1_{\di } \Lambda^r(U)$, i.e.,
\begin{equation*}
H^1_{\di , 0} \Lambda^r(U) := \bar{C_c^\infty \Lambda^r (U)}^{H^1_{\di } \Lambda^r(U)}.
\end{equation*}
Symmetrically, for each integer $r \in \{ 0 , \ldots , d-1 \}$, we define $H^1_{\delta} \Lambda^r(U)$ to be the set of forms $u$ in~$L^2\Lambda^r(U)$ such that the codifferential of $u$ belongs to the space $L^2\Lambda^{r-1}(U)$, i.e.,
\begin{equation*}
H^1_\delta \Lambda^r (U) := \left\{ u \in L^2\Lambda^r(U) ~:~ \exists f \in L^2\Lambda^{r-1}(U), \forall v \in  C^\infty_c \Lambda^{d-r+1} (U) ,\int_U \left( u \wedge \delta v  + (-1)^{d-r} f \wedge v \right) = 0 \right\}.
\end{equation*}
and in that case, we denote by $\delta u$ the unique form which satisfies, for each $v \in  C^\infty_c \Lambda^{d-r+1} (U)$, $\int_U u \wedge \delta v  + (-1)^{d-r} \delta u \wedge v  = 0$. In the case $r=0$, we have $\delta u = 0$ for each function $u \in L^2  (U)$. This implies the equality $ H^1_\delta \Lambda^0(U) = L^2 (U)$. We also denote by $H^1_{\delta , 0} \Lambda^r(U)$ the closure of $C_c^\infty \Lambda^r (U)$ in $H^1_{\delta } \Lambda^r(U)$, i.e.,
\begin{equation*}
H^1_{\delta , 0} \Lambda^r(U) := \bar{C_c^\infty \Lambda^r (U)}^{H^1_{\delta } \Lambda^r(U)}.
\end{equation*}
We say that a form~$u \in H^1_\di \Lambda^r(U)$ is closed (resp. co-closed) if it satisfies $\di u = 0$ (resp. $\delta u = 0$). We denote by $C^r_\di(U)$ and $C^r_{\di,0} (U)$ the subsets of closed $r$ forms and closed $r$ forms with vanishing tangential trace, i.e.,
\begin{equation*}
C^r_\di (U) := \left\{ u \in  H^1_\di\Lambda^r(U) ~ : ~ \di u = 0  \right\} \hspace{3mm} \mbox{and} \hspace{3mm} C^r_{\di,0} (U) := C^r_\di (U) \cap H^1_{\di , 0} \Lambda^r(U).
\end{equation*}
Symmetrically, we denote by $C^r_\delta(U)$ and $C^r_{\delta,0} (U)$ the subsets of co-closed $r$ forms and co-closed $r$ forms with vanishing normal trace, i.e.,
\begin{equation*}
C^r_\delta (U) := \left\{ u \in  H^1_\di\Lambda^r(U) ~ : ~ \delta u = 0  \right\} \hspace{3mm} \mbox{and} \hspace{3mm} C^r_{\delta,0} (U) := C^r_\delta (U) \cap H^1_{\delta , 0}\Lambda^r(U).
\end{equation*}

\subsubsection{Notation related to the probability space}

In this section, we record some properties of the notation $\O_s$ introduced in Definition~\ref{def.stochint1.1} to measure stochastic integrability. The notation is homogeneous:
$
X \leq \O_s (C) \Longleftrightarrow X/C \leq \O_s(1).
$

More generally, for $\theta_0 , \theta_1 , \ldots , \theta_n \in \R^+$ and $C_1 , \ldots ,C_n \in \R^+_*$, we write
$
X \leq \theta_0 + \theta_1 \O_s\left(C_1 \right) + \cdots + \theta_n \O_s\left(C_n \right)
$
to mean that there exist nonnegative random variables $X_1, \ldots, X_n$ satisfying the estimate $X_i \leq \O_s\left(C_n \right)$ such that
$
X \leq  \theta_0 + \theta_1X_1 + \cdots + \theta_n X_n.
$
We record a property about this notation, the proof of which can be found in~\cite[Lemma A.4]{armstrong2017quantitative}.

\begin{proposition}\label{propertiesbigO} For each $s \in (0, \infty)$, there exists a constant $C_s < \infty$ such that the following holds. Let $\mu$ be a measure over an arbitrary measurable space $E$, let $\theta: E \rightarrow (0, \infty)$ be a measurable function and $(X(x))_{x \in E}$ be a jointly measurable family of nonnegative random variables such that, for every $x \in E, X(x) \leq \O_s \left( \theta(x) \right)$. Then we have
\begin{equation} \label{e.0mean}
\int_E X(x)  \mu(dx) \leq \O_s \left( C_s \int_E \theta(x) \mu( dx) \right).
\end{equation}
Consequently, if $X_1, \ldots , X_n$ are non-negative random variables satisfying, for each $i \in \left\{ 1, \ldots, n \right\}$, $X_i \leq  \O_s \left( C_i\right)$, then
\begin{equation*}
\sum_{i = 1}^{n} X_i \leq \O_s \left( C_s  \sum_{i= 1}^{n} C_i \right).
\end{equation*}
\end{proposition}

\subsubsection{Negative Sobolev spaces.} \label{sectionH-1norms} We introduce three $H^{-1}$-norms for differential forms. We fix an integer $r \in \{ 0 , \ldots , d\}$ and a bounded open set $U \subseteq \Rd$. We define the following spaces:
\begin{itemize}
\item The space $H^{-1}\Lambda^r \left( U \right)$ is the topological dual of the space $H^{1}_0\Lambda^r \left( U \right)$, it is equipped with the norm, for each $\alpha \in H^{-1}\Lambda^r \left( U \right)$,
\begin{equation*}
\left\| \alpha \right\|_{H^{-1}\Lambda^r \left( U \right)} := \sup \left\{ \left\langle \alpha, \omega \right\rangle \, : \, \left\| \omega \right\|_{H^1_0\Lambda^r(U)} \leq 1 \right\},
\end{equation*} 
where $\left\langle \cdot , \cdot \right\rangle$ denotes the duality product between the spaces $H^{-1}\Lambda^r \left( U \right)$ and $H^{1}_0\Lambda^r \left( U \right)$.
\item The space $H^{-1}_{\di}\Lambda^r \left( U \right)$ is the topological dual of the space $H^{1}_{\di , 0}\Lambda^r \left( U \right)$, it is equipped with the norm, for each $\alpha \in H^{-1}_{\di}\Lambda^r \left( U \right)$,
\begin{equation*}
\left\| \alpha \right\|_{H^{-1}_{\di}\Lambda^r \left( U \right)} := \sup \left\{ \left\langle \alpha, \omega \right\rangle \, : \, \left\| \omega \right\|_{H^1_{\di , 0}\Lambda^r(U)} \leq 1 \right\},
\end{equation*}
where $\left\langle \cdot , \cdot \right\rangle$ denotes the duality product between the spaces $H^{-1}_{\di}\Lambda^r \left( U \right)$ and $H^{1}_{\di , 0}\Lambda^r \left( U \right)$.
\item The space $\underline{H}^{-1}\Lambda^r \left( U \right)$ is the topological dual of the space $H^{1}\Lambda^r \left( U \right)$, it is equipped with the scales norm, for each $\alpha \in H^{-1}\Lambda^r \left( U \right)$,
\begin{equation*}
\left\| \alpha \right\|_{\underline{H}^{-1}\Lambda^r \left( U \right)} := \sup \left\{ \left\langle \alpha, \omega \right\rangle \, : \, \left\| \nabla \omega \right\|_{\underline{L}^2\Lambda^r(U)} + |U|^{-\frac 1d}\left| \left( \omega \right)_{U} \right| \leq 1 \right\},
\end{equation*}
where $\left\langle \cdot , \cdot \right\rangle$ denotes the duality product between the spaces $\underline{H}^{-1}\Lambda^r \left( U \right)$ and $H^{1}\Lambda^r \left( U \right)$.
\end{itemize}
By definition of these spaces, we have the following continuous inclusions $H^{-1}_{\di}\Lambda^r \left( U \right) \subseteq H^{-1}\Lambda^r \left( U \right)$  and $\underline{H}^{-1}\Lambda^r \left( U \right) \subseteq H^{-1}\Lambda^r \left( U \right).$

\subsubsection{Notation and assumptions related to homogenization} \label{section:NARH}
Given $\lambda \in (0 ,1]$ and $1 \leq r \leq d$, we consider the space of measurable functions from $\Rd$ to $\mathcal{L} \left( \Lambda^{r}(\Rd) , \Lambda^{(d-r)}(\Rd) \right)$ satisfying the symmetry assumption, for each $x \in \Rd$ and each $ p_1,p_2 \in \Lambda^{r}(\Rd)$,
\begin{equation} \label{symmetryassumption}
p_1 \wedge \a(x) p_2 = p_1 \wedge \a(x) p_2,
\end{equation}
and the ellipticity assumption, for each $x \in \Rd$ and each $p \in \Lambda^{r} (\Rd)$,
\begin{equation} \label{ellipticityassumption}
\lambda |p|^2  \leq \star ( p \wedge \a(x) p) \leq \frac 1\lambda |p|^2.
\end{equation}
We denote by $\Omega_r$ the collection of all such measurable maps,
\begin{multline} \label{defOmegar}
\Omega_r := \Big\{ \a (\cdot) \, : \, \a : \Rd \rightarrow \mathcal{L} \left( \Lambda^{r}(\Rd) , \Lambda^{(d-r)} \left( \Rd \right) \right)  \mbox{ is Lebesgue measurable}  \\   \mbox{and satisfies \eqref{symmetryassumption} and \eqref{ellipticityassumption}}\Big\}.
\end{multline}
A generic element of the set $\Omega_r$ is denoted by $\a$ and referred to as an \emph{environment}. We endow $\Omega_r$ with the translation group $(\tau_y)_{y \in \Rd}$, acting on $\Omega_r$ via
$
\left( \tau_y \a \right) (x) := \a (x+y) ,
$
and with the family $\left\{ \mathcal{F}_r(U) \right\}$ of $\sigma$-algebras on $\Omega_r$, with $\mathcal{F}_r(U)$ defined for each Borel subset $U \subseteq \Rd$ by the formula
\begin{multline*}
\mathcal{F}_r(U) := \bigg\{ \sigma \mbox{-algebra on } \Omega_r \mbox{ generated by the family of maps } \\
						 \a \rightarrow \int_U p \wedge \a(x) q \phi(x), \, p,q \in \Lambda^{r}(\Rd), \, \phi \in C^\infty_c (U) \bigg\}.
\end{multline*}
The largest of these $\sigma$-algebras is $ \F_r(\Rd)$, simply denoted by $\F_r$. The translation
group may be naturally extended to the $\sigma$-algebra $\F_r$ by defining, for $A \in \F_r$,
$
\tau_y A := \left\{ \tau_y \a \, : \, \a \in A \right\}.
$
We then endow the measurable space $(\Omega_r , \F_r)$ with a probability measure $\P_r$ satisfying the two following assumptions:
\begin{itemize}
\item The measure $\P_r$ is invariant under $\Zd$-translations: for every $z \in \Zd$, $A \in \F_r$,
\begin{equation} \label{stationarityassumption}
\P \left[ \tau_z A \right] = \P \left[ A \right];
\end{equation}
\item The measure $\P_r$ has a unit range dependence: for every pair of Borel subsets $U,V \subseteq \Rd$ with $\dist(U,V) \geq 1$,
\begin{equation} \label{independenceassumption}
\F_r(U) \mbox{ and } \F_r(V) \mbox{ are independent}.
\end{equation}
\end{itemize}
The expectation of an $\F_r$-measurable random variable $X$ with respect to $\P_r$ is denoted by $\E_r[X]$ or simply $\E[X]$ when there is no confusion about the value of $r$.

Given an integer $1 \leq r \leq d$, an environment $\a \in \Omega_r$ and an open subset $U \subseteq \Rd$, we say that $u \in H^1_\di \Lambda^{r-1} (U)$ is a solution of the equation
$
\di ( \a \di u ) = 0,
$
if for every smooth compactly supported form $v \in C^\infty_c \Lambda^r (U)$,
$
\int_U \di u \wedge \a \di v = 0.
$
We denote by $\A_r^\a(U)$ the set of solutions, i.e.,
\begin{equation} \label{defsolution}
\A_r^\a(U) := \left\{ u \in H^1_\di \Lambda^r (U) \, :\, \forall v \in C^\infty_c \Lambda^r (U),\int_U \di u \wedge \a \di v = 0 \right\}.
\end{equation}
When there is no confusion, we omit the subscripts $r$ and $\a$ and only write $\A(U)$.

\bigskip

\noindent \textit{Acknowledgement.} I would like to thank Jean-Christophe Mourrat and Scott Armstrong for helpful discussions and comments.

\section{Some results pertaining to differential forms} \label{section3}

In this section, we record some properties pertaining to the spaces $H^1_\di \Lambda^r (U)$, $H^1_\delta \Lambda^r (U)$ and $C^r_\di(U)$. Most of these results and their proofs can be found in~\cite{MMM08} and~\cite{MMS08}.

\subsection{Tangential and normal trace of a differential form.} \label{section2.1} Given a bounded Lipschitz domain $U$ of $\Rd$, we define the Sobolev space $H^{1/2} (\partial U)$ as the set of functions of $L^2(\partial U)$ which satisfy
\begin{equation*}
\left[ g \right]_{H^{1/2} (\partial U)} := \left( \int_{\partial U}  \int_{\partial U} \frac{|g(x) - g(y)|^2}{|x-y|^{d+1}}  \di \mathcal{H}^{d-1}(x)  \di \mathcal{H}^{d-1}(y) \right)^{\frac 1 2} < \infty,
\end{equation*}
where $\mathcal{H}^{d-1}$ denotes the Hausdorff measure of dimension $d-1$ of $\Rd$. It is a Hilbert space equipped with the norm
$
\| g \|_{H^{1/2} (\partial U)} := \| g \|_{L^2 (\partial U)}+ \left[ g \right]_{H^{1/2} (\partial U)}.
$
We define $H^{-1/2} (\partial U) $ to be the dual of $ H^{1/2} (\partial U)$.
We can then extend this definition to differential forms by defining, for each integer $r$ in $ \{ 0, \ldots, d \}$,
\begin{equation*}
H^{1/2} \Lambda^r (\partial U) := \left\{ u \in L^2 \Lambda^r (\partial U) \mbox{ s.t. } u = \sum_{|I|=r} u_I \di x_I \mbox{ and } \forall I, \, \left[ u_I \right]_{H^{1/2} (\partial U)} < \infty   \right\}.
\end{equation*}
This is also a Hilbert space equipped with the norm,
$
\| u \|_{H^{1/2} \Lambda^r (\partial U)} := \| u \|_{L^2 \Lambda^r (\partial U)}+ \sum_{|I|=r}  \left[ u_I \right]_{H^{1/2} (\partial U)}.
$
We define the space $H^{-1/2}  \Lambda^r (\partial U) $ to be the dual of the space $H^{1/2}  \Lambda^{d- r} (\partial U)$.

By the Sobolev Trace Theorem for Lipschitz domains (see~\cite[Chapter VII, Theorem 1]{jonsson1984function} or~\cite{marschall1987trace}), if $U$ is a bounded Lipschitz domain of $\Rd$, then the linear operator $C^\infty \left(U \right) \rightarrow \mathrm{Lip} ( \partial U)$ that restricts a smooth function defined on $\overline{U}$ to the boundary $\partial U$ has an extension to a bounded linear mapping from $H^1(U)$ into $H^{1/2}  (\partial U)$. We denote this operator by Tr and extend to differential forms by setting, for $u = \sum_{|I| =r} u_I \di x_I \in H^1\Lambda^{r}(U)$,
$
\mathrm{Tr} \, u = \sum_{|I| =r} \mathrm{Tr} \, u_I \, \di x_I \in H^{1/2}  \Lambda^r (\partial U).
$

In the case when $u$ does not belong to the space $H^1 \Lambda^r(U)$ but only belongs to the larger space $H^1_\di \Lambda^r (U)$, one still has a Sobolev trace theorem, but one can only get information on the tangential component of the trace of $u$ as is explained in the following proposition, which is a specific case of~\cite[Proposition 4.1 and Proposition 4.3]{MMS08}.
\begin{proposition}[\cite{MMS08}, Proposition 4.1 and Proposition 4.3] \label{tangentialtrace}
For each $u \in H^1_\di \Lambda^{r-1} (U)$ (resp. $u \in H^1_\delta \Lambda^{r+1} (U)$), the map
\begin{equation*}
\left\langle \mathbf{t} u , \cdot \right\rangle : \left\{ \begin{aligned}
 H^{1/2}  \Lambda^{(d-r)} (\partial U)  \rightarrow \,  \R, \hspace{19mm} \\
\psi   \rightarrow \int_U \left( \di u \wedge \Psi +  (-1)^r u \wedge \di \Psi \right),
\end{aligned}  \right.
\quad \mbox{resp.} \quad
\left\langle \mathbf{n} v , \cdot \right\rangle : \left\{ \begin{aligned}
H^{1/2}  \Lambda^{(d-r)} (\partial U)  \rightarrow \,  \R,  \hspace{21mm} \\
\psi   \rightarrow \int_U \left( \delta v \wedge \Psi +  (-1)^{d-r} v \wedge \delta \Psi \right),
\end{aligned}  \right.
\end{equation*}
where $\Psi \in H^1 \Lambda^{d-r} (U)$ is chosen such that $\mathrm{Tr} \,  \Psi = \psi$, is well-defined, linear and bounded. Then the tangential trace operator defined by the formula
\begin{equation*}
 \mathbf{t} : \left\{ \begin{aligned}
H^1_\di \Lambda^r (U)  \rightarrow &\, H^{-1/2}  \Lambda^{r} ( \partial U ) \\
u   \rightarrow \left\langle \mathbf{t}u, \cdot \right\rangle
\end{aligned}  \right.
\quad \mbox{resp.} \quad
 \mathbf{n} : \left\{ \begin{aligned}
H^1_\delta \Lambda^r (U)  \rightarrow &\, H^{-1/2}  \Lambda^{r} ( \partial U ) \\
u   \rightarrow \left\langle \mathbf{n}u, \cdot \right\rangle
\end{aligned}  \right.
\end{equation*}
is linear and continuous. Additionally this notation is consistent with the tangential (resp. normal) component introduced in~\eqref{tangentialcomponent}.
\end{proposition}

The following property shows that, for a Lipschitz domain $U$ in $\Rd$, the space $H^1_{\di,0} \Lambda^r (U)$ (resp. $H^1_{\delta,0} \Lambda^r (U)$) is also the space of differential forms in $H^1_\di \Lambda^r (U)$ (resp. $H^1_\delta \Lambda^r (U)$) with tangential (resp. normal) trace equal to $0$. A proof for these results can be found in~\cite[Lemma 2.13]{MMM08}.

\begin{proposition}[\cite{MMM08}, Lemma 2.13] \label{propdensity}
Let $U$ be an open bounded Lipschitz subset of $\Rd$. For each integer $r \in \{0 , \ldots , d \}$, the following results hold:
\begin{itemize}
\item The space of smooth differential forms $C^\infty \Lambda^r \left( U \right)$ is dense in $H^1_\di \Lambda^r (U)$ (resp. $H^1_\delta \Lambda^r (U)$);
\item The space $C^\infty_c \Lambda^r \left( U \right)$ of smooth and compactly supported differential forms is
dense in $\left\{ u \in H^1_{\di} \Lambda^r (U) ~:~ \mathbf{t}u = 0 \right\}$ and in $\left\{ u \in H^1_{\delta} \Lambda^r (U) ~:~ \mathbf{n}u = 0 \right\}$. In particular, one has
\begin{equation*}
H^1_{\di,0} \Lambda^r (U) = \left\{ u \in H^1_{\di} \Lambda^r (U) \,:\, \mathbf{t}u = 0 \right\} \quad\mbox{and}\quad H^1_{\delta,0} \Lambda^r (U) = \left\{ u \in H^1_{\delta} \Lambda^r (U) \,:\, \mathbf{n}u = 0 \right\}.
\end{equation*}
\end{itemize}
\end{proposition}
 A corollary of this proposition is that the space of solutions $\A(U)$, defined in~\eqref{defsolution}, can be equivalently defined by the formula
\begin{equation} \label{defsolutionbis}
\A(U) := \left\{ u \in H^1_\di \Lambda^r (U) \, :\, \forall v \in H^1_{\di,0} \Lambda^{d-r} (U), \int_U \di u \wedge \a \di v = 0 \right\}.
\end{equation}

\subsection{Solvability of the equation $\di u = f$.}
We record a result concerning the solvability of the equation $\di u = f$ on bounded Lipschitz star-shaped domains.

\begin{proposition}[\cite{MMM08}, Theorem 1.5 and Theorem 4.1] \label{Poincarelemma}
Let $U \subseteq \Rd$ be a bounded Lipschitz star-shaped domain. Then the following statements hold:
\begin{itemize}
\item For  $1 \leq r \leq d$ (resp. $0 \leq r \leq d - 1$), given a differential form $f \in L^2 \Lambda^r ( U)$, the problem
\begin{equation} \label{e.idhanbis}
 \left\lbrace
       \begin{array}{l}
         \di u = f \mbox{ in } U, \\
         u \in H^1_\di \Lambda^{r-1} (U), \\
       \end{array}
 \right.
\qquad \mbox{resp.} \qquad
\left\lbrace
    \begin{array}{l}
   \delta u = f \mbox{ in } U, \\
    u \in H^1_\delta \Lambda^{r+1} (U), \\
    \end{array}
    \right.
\end{equation}
has a solution if and only if the form $f$ satisfies $\di f = 0$ (resp. $\delta f = 0 $). In this case, there exist a constant $C(U) < \infty$ and a solution $u$ of the equation~\eqref{e.idhanbis} which belongs to the space $H^1 \Lambda^{r-1} (U)$ (resp. to the space $H^1 \Lambda^{r+1} (U)$) and satisfies
\begin{equation*}  
\| u \|_{ H^1 \Lambda^{r-1} (U)} \leq C \| f \|_{ L^2 \Lambda^r ( U)}  \hspace{3mm} \mbox{resp. } \| u \|_{ H^1 \Lambda^{r+1} (U)} \leq C \| f \|_{ L^2 \Lambda^r ( U)}.
\end{equation*}
\item For $1 \leq r \leq d-1$, given a differential form $f \in L^2 \Lambda^r (U)$, the problem
\begin{equation}\label{e.idhan}
 \left\lbrace
       \begin{array}{l}
        	\di u = f \mbox{ in } U, \\
		u \in H^1_{\di, 0} \Lambda^{r-1} (U),
       \end{array}
 \right.
\qquad \mbox{resp.} \qquad
\left\lbrace
    \begin{array}{l}
   \delta u = f \mbox{ in } U, \\
    u \in H^1_{\delta , 0} \Lambda^{r+1} (U), \\
    \end{array}
    \right.
\end{equation}
has a solution if and only if $f$ satisfies $\di f = 0$ in $U$ and $\mathbf{t} f = 0$ on $\partial U$ (resp. $\delta f = 0$ in $U$ and $\mathbf{n} f = 0$ on $\partial U$).
In this case, there exist a constant $C( U) < \infty$ and a solution $u$ of the equation~\eqref{e.idhan} which belongs to the space $H^1 \Lambda^{r-1} (U)$ (resp. to the space $ H^1 \Lambda^{r+1} (U)$) and satisfies
\begin{equation} \label{estimate3.2}
\| u \|_{ H^1\Lambda^{r-1} (U)} \leq C \| f \|_{ L^2 \Lambda^r ( U)}  \hspace{3mm} \mbox{resp. } \| u \|_{ H^1 \Lambda^{r+1} (U)} \leq C \| f \|_{ L^2 \Lambda^r ( U)}.
\end{equation}
\item For $r = d$ (resp. $r = 0$), given a differential form $f \in L^2 \Lambda^r (U)$, the problem
\begin{equation*}
 \left\lbrace
       \begin{array}{l}
        	\di u = f \mbox{ in } U,\\
		u \in H^1_{\di, 0} \Lambda^{d-1} (U),
       \end{array}
 \right.
\qquad \mbox{resp.} \qquad
\left\lbrace
    \begin{array}{l}
   \delta u = f \mbox{ in } U, \\
    u \in H^1_{\delta , 0} \Lambda^{1} (U), \\
    \end{array}
    \right.
\end{equation*}
has a solution if and only if $f$ satisfies
$\int_U f = 0$ (resp. $\int_U \star f = 0).$
Moreover there exists a solution $u \in  H^1 \Lambda^{d-1} (U)$ (resp. $u \in  H^1 \Lambda^{1} (U)$) which satisfies the estimate~\eqref{estimate3.2}.
\end{itemize}
\end{proposition}

\subsection{The Hodge-Morrey Decomposition Theorem.} In this section, we record the Hodge-Morrey Decomposition Theorem. This requires to introduce the subspaces of exact, co-exact and harmonic forms.
\begin{definition}
For each open set $U \subseteq \Rd$ and each integer $r \in \{ 1 , \ldots , d \}$, we denote by $\mathcal{E}^r(U)$ the subset of exact $r$-forms with vanishing tangential trace, i.e.,
\begin{equation*}
\mathcal{E}^r(U) := \left\{ u \in  H^1_\di\Lambda^r(U) ~ : ~ \exists \alpha \in H^1_{\di,0} \Lambda^{r-1}(U) ~ \mbox{such that} ~ \di \alpha = u \right\} \subseteq C^r_\di (U),
\end{equation*}
by $\mathcal{C}^r (U)$ the subset of co-exact $r$-forms with vanishing normal trace $\mathcal{C}^r(U)$, i.e.,
\begin{equation*}
\mathcal{C}^r (U) := \left\{ v \in  H^1_\di\Lambda^r(U) ~ : ~ \exists  \beta \in H^1_{\delta} \Lambda^{r+1}(U) ~ \mbox{such that} ~ \delta \beta = v  \right\} \subseteq C^r_\delta (U),
\end{equation*}
and by $\mathcal{H}^r (U)$ the subset of $r$ harmonic forms, i.e.,
\begin{equation*}
\mathcal{H}^r (U) := \left\{ w \in  L^2\Lambda^r(U) ~ : ~ \di w = 0 ~ \mbox{and} ~ \delta w = 0\right\}.
\end{equation*}
\end{definition}

We now state the Hodge decomposition Theorem. This theorem is stated for two types of bounded domains, the convex domains in which case the situation is simple and the result can be deduced from Proposition~\ref{Poincarelemma}, and the smooth domains. In the latter case the proof is more involved and we refer to~\cite[Theorem 2.4.2]{Sch95} for the demonstration.

\begin{proposition}[Hodge-Morrey Decomposition, Theorem 2.4.2 of \cite{Sch95}] \label{hodgedecomp}
Let $U$ be a domain of $\Rd$. We assume that this domain is either convex or smooth. Then for each integer $r \in \{1, \ldots, d\}$, 
\begin{enumerate}
\item[\textit{(i)}] the spaces $\mathcal{E}^r(U)$, $\mathcal{C}^r (U)$ and $\mathcal{H}^r (U)$ are closed in the $L^2 \Lambda^r(U)$-topology;
 \item[\textit{(ii)}] one has the orthogonal decomposition
\begin{equation*}
L^2 \Lambda^r (U) = \mathcal{E}^r(U)  \overset{\perp}{\oplus} \mathcal{C}^r (U) \overset{\perp}{\oplus} \mathcal{H}^r (U).
\end{equation*}
\end{enumerate}
\end{proposition}

\section{Functional inequalities and differential forms} \label{section4}

The goal of this section is to prove some functional inequalities which are necessary in the proof of Theorem~\ref{maintheorem}. We first deduce from the results of the previous section the Poincar\'e inequality for differential forms on convex or smooth bounded domains of $\Rd$, Proposition~\ref{PoincareineqH10conv}. We then state, without proof, the Gaffney-Friedrichs inequality for convex or smooth bounded domains of $\Rd$. We deduce from these propositions the multiscale Poincar\'e inequality, Proposition~\ref{multscalepoincare}. We finally conclude this section by stating the Caccioppoli inequality for differential forms. 

\subsection{The Poincar\'e inequality.} The goal of this section is to generalize the Poincar\'e inequalities to the setting of differential forms.

\begin{proposition}[Poincar\'e inequalities] \label{PoincareineqH10conv} 
Let $U$ be a bounded domain of $\Rd$ which is either smooth or convex. Then there exists a constant $C := C(U) < \infty$, such that for each integer $r \in \{1, \ldots, d \}$ and for each form $v \in H^1_{\di,0}\Lambda^r(U)$,
\begin{equation} \label{Poincareineqform}
\inf_{\alpha \in C_{\di,0}^r(U)} \|v  - \alpha \|_{L^2 \Lambda^r(U)} \leq C \|\di v \|_{L^2 \Lambda^{r+1}(U)},
\end{equation}
and for each form $w \in H^1_\di\Lambda^r(U)$,
\begin{equation} \label{PoincareWineqform}
\inf_{\alpha \in C_\di^r(U)} \|w  - \alpha \|_{L^2 \Lambda^r(U)} \leq C \|\di w \|_{L^2 \Lambda^{r+1}(U)}.
\end{equation}
Moreover, the constant $C$ has the following scaling property, for each $\lambda > 0$, $C(U) = \lambda C( \lambda^{-1} U).$
\end{proposition}

\begin{proof}
We first notice that both estimates are simple when $r = d$ since in that case $ C_\di^r(U) = H^1_\di \Lambda^d (U)$. From now on, we assume $0 \leq r \leq d-1$. In the case when $U$ convex, both inequalities~\eqref{Poincareineqform} and~\eqref{PoincareWineqform} are a consequence of Proposition~\ref{Poincarelemma}. In the case when $U$ is smooth, the proof can be split into two steps:
\begin{itemize}
\item In Step 1, we prove that the space
$
\left\{ u \in L^2 \Lambda^{r+1}(U) ~:~ \exists \alpha \in H^1_{\di } \Lambda^r (U) ~ \mbox{such that} ~ u = \di \alpha \right\}
$
is closed in the $L^2\Lambda^{r+1}(U)$-topology;
\item In Step 2, we deduce the estimates~\eqref{Poincareineqform} and~\eqref{PoincareWineqform} from Step 1 and Proposition~\ref{hodgedecomp}.
\end{itemize}

\textit{Step 1.} The argument relies on a decomposition of the space $\mathcal{H}^{r+1}(U)$ of harmonic forms, called the Friedrichs decomposition. By~\cite[Theorem 2.4.8]{Sch95}, we have the following orthogonal decomposition,
\begin{equation*}
\mathcal{H}^{r+1}(U) = \left( \mathcal{H}^{r+1}(U) \cap H^1_{\delta , 0} \Lambda^{r+1} (U) \right) \overset{\perp}{\oplus} \left\{ u \in \mathcal{H}^{r+1}(U) ~:~ \exists \alpha \in  H^1_{\di} \Lambda^{r} (U) ~\mbox{such that}~ u = \di \alpha \right\}.
\end{equation*}
Combining this result with Proposition~\ref{hodgedecomp} shows that the space
\begin{multline*}
\left\{ u \in L^2 \Lambda^{r+1}(U) ~:~ \exists \alpha \in H^1_{\di } \Lambda^r (U) ~ \mbox{such that} ~ u = \di \alpha \right\} \\ = \mathcal{E}^r(U) \overset{\perp}{\oplus} \left\{ u \in \mathcal{H}^{r+1}(U) ~:~ \exists \alpha \in  H^1_{\di } \Lambda^{r} (U) ~\mbox{such that}~ u = \di \alpha \right\}
\end{multline*}
is closed for the $L^2\Lambda^{r+1}(U)$-topology.

\smallskip

\textit{Step 2.}
We first prove the inequality~\eqref{Poincareineqform}. By Proposition~\ref{hodgedecomp}, we know that the space $\mathcal{E}^{r}$ is closed in $L^2 \Lambda^{r+1}(U)$. This implies that the range of the linear operator
\begin{equation*}
\di : \left\{ \begin{aligned}
H^1_{\di, 0} \Lambda^{r} (U) &\rightarrow L^2 \Lambda^{r+1} (U), \\
u &\rightarrow \di u,
\end{aligned}  \right.
\end{equation*}
is closed. Thus, by the Open Mapping Theorem, see \cite[Theorem 2.6 and Corollary 2.7]{brezis2010functional}, there exists a constant $C(U) < \infty$ such that for each form $v \in H^1_{\di , 0} \Lambda^r(U)$,
\begin{equation*}
\inf_{\alpha \in \ker \di} \| v - \alpha \|_{L^2\Lambda^{r} (U)} \leq C \| \di v \|.
\end{equation*}
But one has $\ker \di = C_\di^r(U) \cap H^1_{\di , 0} \Lambda^r (U)$. This completes the proof of~\eqref{Poincareineqform}.

The proof of the estimate~\eqref{PoincareWineqform} is similar, the only difference is that we use Step 1, instead of Proposition~\ref{hodgedecomp}, to obtain that the set
$
\left\{ u \in L^2 \Lambda^{r+1}(U) ~:~ \exists \alpha \in H^1_{\di } \Lambda^r (U) ~ \mbox{such that} ~ u = \di \alpha \right\}
$
is closed in the $L^2\Lambda^{r+1}(U)$-topology.

The scaling of the constant can be deduced by applying the change of variables $x \rightarrow \lambda x$.
\end{proof}

\subsection{The Gaffney-Friedrichs inequality.} This section is devoted to the statement of the Gaffney-Friedrichs inequality (originally due to~\cite{gaffney1951harmonic, friedrichs1955differential}). Let us first introduce the space of harmonic forms with Dirichlet boundary condition in a bounded domain $U$,
\begin{equation*}
\mathcal{H}_{D}^r(U) := \left\{ u \in L^2 \Lambda^r (U) ~: ~ \di u = 0, ~ \delta u = 0~ \mbox{and}~ \mathbf{t} u = 0 ~\mbox{on}~ \partial U \right\}.
\end{equation*} 
The Gaffney-Friedrichs inequality states that if a differential form satisfies $\di u \in L^2 \Lambda^{r+1} \left( U \right)$, $\di u \in L^2 \Lambda^{r-1} \left( U \right)$, $\mathbf{t}u = 0$ on the boundary $\partial U$ and is orthogonal to the space $\mathcal{H}_{D}^r$, then it belongs to the Sobolev space $H^1$, and the $L^2$-norm of its gradient can be estimated by the $L^2$-norms of the forms $\di u$ and $\delta u$. The proof of the version of the inequality stated below can be found in~\cite[Proposition 2.2.3 and Theorem 5.5]{Sch95}.

\begin{proposition}[Gaffney-Friedrichs inequality for smooth and convex domains] \label{GFCDsmooth}
Let $U$ be a bounded domain of $\Rd$ which is either smooth or convex. There exists a constant $C := C(d,U) < \infty$ such that for any form $u \in L^2\Lambda^r(U)$ satisfying $d u \in L^2\Lambda^{r+1}(U)$, $\delta u \in L^2\Lambda^{r-1}(U)$, $\mathbf{t} u = 0 $ on $\partial U$ and $u \in \left( \mathcal{H}_{D}^r(U)\right)^\perp$, one has $u \in H^1 \Lambda^r (U)$ and
\begin{equation*}
\| \nabla u \|_{L^2 \Lambda^r (U)} \leq C \left( \| \di u \|_{L^2 \Lambda^{r+1} (U)} + \| \delta u \|_{L^2  \Lambda^{r-1} (U)} \right).
\end{equation*}
\end{proposition}

\begin{remark}
In the case when the domain $U$ is convex, one has $\mathcal{H}_{D}^r(U) = \left\{ 0 \right\}$ and the restriction $u \in \left( \mathcal{H}_{D}^r(U)\right)^\perp$ is vacuous.
\end{remark}

\subsection{The multiscale Poincar\'e inequality.}
Another ingredient needed in the proof of Theorem~\ref{maintheorem} is the multiscale Poincar\'e inequality stated in Proposition~\ref{multscalepoincare} below. We first define the mean of a form on a cube.
\begin{definition}
Given a cube $\cu$ of $\Rd$, an integer $r \in \{0 , \ldots , d\}$ and a form $\alpha = \sum_{|I| = r} \alpha_I \di x_I \in L^2 \Lambda^r(\cu)$. We denote by
$
\left( \alpha \right)_\cu :=  \sum_{|I| = r} \left( \fint_\cu \alpha_I (x) \, dx \right) \di x_I \in \Lambda^{r}(\Rd).
$
\end{definition}
The multiscale Poincar\'e inequality is stated in the following proposition.
\begin{proposition}[Multiscale Poincar\'e]  \label{multscalepoincare}
Fix $m \in \N$ and, for each $0 \leq r \leq d$, each $n \in \N$, $n \leq m$, define $\mathcal{Z}_{m,n}:= 3^n \Zd \cap \cu_m$. There exists a constant $C(d) < \infty$ such that, for every form 
\begin{equation*}
\| u  \|_{\underline{L}^2(\cu_m)}  \leq C \| \di u \|_{\underline{L}^2(\cu_m)} + C \sum_{n = 0 }^{m-1} 3^n \left( \left| \mathcal{Z}_{m,n} \right|^{-1} \sum_{z \in \mathcal{Z}_{m,n}}  \left| (\di u)_{z + \cu_n}\right|^2 \right)^{\frac 12}.
\end{equation*}
\end{proposition}

The multiscale Poincar\'e inequality is a consequence of the following improved version of the Poincar\'e-Wirtinger inequality, for which we recall the definition of the $\underline{H}^{-1}\Lambda^r$-norm stated in Section~\ref{sectionH-1norms}. The specific case $r = 0$ of this statement can be found in~\cite[Lemma~1.9]{armstrong2017quantitative}.

\begin{proposition} \label{Poincareineq}
There exists a constant $C := C(d) < \infty$ such that for every cube $\cu$ of $\Rd$, every integer $r \in \{1 , \ldots, d\}$ and every form $u \in H^1_\di \Lambda^r (\cu),$ one has the estimate
\begin{equation*}
\inf_{\alpha \in C^r_\di (\cu)}  \| u- \alpha \|_{\underline{L}^2 \Lambda^{r}(\cu)} \leq C \| \di u \|_{\underline{H}^{-1}\Lambda^{r+1}(\cu)}.
\end{equation*}
\end{proposition}
The proof of Proposition~\ref{Poincareineq} relies on the following lemma.

\begin{lemma} \label{lemmaneumannpb}
There exists a constant $C := C(d) < \infty$ such that for each cube $\cu$ of $\Rd$, each integer $r \in \{1 , \ldots, d-1\}$ and each form $u \in C^r_\di(\cu)^\perp$, there exists a unique form $w \in H^1_\di \Lambda^r(\cu) \cap C^r_\di(\cu)^\perp$ solution of the Neumann problem
\begin{equation} \label{Neumannpb}
\left\{ \begin{aligned}
\delta \di w = u & \mbox{ in } \cu, \\
\mathbf{n} \di w = 0  & \mbox{ on } \partial \cu, \\
\end{aligned} \right.
\end{equation}
in the sense that, for each form $v \in H^1_\di \Lambda^{r} (U)$,
$
 \left\langle \di w, \di v \right\rangle_\cu = \left\langle u, v \right\rangle_\cu.
$
Moreover, the exterior derivative $\di w$ belongs to the space $H^1\Lambda^{r+1}(\cu) $ and satisfies the estimate
\begin{equation} \label{regNeumann}
\| \nabla \di w \|_{L^2 \Lambda^{r+1} (\cu)} \leq C \| u \|_{L^2 \Lambda^r (\cu)}.
\end{equation} 

\end{lemma}

\begin{proof}
The proof can be split in two steps, we first prove that there exists a function $w$ in $H^1_\di\Lambda^r(\cu)$ solution of the Neumann problem~\eqref{Neumannpb} and then that the solution $w$ satisfies $\di w \in H^1\Lambda^{r+1}(\cu)$ with the regularity estimate~\eqref{regNeumann}.

To solve the equation~\eqref{Neumannpb}, we define, for $v \in H^1_\di \Lambda^r(\cu)$,
$
\mathcal{J} (v) := \left\langle \di v, \di v \right\rangle_\cu - \left\langle u, v \right\rangle_\cu
$
and consider the variational problem
$
\inf_{v \in H^1_\di\Lambda^r(\cu) \cap C^r_\di(\cu)^\perp }\mathcal{J}(v).
$ 
By the standard minimization techniques of the calculus of variations and the Poincar\'e-Wirtinger inequality (Proposition~\ref{PoincareineqH10conv}), it is straightforward to prove that there exists a unique minimizer $w$ of this problem. By the first variation, the minimizer $w$ solves the equation~\eqref{Neumannpb}.

There remains to prove the regularity estimate~\eqref{regNeumann}. The argument is an adaptation of~\cite[Corollary 6.6]{mitrea2001dirichlet}. The main ingredient of this step is the Gaffney-Friedrichs inequality stated in Proposition~\ref{GFCDsmooth} applied with $U = \cu$ and $\omega = \di w$. This form satisfies $ \omega  \in L^2\Lambda^{r+1}(\cu)$, $\di  \omega  = \di \di u = 0 \in  L^2\Lambda^{r+2}(\cu)$, $\delta  \omega  = u \in L^2\Lambda^{r}(\cu)$ and $\mathbf{n} \omega = 0$. Thus $\omega$ belongs to the space $H^1\Lambda^{r+1} (\cu)$ and there exists a constant $C := C(\cu) < \infty$,
$
\| \nabla \omega \|_{L^2  \Lambda^{r+1}(\cu)} \leq C \| u \|_{L^2 \Lambda^r (\cu)}.
$
By translation and scaling invariance, one obtains the existence of a constant $C := C(d) < \infty$ such that
$
\| \nabla \omega \|_{L^2  \Lambda^{r+1}(\cu)} \leq C \| u \|_{L^2 \Lambda^r (\cu)}.
$
The proof of~\eqref{regNeumann} is complete.
\end{proof}

We now apply Lemma~\ref{lemmaneumannpb} to prove Proposition~\ref{Poincareineq}.

\begin{proof}[Proof of Proposition~\ref{Poincareineq}]
We first note that it is enough to prove the result when $u$ belongs to the space $H^1_\di \Lambda^r (\cu) \cap \left( C^r_\di (\cu) \right)^\perp$. Using the function $w \in H^1_\di \Lambda^r (\cu)$ solution of the Neumann problem~\eqref{Neumannpb} in the cube $\cu$, one has
\begin{equation*}
\|u \|_{\underline{L}^2 \Lambda^r (\cu)}^2
				 =   \frac{1}{|\cu|} \left\langle \di u, \di w \right\rangle_\cu 
				 \leq \| \di u \|_{\underline{H}^{-1}\Lambda^{r+1}(\cu)} \left( \size(\cu)^{-1} \left| \left( \di w \right)_{\cu} \right|  + \| \nabla \di w \|_{\underline{L}^2\Lambda^{r+1}(\cu)}\right).
\end{equation*}
By Lemma~\ref{lemmaneumannpb}, one has the estimate
\begin{equation*}
\| \nabla \di w \|_{\underline{L}^2\Lambda^{r+1}(\cu)} \leq  C \| u \|_{\underline{L}^2\Lambda^r(\cu)}.
\end{equation*}
To complete the proof, there remains to estimate $\left| \left( \di w \right)_{\cu} \right| $. To this end, we denote by
\begin{equation*}
p = \sum_{i_1 < \cdots < i_p} p_{i_1, \cdots , i_p} \di x_{i_1} \wedge \cdots \wedge  \di x_{i_1} := \frac{ \left( \di w \right)_{\cu}}{ \left| \left( \di w \right)_{\cu} \right| } ,
\end{equation*}
and denote by
\begin{equation} \label{def.lp}
l_p : \left\{ \begin{aligned}
\Rd & \rightarrow  \Lambda^p(\Rd), \\
 x  & \mapsto  \sum_{i_1 < \cdots < i_p} p_{i_1 , \cdots, i_p} x_{i_1} \di x_{i_2} \wedge \cdots \wedge \di x_{i_p},
\end{aligned}  \right.
\end{equation}
so that $\di l_p = p$. Testing the equation~\eqref{Neumannpb} with $\alpha = l_p$, one obtains
\begin{equation*}
\left| \left( \di w \right)_{\cu} \right|   =  \frac{1}{|\cu|}  \left| \left\langle p, \di w  \right\rangle_\cu \right|
								=  \frac{1}{|\cu|}  \left| \left\langle \di l_p, \di w  \right\rangle_\cu \right| 
								 =  \frac{1}{|\cu|}  \left|  \left\langle  l_p, u  \right\rangle_\cu \right| 
								 \leq C \size(\cu) \| u \|_{\underline{L}^2\Lambda^r(\cu)}.
\end{equation*}
Combining the previous results completes the proof of the proposition.
\end{proof}

We are now ready to prove Proposition~\ref{multscalepoincare}.

\begin{proof}[Proof of Proposition~\ref{multscalepoincare}]
The result is a consequence of Proposition~\ref{Poincareineq} and the version of the multiscale Poincar\'e inequality for functions stated in~\cite[Proposition 1.8]{armstrong2017quantitative} with the choice of function $f = \di u$.
\end{proof}

\subsection{The Caccioppoli inequality.} We complete Section~\ref{section4} by stating a version of the Caccioppoli inequality for differential forms. Recall the definition of the space $\Omega_r$ stated in~\eqref{defOmegar} and, given an environment $\a \in \Omega_r$, the definition of the space of solutions $\A(U)$ stated in~\eqref{defsolution}.

\begin{proposition}[Caccioppoli inequality]  \label{caccioppoli}
There exists a constant $C := C(d, \lambda) < \infty$ such that, for every $1 \leq r \leq d$, every open subsets $V, U \subseteq \Rd$ satisfying $\overline{V} \subseteq U$, and every $u \in \A(U)$,
\begin{equation*}
\|\di u \|_{L^2 \Lambda^{r+1}(V)} \leq  \frac{C}{\dist(V, \partial U)}\| u \|_{L^2 \Lambda^{r}(U \setminus V)}.
\end{equation*}
\end{proposition}
The proof of this inequality is identical to the standard proof for solutions of elliptic equations and the details are omitted.

\section{Quantitative Homogenization} \label{section5}

The goal of this section is to study the energy functional $J$ defined by the formula, for each pair $(p, q) \in \Lambda^{r}(\Rd) \times \Lambda^{d-r}(\Rd)$,
\begin{equation*}
J(U ,p ,q) :=\sup_{v \in \A(U) } \fint_U \left( -\frac 12 \di v \wedge \a \di v -  p \wedge \a \di v +   \di v \wedge q \right).
\end{equation*}
Thanks to the Poincar\'e-Wirtinger inequality, Proposition~\ref{PoincareineqH10conv}, one can prove that there exists a unique maximizer in the space $\A(U) \cap C^{r-1}_\di(U)^\perp$, denoted by $v( \cdot, U , p,q)$. The proof is similar to Step 1 of the proof of Lemma~\ref{lemmaneumannpb} and the details are omitted.

The objective of this section is to prove Theorem~\ref{maintheorem}; it is organized as follows. We first record in Proposition~\ref{basicpropJ} some useful properties about the quantity $J$. We then establish a series of lemmas, Lemmas~\ref{l.descendre} to~\ref{lemmaJflat}, before proving Theorem~\ref{maintheorem}. We finally deduce from Theorem~\ref{maintheorem} the quantitative sublinearity of the map $v( \cdot, U , p,q)$ and of the flux $\a \di v( \cdot, U , p,q)$; the result is stated in Proposition~\ref{weakcvggradientandflux}.

\subsection{The subadditive quantity $J$ and its basic properties.} \label{section4.1}

We first record some basic properties of the functional $J$; they are listed in the following proposition.

\begin{proposition}[Basic properties of $J$] \label{basicpropJ}
Fix a bounded Lipschitz domain $U \subseteq \Rd$. There exists a constant $C(d, \lambda) < \infty$ such that, for each $1 \leq r \leq d$, the quantity $J$ and the maximizer $v$ satisfy the following properties:
\begin{enumerate}
\item The map $(p , q) \mapsto J(U , p , q)$ is quadratic, and one has the estimates, for any $p , q \in \Lambda^{r}(\Rd) \times \Lambda^{d-r}(\Rd)$,
\begin{equation} \label{eq:170708}
    C^{-1} |p|^2 \leq J(U , p , 0) \leq C |p|^2 \hspace{5mm} \mbox{and} \hspace{5mm} C^{-1} |q|^2 \leq J(U , 0 , q) \leq C |q|^2.
\end{equation}
Additionally, the mapping $J$ can be decomposed as follows, for any pair $p , q \in \Lambda^{r}(\Rd) \times \Lambda^{d-r}(\Rd)$,
\begin{equation} \label{eq:170808}
    J (U,p, q)  = J (U,p, 0)  + J (U,0, q) - \star p \wedge q.
\end{equation}
\item The mapping $(p, q ) \mapsto v(\cdot,U,p,q)$ is linear.
\item One has the upper bounds, for any pair $(p,q) \in \Lambda^{r}(\Rd) \times \Lambda^{d-r}(\Rd)$,
\begin{equation} \label{eq:171508}
    0 \leq J(U,p,q) \leq C \left\| \di v(\cdot , p , q) \right\|_{\underline{L}^2 \left( U\right)} \leq C ( |p|^2+|q|^2).
\end{equation}
\item For each $(p,q) \in \Lambda^{r}(\Rd) \times  \Lambda^{d-r}(\Rd) $, the function $v ( \cdot, U,p,q)$ is characterized as the unique element of $\A(U) \cap C^{r-1}_\di(U)^\perp$ which satisfies, for each form $u \in \A(U)$,
\begin{equation} \label{e.firstvariation}
\int_U \di v \wedge \a \di u = \int_U ( - p \wedge \a \di u + \di u \wedge q ).
\end{equation}
\item Let $U_1, \ldots, U_n \subseteq U$ be bounded Lipschitz domains that form a partition of $U$, i.e., $U_i \cap U_j = \emptyset$ if $i\neq j$ and $\left| U \setminus \bigcup_{i = 1}^N U_i \right| = 0
$. Then, for each pair $(p,q) \in \Lambda^{r}(\Rd) \times  \Lambda^{d-r}(\Rd)$,
\begin{equation} \label{controloptimizertau}
\sum_{i=1}^{n} \frac{|U_i|}{|U|} \left\| \di v(\cdot ,U,p,q) - \di v(\cdot ,U_i,p,q) \right\|_{\underline{L}^2 \Lambda^r \left(U_i \right)}^2 \leq C \sum_{i=1}^{n} \frac{|U_i|}{|U|} \left( J(U_i,p,q) - J(U,p,q) \right). 
\end{equation}
We note that the inequality~\eqref{controloptimizertau} implies that the term in the right-hand side is non-negative.
\item For every linear mapping $\tilde \a \in  \mathcal{L} \left(\Lambda^{r}(\Rd), \Lambda^{(d-r)}(\Rd) \right)$ satisfying the assumptions of symmetry and ellipticity~\eqref{symandelllassumption}, one has the estimate
\begin{equation*}
\sup_{(p, q) \in B_1\Lambda^{r}(\Rd) \times B_1\Lambda^{d-r}(\Rd)}  \left| J(U ,p , q) - \frac 12 p \wedge \tilde \a p -  \frac 12 \tilde \a^{-1} q \wedge q + \star \left( p \wedge q \right) \right| \leq \sup_{p \in B_1\Lambda^{r}(\Rd)} \left(  J\left( U , p , \tilde \a p \right) \right)^\frac12.
\end{equation*}
\end{enumerate}
\end{proposition}

\begin{proof}
These properties are easy to check and their proofs are almost the same of those of \cite[Lemma 2.2]{armstrong2017quantitative} and~\cite[Lemma 2.7]{armstrong2017quantitative} for the property (6), so we omit the details.
\end{proof}
We first proceed to the reduction explained in Section~\ref{section1.3} and note that, thanks to the property~(6) of Proposition~\ref{basicpropJ}, Theorem~\ref{maintheorem} is implied by the following proposition.

\begin{proposition} \label{maintheorem2}
Let $r$ be an integer in $\{ 1 , \ldots, d \}$. There exist an exponent $\alpha(d, \lambda) > 0$, a constant $C(d, \lambda) < \infty$ and a linear mapping $\ahom \in \mathcal{L} \left(\Lambda^{r}(\Rd), \Lambda^{(d-r)}(\Rd) \right)$ satisfying the assumptions~\eqref{symandelllassumption} such that, for every integer $n \in \N$,
\begin{equation} \label{statementmaintheorem2}
\sup_{p \in B_1\Lambda^{r}(\Rd)}  J(\cu_n ,p , \ahom p)  \leq \O_1 \left( C 3^{-n \alpha}\right).
\end{equation}
\end{proposition}

The rest of this section is devoted to the demonstration of Proposition~\ref{maintheorem2} and we now turn to the proofs of a series of lemmas, which are then used in its proof. In the following lemma, we use the inequality~\eqref{controloptimizertau} to obtain a control on the spatial average of the exterior derivative of the maximizer $v( \cdot , \cu_m , p , q)$ in the triadic cube $\cu_m$ in terms of the subadditivity defect of the energy $J$. 

We first introduce a few definitions. For any pair of integers $m,n\in\N$ such that $n<m$, we denote by $\Ze_{m,n} := 3^n \Zd \cap \cu_m$; it is a finite set of cardinality $3^{d(m-n)}$. For each pair $(p,q) \in \Lambda^{r}(\Rd) \times \Lambda^{d-r}(\Rd)$ and each collection $\{ q_z' \}_{z\in \Ze_{m,n}} \in \Lambda^{d-r}(\Rd)$, we introduce the following notation,
$v:= v(\cdot,\cu_m,p,q)$ and $v_z:= v(\cdot,z+\cu_n,p,q).
$

\begin{lemma}
\label{l.descendre}
One has the estimate
\begin{equation}
\label{e.descendre}
\frac{1}{|\cu_m|} \sum_{z\in \Ze_{n,m}} \left| \int_{z + \cu_n}  \left( \di v - \di v_z \right) \wedge q_z' \right|
\leq  C
\left( \avsum_{z\in \Ze_{n,m}} |q_z'|^2  \right)^{\frac12} 
\left( \avsum_{z\in \Ze_{n,m}} J(z+\cu_n,p,q) - J(\cu_m,p,q) \right)^{\frac12}.
\end{equation}
\end{lemma}
\begin{proof}
We compute, using the H\"older inequality, 
\begin{align*}
\lefteqn{ 
\frac{1}{|\cu_m|} \sum_{z\in \Ze_{n,m}} \left| \int_{z + \cu_n} q_z' \wedge  \left( \di v - \di v_z \right) \right|
} \qquad & 
\\ & 
\leq  \frac{C}{|\cu_m|} \sum_{z\in \Ze_{n,m}} |q_z'| \left\| \di v - \di v_z  \right\|_{L^2(z+\cu_n)} 
\\ & 
\leq C\left(  \frac{1}{|\cu_m|} \sum_{z\in \Ze_{n,m}} 
 |q_z'|^2 \right)
\left(  \frac{1}{|\cu_m|}  \sum_{z\in \Ze_{n,m}} \left\| \left( \di v - \di v_z \right)  \right\|_{L^2(z+\cu_n)}^2 \right)^{\frac12}.
\end{align*}
Applying the estimate~\eqref{controloptimizertau} completes the proof of the inequality~\eqref{e.descendre}.
\end{proof}

\subsection{Estimate on the variance of the slope of the maximizer $v$.} \label{section4.2}
Given a differential form $u$, by analogy to the case of functions, we refer to the slope of $u$ over a bounded domain $U \subseteq \Rd$ as the mean value of its exterior derivative $\left( \di u(\cdot ,\cu_m,p,q) \right)_{U}.$ This first lemma shows that, under the assumption of finite range dependence of the environment, the variance of the slope of the differential form $v$ has to contract. The statement is quantified in terms of the expectation of the subadditivivity defect $\tau_n$ which is defined by the formula 
$$\tau_n := \sup_{(p,q) \in B_1\Lambda^{r}(\Rd) \times  B_1\Lambda^{d-r}(\Rd)}  \E \left[ J(\cu_n,p,q) - J(\cu_{n+1},p,q) \right].$$

The proof of Lemma~\ref{l.flatness} relies the standard concentration estimate for the variance of a sum of random variables and the estimate~\eqref{controloptimizertau}. 

\begin{lemma}
\label{l.flatness}
Let $m,n \in\N$ with $0 \leq n\leq m-2$. Then there exists a constant $C(d, \lambda)<\infty$ such that, for every pair $(p,q) \in B_1\Lambda^{r}(\Rd) \times  B_1\Lambda^{d-r}(\Rd)$,
\begin{equation} \label{e.flatness0}
\var\left[ \left( \di v(\cdot ,\cu_m,p,q) \right)_{\cu_m} \right] 
\leq 
C 3^{-d(m-n)} \var\left[ \left(  \di v(\cdot,\cu_n,p,q) \right)_{\cu_n} \right]
+ C\sum_{k=m}^n \tau_k. 
\end{equation}
\end{lemma}
\begin{proof}
We first fix two integers $m, n\in\N$ satisfying $n\leq m-2$ and a form $q'\in B_1\Lambda^{d-r}(\Rd)$. We apply Lemma~\ref{l.descendre} with the value $q_z' := q'$ and obtain
\begin{multline}
\label{e.descendre0}
\frac{1}{|\cu_m|} \left|  \int_{ \cu_m}  \di v(\cdot,\cu_m,p,q) \wedge q' - \sum_{z\in \Ze_{n,m}} \int_{z + \cu_n} \di v(\cdot, z + \cu_n,p,q) \wedge q'  \right|
\\
\leq 
C 
\left( 3^{d(n-m)}\avsum_{z\in \Ze_{n,m}} J(z+\cu_n,p,q) - J(\cu_m,p,q) \right)^{\frac12}.
\end{multline}
From the estimate~\eqref{e.descendre0}, we deduce
\begin{align}
\label{e.descendrebis}
\var\left[ \fint_{ \cu_m} \di v(\cdot,\cu_m,p,q) \wedge q' \right] 
 &  
\leq 
2 \var\left[   \frac{1}{|\cu_m|} \sum_{z\in \Ze_{n,m}} \int_{z + \cu_n} \di v(\cdot, z + \cu_n,p,q) \wedge q'  \right]
\\ & \quad 
+ 2 C 3^{d(n-m)} \E \left[
\avsum_{z\in \Ze_{n,m}} J(z+\cu_n,p,q) - J(\cu_m,p,q) \right]. \notag
\end{align}
We take an enumeration $\{ z_{i,j} ~:~ 1 \leq i \leq 3^d, 1 \leq j \leq 3^{d(m-n-1)} \}$ of the set $\mathcal{Z}_{m,n}$ such that for each $1 \leq i \leq 3^d$ and each $1 \leq j,j' \leq  3^{d(m-n-1)}$,
$
|z_{i,j} - z_{i,j'}| \geq  2\cdot 3^n.
$
This gives in particular
$
\dist( z_{i,j} + \cu_n , z_{i,j'} + \cu_n ) \geq 3^n.
$
By the finite range dependence assumption~\eqref{independenceassumption}, we know that the $\sigma$-algebras
$\F_r( z_{i,j} + \cu_n )$ and $\F_r( z_{i,j'} + \cu_n )$ are independent.
We then compute
\begin{align*} \label{}
\lefteqn{ 
\var\left[ \frac{1}{|\cu_m|} \sum_{z\in \Ze_{n,m}} \int_{z + \cu_n} \di v(\cdot, z + \cu_n,p,q) \wedge q'  \right]
} \quad & 
\\ & 
= 3^{-2dm} \var\left[  \sum_{i = 1}^{3^d} \sum_{j=1}^{3^{d(m-n-1)}}  \int_{z_{i,j} + \cu_n} \di v(\cdot, z_{i,j} + \cu_n ,p,q) \wedge q' \right]
\\ &
\leq 3^{-2dm + d} \sum_{i = 1}^{3^d}  \var\left[ \sum_{j=1}^{3^{d(m-n-1)}}  \int_{z_{i,j} + \cu_n} \di v(\cdot, z_{i,j} + \cu_n ,p,q) \wedge q' \right]
\\ &
\leq 3^{-2dm + d} \sum_{i = 1}^{3^d} \sum_{j=1}^{3^{d(m-n-1)}}  \var\left[ \int_{z_{i,j} + \cu_n} \di v(\cdot, z_{i,j} + \cu_n ,p,q) \wedge q' \right]
\\ &
\leq 3^{d(-m + 1 - n )} \var\left[ \int_{ \cu_n}  \di v(\cdot, \cu_n ,p,q)  \wedge q' \right]
\\ &
\leq  C 3^{-d(m - n)} \var\left[\fint_{ \cu_n} \di v(\cdot, \cu_n ,p,q) \wedge q' \right].
\end{align*}
Combining the previous display with~\eqref{e.descendrebis} and taking the supremum over $q' \in B_1\Lambda^{d-r}(\Rd)$ completes the proof of the lemma.
\end{proof}

We now iterate Lemma~\ref{l.flatness} to obtain a control on the variance of the slope of the maximizer $ v(\cdot,\cu_m,p,q)$ in terms of the subadditivity defect $\tau_k$.
\begin{lemma} \label{l.flatness2}
There exist a constant $C(d, \lambda ) < \infty$ and an exponent $\beta := \beta (d, \lambda) > 0$ such that for every $(p,q) \in  B_1\Lambda^{r}(\Rd) \times  B_1\Lambda^{d-r}(\Rd)$ and every integer $m \in \N$,
\begin{equation} \label{e.flatness}
\var\left[ \left( \di v(\cdot,\cu_m,p,q) \right)_{\cu_m} \right] \leq C \sum_{n=0}^m  3^{\beta (n-m)} \tau_n + C 3^{-\beta m}.
\end{equation}
\end{lemma}

\begin{proof}
We denote by $C := C(d , \lambda) < \infty$ the constant of Lemma~\ref{l.flatness} and select an integer $l := l(d, \lambda) \in \N$ such that 
$
3^{-d - 1} < C 3^{-dl} \leq \frac 13 .
$
The inequality~\eqref{e.flatness0} applied with $n = m-l$ yields
\begin{equation*}
\var\left[\left( \di v(\cdot,\cu_m,p,q) \right)_{\cu_{m-l}} \right] 
\leq 
\frac 13 \var\left[ \left( \di v(\cdot ,\cu_m,p,q) \right)_{\cu_{m-l}} \right]
+ C \sum_{k = n-l}^n \tau_k. 
\end{equation*}
Iterating this estimate and using the bound on the $L^2$-norm of the exterior derivative $\di v$ stated in~\eqref{eq:171508} gives, for some constant $C := C(d , \lambda) < \infty,$
\begin{equation*}
\var\left[ \left( \di v(\cdot,\cu_m,p,q) \right)_{\cu_m} \right] 
\leq 
C 3^{-\frac nl}
+ C \sum_{k = 0}^n 3^{\frac{n-k}{l}} \tau_k. 
\end{equation*}
This completes the proof of the lemma with the value $\beta := \frac 1l$.
\end{proof}

\subsection{Flatness of the maximizers and control of the energy.} \label{section4.3}
We begin this section by giving a definition of the homogenized tensor $\ahom$, and, for a bounded domain $U \subseteq \Rd$, by defining the best guess approximation one can make of the matrix $\ahom$ by only looking at the environment in the domain~$U$. The precise definition is the following: using that the map $q \mapsto J(\cu_n,0,q)$ is quadratic and bounded from above and below according to the estimate~\eqref{eq:170708}, we let $\ahom_n$ be the unique linear map, which sends $\Lambda^{r}(\Rd)$ to $\Lambda^{d-r}(\Rd)$, satisfies the symmetry assumption~\eqref{symmetryassumption} and such that, for every $q \in \Lambda^{d-r}(\Rd)$,
\begin{equation} \label{def.ahomU}
\E \left[ J(\cu_n , 0, q) \right] = \frac 12  \star ( \ahom_n^{-1} q \wedge  q) .
\end{equation}

We record three properties about this quantity:
\begin{itemize}
\item Since $J$ satisfies the subadditivity property~\eqref{controloptimizertau} and since the environment satisfies the stationarity assumption~\eqref{stationarityassumption}, the sequence $\left( \E \left[ J(\cu_n , 0, q) \right] \right)_{n \in \N}$ is decreasing. Consequently it converges. This gives the definition of the homogenized tensor: it is the one which satisfies, for each $q \in \Lambda^{d-r}(\Rd)$,
\begin{equation*}
\E \left[ J(\cu_n , 0, q) \right] \underset{n \rightarrow \infty}{\longrightarrow}  \frac 12  \star (\ahom^{-1} q \wedge  q).
\end{equation*}
It is defined equivalently to be the limit
$
 \ahom_n \rightarrow  \ahom \mbox{ in } \mathcal{L} \left( \Lambda^{r}(\Rd) ,  \Lambda^{d-r}(\Rd)  \right).
$
Moreover, by~\eqref{eq:170708}, one has the estimate, for each $p \in \Lambda^{r}(\Rd)$ and each $n \in \N$,
$
C^{-1} |p|^2 \leq p \wedge \ahom_n p \leq C |p|^2.
$
Sending~$n$ to infinity shows that the same estimate holds for the tensor $\ahom$.
\item One has the formula, for each $q \in \Lambda^{d-r}(\Rd)$,
\begin{equation} \label{equivalentdefahomn}
 \ahom_n^{-1} q = \E \left[ \left( \di v(\cdot, \cu_n,0,q) \right)_{\cu_n} \right].
\end{equation}
which is a consequence of the first variation~\eqref{e.firstvariation}.
\item The difference between two values of the sequence $(\a_n)_{n \in \N}$ can be estimated in terms of the subadditivity defect $\tau_n$: for every $q \in B_1\Lambda^{d-r}(\Rd)$, $m,n \in \N$ such that $n < m$, we have
\begin{align} \label{ahomnahomm}
\lefteqn{|\ahom_m^{-1} q - \ahom_n^{-1} q|^2 } \qquad & \\ &
 = \left| \E \left[\left( \di v(\cdot ,\cu_m,0,q) \right)_{\cu_m} - 3^{d(n-m)} \sum_{z \in 3^n \Zd \cap \cu_m}\left( \di v(\cdot ,z + \cu_n,0,q) \right)_{z + \cu_n} \right] \right|^2 \notag \\ &
\leq \E \left[3^{d(n-m)} \sum_{z \in 3^n \Zd \cap \cu_m} \left\| \di v(\cdot, \cu_m, 0, q) - \di v(\cdot, z + \cu_n, 0, q) \right\|^2_{\underline{L}^2 \Lambda^r(U) } \right] \notag \\ &
\leq C \E \left[ J(\cu_n , 0 , q) - J(\cu_m , 0 , q)  \right] \notag \\ &
\leq C \sum_{k = n}^{m-1} \tau_k. \notag
\end{align}
\end{itemize}

For $p \in \Lambda^{r}(\Rd)$ and $m \in \N$, we denote by $l_p^m$ the unique element of $C^{r-1}_\di(\cu_m)^\perp$ such that $\di l_p^m = p$. It is the projection of the affine function $l_p$ defined in~\eqref{def.lp} on the space $C^{r-1}_\di (\cu_m)^\perp$. In the following lemma, we deduce from Lemma~\ref{l.flatness2} and the multiscale Poincar\'e inequality that the form $v( \cdot , \cu_{m+1}, p ,q )$ is close to the affine function $ l_{\ahom_m^{-1} q - p}^{m+1}$. The estimate is quantified in terms of the subadditivity defect $\tau_n$ or more precisely the suitably averaged sum of the subadditivity defects over all the scales from $1$ to $3^n$.

\begin{lemma} \label{lemmavflat}
There exists a constant $C := C(d, \lambda) < \infty$ such that, for every $m \in \N$, $(p,q) \in B_1\Lambda^{r}(\Rd) \times  B_1\Lambda^{d-r}(\Rd)$,
\begin{equation*}
\E \left[ \left\| v( \cdot , \cu_{m+1}, p ,q ) - l_{\ahom_m^{-1} q - p}^{m+1} \right\|^2_{\underline{L}^2 \Lambda^{r-1}(\cu_{m+1})} \right] \leq C 3^{( 2 -\beta )m} + C 3^{( 2 -\beta )m}  \sum_{k=0}^m 3^{\beta k} \tau_k.
\end{equation*}
\end{lemma}

\begin{proof}
Fix $(p,q) \in B_1\Lambda^{r}(\Rd) \times  B_1\Lambda^{d-r}(\Rd)$. Since, by definition, both forms $v(\cdot, \cu_{m+1},p,q)$ and $ l_{\ahom_m^{-1} q - p}^{m+1}$ belong to the space $C^{r-1}_\di(\cu_{m+1})^\perp$, the difference belongs to the space $C^{r-1}_\di(\cu_{m+1})^\perp$. We can thus apply the multiscale Poincar\'e inequality stated in Proposition~\ref{multscalepoincare},
\begin{align} \label{multiscalepoincarev}
\lefteqn{ \left\| v( \cdot , \cu_{m+1}, p ,q ) - l_{\ahom_m^{-1} q - p}^{m+1} \right\|^2_{\underline{L}^2 \Lambda^{r-1}(\cu_{m+1})}} \qquad & \\ & \leq C \left\| \di v( \cdot , \cu_{m+1}, p ,q ) -\ahom_m^{-1} q + p \right\|^2_{\underline{L}^2 \Lambda^r(\cu_{m+1})}   \notag\\
& + C \left( \sum_{n=0}^m 3^n \left( |\Ze_{m,n}|^{-1} \sum_{y \in \Ze_{m,n}} \left| \left( \di v( x , \cu_{m+1}, p ,q ) \, dx - \ahom_m^{-1} q + p  \right)_{z + \cu_n} \right|^2 \right)^\frac 12 \right)^2. \notag
\end{align}
The first term on the right side is estimated by the following computation
\begin{equation*}
  \left\| \di v( \cdot , \cu_{m+1}, p ,q ) -\ahom_m^{-1} q + p \right\|^2_{\underline{L}^2 \Lambda^r(\cu_{m+1})} 
\leq 2 \left| -\ahom_m^{-1} q + p \right|^2 + 2 \left\| \di v( \cdot , \cu_{m+1}, p ,q ) \right\|^2_{\underline{L}^2 \Lambda^r(\cu_{m+1})} \leq C .
\end{equation*}
We then estimate the second term on the right side of~\eqref{multiscalepoincarev} and prove the estimate, for every $0 \leq n \leq m$,
\begin{equation*}
\E \left[ |\Ze_{m,n}|^{-1} \sum_{y \in \Ze_{m,n}} \left|  \left( \di v( \cdot , \cu_{m+1}, p ,q ) - \ahom_m^{-1} q + p \right)_{y + \cu_{n}} \right|^2 \right]  \leq C \left( 3^{-n} + \sum_{k=0}^n 3^{k-n} \tau_k + \sum_{k=n}^m \tau_k \right).
\end{equation*}
By the inequality~\eqref{controloptimizertau}, we have, for every $(p,q) \in B_1 \Lambda^r(\Rd) \times  B_1 \Lambda^{d-r}(\Rd)$,
\begin{multline*}
|\Ze_{m,n}|^{-1} \sum_{y \in\Ze_{m,n}} \left\|  \di v( \cdot , \cu_{m+1}, p ,q ) -  \di v( \cdot ,z + \cu_{n}, p ,q ) \right\|^2_{\underline{L}^2 \Lambda^r (y + \cu_n)} \\
\leq C |\Ze_{m,n}|^{-1} \sum_{z \in \Ze_{m,n}}  \left( J(z + \cu_n, p ,q ) - J( \cu_m, p , q ) \right).
\end{multline*}
Taking expectations and using the stationarity of the environment yields
\begin{multline*}
|\Ze_{m,n}|^{-1} \E \left[  \sum_{y \in \Ze_{m,n}} \left\|  \di v(\cdot , \cu_{m+1}, p ,q ) -  \di v( \cdot ,z + \cu_{n}, p ,q ) \right\|^2_{\underline{L}^2 \Lambda^r (y + \cu_n)} \right] \\
\leq C \E \left[ J(\cu_n, p ,q ) - J(\cu_m, p , q ) \right] \leq C \sum_{k = n}^{m} \tau_k.
\end{multline*}
The triangle inequality, the previous display and Lemma~\ref{l.flatness2} then yield
\begin{align*}
\lefteqn{  |\Ze_{m,n}|^{-1} \sum_{y \in \Ze_{m,n}} \E \left[ \left|  \left( \di v( \cdot , \cu_{m+1}, p ,q )  -  \ahom_n^{-1} q + p \right)_{y+\cu_n}  \right|^2 \right] } \qquad & \\ &
\leq 3  |\Ze_{m,n}|^{-1} \sum_{y \in \Ze_{m,n}}\E \left[  \left| \left( \di v( \cdot , \cu_{m+1}, p ,q ) -  \di v( x ,y + \cu_{n}, p ,q ) \right)_{y + \cu_{n}}  \right|^2 \right]  \\
& \hspace{4mm} + 3  |\Ze_{m,n}|^{-1} \sum_{y \in \Ze_{m,n}} \E \left[ \left|    \left(  \di v( \cdot  ,y + \cu_{n}, p ,q )  - \ahom_n^{-1} q + p  \right)_{y + \cu_{n}} \right|^2  \right] \\
& \hspace{4mm} + 3 |\ahom_m^{-1} q - \ahom_n^{-1} q|^2 \\
& \leq C \sum_{k = n}^{m} \tau_k + C \sum_{k=0}^n 3^{\beta(k-n)} + C 3^{-\beta n }.
\end{align*}
Combining this estimate and inequality~\eqref{multiscalepoincarev} shows
\begin{equation} \label{vaffineXm}
 \left\| v( \cdot , \cu_{m+1}, p ,q ) - l_{\ahom_m^{-1} q - p}^{m+1} \right\|^2_{\underline{L}^2 \Lambda^{r-1}(\cu_{m+1})}  \leq C \left( 1 + \left( \sum_{n=0}^m 3^n X_n^{\frac 12} \right)^2\right),
\end{equation}
where the random variable
\begin{equation*}
X_n :=  |\Ze_{m,n}|^{-1} \sum_{y \in \Ze_{m,n}} \left| \left(  \di v( x , \cu_{m+1}, p ,q ) \, dx - \ahom_m^{-1} q + p  \right)_{y + \cu_{n}} \right|^2  
\end{equation*}
satisfies the inequality
\begin{equation*}
\E \left[ X_n \right] \leq C \sum_{k = n}^{m} \tau_k + C \sum_{k=0}^n 3^{\beta(k-n)} \tau_k + C 3^{-\beta n }.
\end{equation*}
We then apply the Cauchy-Schwarz inequality,
\begin{equation*}
 \left( \sum_{n=0}^m 3^n X_n^{\frac 12} \right)^2 \leq \left( \sum_{n=0}^m 3^n \right) \left( \sum_{n=0}^m 3^n X_n \right) \leq C 3^m  \sum_{n=0}^m 3^n X_n,
\end{equation*}
and take the expectation to obtain
\begin{equation*}
\E \left[ \left( \sum_{n=0}^m 3^n X_n^{\frac 12} \right)^2  \right] \leq C 3^m \left(  \sum_{n=0}^m \sum_{k = n}^{m} 3^n \tau_k + C  \sum_{n=0}^m  \sum_{k=0}^n 3^{(1- \beta)n} 3^{\beta k} \tau_k + C  \sum_{n=0}^m  3^{(1-\beta) n} \right).
\end{equation*}
We then simplify the term on the right-hand side and write
\begin{equation*}
 \sum_{n=0}^m \sum_{k = n}^{m} 3^n \tau_k = \sum_ {k = 0}^m \sum_{n = 0}^k 3^n \tau_k \leq C \sum_ {k = 0}^m 3^k \tau_k
\end{equation*}
and
\begin{equation*}
 \sum_{n=0}^m  \sum_{k=0}^n 3^{(1- \beta)n} 3^{\beta k} \tau_k \leq \sum_{k=0}^m  \sum_{n=k}^n 3^{(1- \beta)n} 3^{\beta k} \tau_k \leq C 3^{(1 - \beta) m}  \sum_{k=0}^m 3^{\beta k} \tau_k.
\end{equation*}
Combining the three previous displays shows
\begin{equation} \label{expectationXm}
\E \left[ \left( \sum_{n=0}^m 3^n X_n^{\frac 12} \right)^2  \right] \leq  C 3^{( 2 -\beta )m} + C 3^{( 2 -\beta )m}  \sum_{k=0}^m 3^{\beta k} \tau_k + C 3^m \sum_ {k = 0}^m 3^k \tau_k.
\end{equation}
Moreover, since the exponent $\beta$ belongs to the interval $(0,1]$, we note that, for each pair of integers $k,m \in \N$ satisfying $k \leq m,$ one has the estimate $3^{ (k-m)} \leq 3^{\beta (k-m)}$. In particular the third term on the right side of~\eqref{expectationXm} is smaller than the second term on the right-hand side. Consequently, the estimate~\eqref{expectationXm} can be simplified and we obtain
\begin{equation*}
\E \left[ \left( \sum_{n=0}^m 3^n X_n^{\frac 12} \right)^2  \right] \leq  C 3^{( 2 -\beta )m} + C 3^{( 2 -\beta )m}  \sum_{k=0}^m 3^{\beta k} \tau_k.
\end{equation*}
Thus the estimate~\eqref{vaffineXm} becomes
\begin{equation*}
\E \left[ \left\| v( \cdot , \cu_{m+1}, p ,q ) - l_{\ahom_m^{-1} q - p}^{m+1} \right\|^2_{\underline{L}^2 \Lambda^{r-1}(\cu_{m+1})}  \right] \leq C 3^{( 2 -\beta )m} + C 3^{( 2 -\beta )m}  \sum_{k=0}^m 3^{\beta k} \tau_k.
\end{equation*}
The proof of Lemma~\ref{lemmavflat} is complete.
\end{proof}

Applying Lemma~\ref{lemmavflat} with the specific value $q = \ahom_m p$, we obtain that the form $v \left(\cdot , \cu_m , p , \ahom_m p \right)$ is close to the constant function equal to $0$ in the $L^2$-norm. Combining this result with the Caccioppoli inequality, we obtain that the $L^2$-norm of the exterior derivative $\di v \left(\cdot , \cu_m , p , \ahom_m p \right)$ is close to $0$ and can be estimated in terms of the subadditivity defect $\tau_n$. This implies that the expectation of the energy $\E \left[ J(\cu_m,p,\ahom_m p) \right]$ is small and can be estimated in terms of the subadditivity defect $\tau_n$. This result is proved in the following lemma.

\begin{lemma}\label{lemmaJflat}
There exists a constant $C(d,\lambda) < \infty$ such that, for every $m \in \N$ and $p \in B_1 \Lambda^{r}(\Rd)$,
\begin{equation*}
\E \left[ J(\cu_m,p,\ahom_m p) \right]  \leq C 3^{-\beta m} + C 3^{ -\beta m}  \sum_{k=0}^m 3^{\beta k} \tau_k.
\end{equation*}
\end{lemma}

\begin{proof}
Fix $p \in B_1 \Lambda^{r}(\Rd)$. By Lemma~\ref{lemmavflat}, we have the estimate
\begin{equation*}
\E \left[  \left\| v(\cdot , \cu_{m+1} , p , \ahom_m p ) \right\|^2_{\underline{L}^2 \Lambda^{r-1} ( \cu_{m+1})} \right]  \leq  C 3^{( 2 -\beta )m} + C 3^{( 2 -\beta )m}  \sum_{k=0}^m 3^{\beta k} \tau_k.
\end{equation*}
Applying the Caccioppoli inequality stated in Proposition~\ref{caccioppoli}, one obtains
\begin{equation} \label{caccioppoliv0}
\E \left[ \left\| \di v(\cdot , \cu_{m+1} , p , \ahom_m p ) \right\|^2_{\underline{L}^2 \Lambda^{r-1} ( \cu_{m})} \right]  \leq  C 3^{-\beta m} + C 3^{ -\beta m}  \sum_{k=0}^m 3^{\beta k} \tau_k.
\end{equation}
By the estimate~\eqref{controloptimizertau}, we have
\begin{equation*}
3^{-d} \sum_{y \in 3^m \Zd \cap \cu_{m+1}} \E \left[ \left\| \di v(\cdot , \cu_{m+1} , p , \ahom_m p )  - \di v(\cdot ,y + \cu_{m} , p , \ahom_m p ) \right\|^2_{\underline{L}^2 \Lambda^r (y + \cu_m)} \right] \leq C \tau_m.
\end{equation*}
In particular, this yields
\begin{equation*}
\E \left[ \left\| \di v(\cdot , \cu_{m+1} , p , \ahom_m p )  - \di v(\cdot , \cu_{m} , p , \ahom_m p ) \right\|^2_{\underline{L}^2 \Lambda^r (\cu_m)} \right] \leq C \tau_m.
\end{equation*}
Combining the previous display with the estimate~\eqref{caccioppoliv0} gives
\begin{align*}
\E \left[ \left\|\di v(\cdot , \cu_{m} , p , \ahom_m p ) \right\|^2_{\underline{L}^2 \Lambda^r (\cu_m)} \right] & \leq C \tau_m + C 3^{-\beta m} + C 3^{ -\beta m}  \sum_{k=0}^m 3^{\beta k} \tau_k \\
& \leq C 3^{-\beta m} + C 3^{ -\beta m}  \sum_{k=0}^m 3^{\beta k} \tau_k.
\end{align*}
Applying the estimate~\eqref{eq:171508}, we deduce
\begin{equation*}
\E \left[  J(\cu_m,p,\ahom_m p) \right] \leq  C 3^{-\beta m} + C 3^{ -\beta m}  \sum_{k=0}^m 3^{\beta k} \tau_k.
\end{equation*}
The proof of the lemma is complete.
\end{proof}

\subsection{Quantitative convergence of the subadditive quantity $J$.} In this section, we show how to deduce Proposition~\ref{maintheorem2} from Lemma~\ref{lemmaJflat}.

\begin{proof}[Proof of Proposition~\ref{maintheorem2}]
We first define the following quantity
\begin{equation*}
D_m = \sum_{i=1}^d \E \left[ J( \cu_m, e_i , \ahom_m e_i ) \right].
\end{equation*}
The reason motivating the introduction of this quantity is twofold:
\begin{itemize}
\item Since, for each integer $m \in \N$, the mapping $p \mapsto \E \left[ J( \cu_m, p , \ahom_m p ) \right]$ is a positive definite quadratic form, we have the estimate
\begin{equation*}
\frac 1d D_m \leq \sup_{p \in B_1\Lambda^{r}(\Rd)} \E \left[ J(\cu_m ,p ,\ahom_m p) \right] \leq D_m.
\end{equation*}
A consequence of the previous inequality is that if we can prove that the sequence $D_m$ converges to $0$ algebraically fast, then the same result is valid for the sequence $ \sup_{p \in B_1\Lambda^{r}(\Rd)} \E \left[ J(\cu_m ,p ,\ahom_m p) \right]$.
\item Second, the quantity $D_m$ satisfies the following estimate, for some constant $c := c(d , \lambda) > 0,$
\begin{equation} \label{Dntaun}
D_m  - D_{m+1} \geq c \tau_m,
\end{equation}
which we now prove. Using the definition of $\ahom_{m+1}$ stated in~\eqref{def.ahomU} and the decomposition of $J$ stated in~\eqref{eq:170808}, we know that for each fixed $p \in \Lambda^{r} (\Rd)$, the quadratic form
\begin{equation*}
q \mapsto \E \left[ J ( \cu_{m+1} , p , q ) \right] = \E \left[ J (\cu_{m+1},p, 0) \right] + \star \left( \frac 12 \ahom_{m+1}^{-1} q \wedge q - p \wedge q \right)
\end{equation*}
attains it minimum at $q = \ahom_{m+1} p.$ Consequently, we have the estimate
\begin{equation*}
D_{m+1} = \sum_{i=1}^d \E \left[ J( \cu_{m+1}, e_i , \ahom_{m+1} e_i ) \right] \leq \sum_{i=1}^d \E \left[ J( \cu_{m+1}, e_i , \ahom_{m} e_i ) \right].
\end{equation*}
Thus we can compute
\begin{align*}
D_m  - D_{m+1} & 
= \sum_{i=1}^d  \left( \E \left[ J( \cu_{m}, e_i , \ahom_{m} e_i ) \right] -  \E \left[ J( \cu_{m+1}, e_i , \ahom_{m+1} e_i ) \right] \right) \\
& \geq \sum_{i=1}^d  \left( \E \left[ J( \cu_{m}, e_i , \ahom_{m} e_i ) \right] -  \E \left[ J( \cu_{m+1}, e_i , \ahom_{m} e_i ) \right] \right) \\
& \geq  \sum_{i=1}^d  \left( \E \left[J ( \cu_{m}, e_i,0) \right] -  \E \left[ J ( \cu_{m+1}, e_i,0 ) \right] \right) \\ &\hspace{5mm}  +  \sum_{i=1}^d  \left( \E \left[ J ( \cu_{m}, 0,\ahom_{m} e_i ) \right] -  \E \left[ J( \cu_{m+1},0, \ahom_{m} e_i ) \right] \right) \\
& \geq c \sup_{p \in B_1 \Lambda^{r}( \Rd) }  \E \left[ J ( \cu_{m}, p,0) \right] -  \E \left[ J ( \cu_{m+1}, p,0) \right] \\
& \hspace{5mm} + c \sup_{q \in B_1 \Lambda^{r}( \Rd) }  \E \left[ J ( \cu_{m},0, q ) \right] -  \E \left[ J ( \cu_{m+1},0, q) \right] \\
& \geq c \tau_m.
\end{align*}
\end{itemize}
We then split the proof of Theorem~\ref{maintheorem} into 4 steps.
\begin{itemize}
\item In Step 1, we let $\beta$ be the exponent which appears in Lemma~\ref{lemmaJflat}, we define the quantity
\begin{equation*}
\tilde{D}_m := 3^{-\frac{\beta m}{2} } \sum_{n = 0}^m 3^{\frac{\beta n}{2}} D_n,
\end{equation*}
and we prove the estimate, for some constants  $\theta(d , \lambda) \in (0,1)$ and $C(d, \lambda) < \infty$,
\begin{equation} \label{dn+1thetadn}
\tilde{D}_{m+1} \leq \theta \tilde{D}_m + C 3^{- \frac{\beta m}{2}}.
\end{equation}
\item In Step 2, we iterate the estimate obtained in Step 1 and prove the inequality, with the value $\alpha = - \frac{\ln \theta}{\ln 3} > 0$,
\begin{equation*}
\tilde{D}_{m} \leq C 3^{- \alpha m}.
\end{equation*}
\item In Step 3, we deduce from Step 2 the estimate
\begin{equation*}
\E \left[ J \left( \cu_m , p , \ahom p \right) \right] \leq C 3^{- \alpha m}.
\end{equation*} 
\item In Step 4, we use the result of Step 3 and a concentration inequality to complete the proof of Theorem~\ref{maintheorem}.
\end{itemize}

\smallskip

\textit{Step 1.}
By the inequality~\eqref{Dntaun} and the bound $D_0 \leq C$, we have
\begin{equation*}
\tilde{D}_{m} - \tilde{D}_{m+1} =  3^{-\frac{\beta m}{2} } \sum_{n = 0}^m 3^{\frac{\beta n}{2}}(  D_{n} - D_{n+1} ) - C  3^{- \frac{\beta m}{2}} \geq  3^{-\frac{\beta m}{2} } \sum_{n = 0}^m 3^{\frac{\beta n}{2}} \tau_n - C  3^{- \frac{\beta m}{2}}.
\end{equation*}
In particular, the previous estimate gives
\begin{equation*}
\tilde{D}_{m+1} \leq \tilde{D}_{m} + C 3^{- \frac{\beta m}{2}}.
\end{equation*}
From the previous inequality and Lemma~\ref{lemmaJflat}, we compute
\begin{align*}
\tilde{D}_{m+1} \leq \tilde{D}_{m} + D_0 3^{- \frac{\beta m}{2}} & = 3^{-\frac{\beta m}{2} } \sum_{n = 0}^m 3^{\frac{\beta n}{2}} D_n  + D_0 3^{- \frac{\beta m}{2}} \\
& \leq C 3^{-\frac{\beta m}{2} } \sum_{n = 0}^m 3^{\frac{\beta n}{2}}  \left( 3^{-\beta n} + 3^{ -\beta n}  \sum_{k=0}^n 3^{\beta k} \tau_k \right) + C 3^{- \frac{\beta m}{2}} \\
& \leq C 3^{-\frac{\beta m}{2} } \sum_{n = 0}^m  \sum_{k=0}^n 3^{- \frac{\beta n}{2}}  3^{\beta k} \tau_k + C 3^{- \frac{\beta m}{2}} \\
& \leq C 3^{-\frac{\beta m}{2} } \sum_{k = 0}^m  \sum_{n=k}^m 3^{- \frac{\beta n}{2}}  3^{\beta k} \tau_k + C 3^{- \frac{\beta m}{2}} \\
& \leq C 3^{-\frac{\beta m}{2} } \sum_{k = 0}^m  3^{\frac{\beta k}{2}} \tau_k + C 3^{- \frac{\beta m}{2}}.
\end{align*}
Combining the two previous displays gives
\begin{equation*}
\tilde{D}_{m+1} \leq C \left( \tilde{D}_{m} - \tilde{D}_{m+1} \right) +  C 3^{- \frac{\beta m}{2}}.
\end{equation*}
A rearrangement of this inequality yields~\eqref{dn+1thetadn}.

\smallskip

\textit{Step 2.} Iterating the estimate~\eqref{dn+1thetadn} gives
\begin{equation*}
 \tilde{D}_{m} \leq \theta^m D_0 + C \sum_{k = 0}^n \theta^k 3^{- \frac{\beta (m -k)}{2}}.
\end{equation*}
Without loss of generality, we can assume $\theta > 3^{- \frac \beta 2}$ (since we can make $\theta$ closer to $1$ if necessary). With this assumption, the second term on the right-hand side can be estimated and we have
\begin{equation*}
 \sum_{k = 0}^n \theta^k 3^{- \frac{\beta (m -k)}{2}} \leq C \theta^m.
\end{equation*}
Combining this with the fact that $\tilde{D}_0 = D_0 \leq C$, we obtain
\begin{equation*}
\tilde{D}_{m} \leq C \theta^m,
\end{equation*}
which can be rewritten, with $\alpha = - \frac{\ln \theta}{\ln 3} > 0$,
\begin{equation*}
\tilde{D}_{m} \leq C 3^{- \alpha m} .
\end{equation*}

\textit{Step 3.} We note that, by the estimate~\eqref{Dntaun} and the posititivity of the sequence $\left(D_m\right)_{m \in \N}$,
\begin{equation*}
c \tau_m \leq D_m - D_{m+1} \leq D_m\leq \tilde{D}_{m}  \leq C 3^{- \alpha m}. 
\end{equation*}
Thus by~\eqref{ahomnahomm}, for every $q \in B_1\Lambda^{d-r}(\Rd)$ and every $m \in \N$,
\begin{align*}
| \ahom^{-1} q - \ahom_m^{-1} q |^2 = \lim\limits_{l \rightarrow \infty} | \ahom_l^{-1} q - \ahom_m^{-1} q |^2  \leq \sum_{k = m}^\infty \tau_k \leq C  \sum_{k = m}^\infty 3^{-\alpha k} \leq C 3^{-\alpha m}.
\end{align*}
Using the ellipticity assumption~\eqref{ellipticityassumption}, we deduce, for each $p \in B_1\Lambda^{d-r}(\Rd)$,
\begin{equation*}
| \ahom p - \ahom_m p |^2  \leq C 3^{-\alpha m}.
\end{equation*}
Using that $J$ is quadratic, one obtains that there exists a constant $C( d , \lambda) < \infty$ such that, for each $m \in \N$, each $p,p' \in \Lambda^{r}(\Rd)$ and each $q, q' \in \Lambda^{d-r}(\Rd)$,
\begin{equation*}
\left| J(\cu_m , p , q) - J(\cu_m , p' , q') \right| \leq C (|p - p'| + |q - q'|) (|p| + |p'| + |q| + |q'|).
\end{equation*}
Consequently, for each $p \in B_1 \Lambda^{r}(\Rd)$ and each $m \in \N$,
\begin{equation*}
\left| J(\cu_m , p , \ahom p) - J(\cu_m , p ,  \ahom_m p) \right|  \leq C  |\ahom p -  \ahom_m p| \left( 1 + |\ahom p| +  |\ahom_n p| \right)  \leq C  3^{-\frac{\alpha}{2} m}.
\end{equation*}
Redefining $\alpha = \frac \alpha 2$ completes the proof of the quantitative homogenization estimate~\eqref{statementmaintheorem2}. 

\smallskip

\textit{Step 4.} We can now complete the proof of Proposition~\ref{maintheorem2} by upgrading the stochastic integrability. We use the following concentration inequality which can be found in~\cite[Lemma 2.14]{armstrong2017quantitative}.

\begin{lemma}
Fix a real number $K > 0$. Suppose that $U \mapsto \rho(U)$ is a (random) map from the set of bounded Lipschitz domains to $[0, K)$ such that $\rho(U)$ is $\mathcal{F}(U)$-measurable
and, whenever $U$ is the disjoint union of $U_1 , \ldots, U_k$ up to a set of zero Lebesgue measure,
\begin{equation*}
\rho(U) \leq \sum_{i=1}^{k} \frac{|U_i|}{|U|} \rho(U_i).
\end{equation*}
Then there exists a universal constant $C < \infty$ such that, for every $m,n \in \N$,
\begin{equation*}
\rho \left( \cu_{n+m+1} \right) \leq 2 \E\left[ \rho \left(\cu_n \right) \right] + \O_1\left( C K 3^{-md}\right).
\end{equation*}
\end{lemma}
Applying this result with
$
\rho(U) := \sup_{p \in B_1} J\left(U,p,\ahom p \right)
$
gives, for each $m,n \in \N$,
\begin{equation*}
\rho(\cu_{n+m+1}) \leq  2 \E\left[ \rho \left(\cu_n \right) \right] + \O_1\left( C 3^{-md}\right) \leq C3^{- n \alpha} + \O_1\left( C 3^{-md}\right).
\end{equation*}
Taking $n=m$ yields, for every $n \in \N$,
\begin{equation*}
\rho(\cu_{2n+1}) \leq C3^{- n \alpha} + \O_1\left( C 3^{-nd}\right).
\end{equation*}
By redefining $\alpha := \min \left( \frac{\alpha}{2}, \frac d2 \right) $, we obtain, for each $n \in \N$,
\begin{equation*}
\rho(\cu_{n}) \leq C3^{- n \alpha} + \O_1\left( C 3^{-n   \alpha}\right) \leq  \O_1 \left( C 3^{-n   \alpha} \right).
\end{equation*}
The proof of Proposition~\ref{maintheorem2} is now complete.
\end{proof}

\subsection{Quantitative convergence of the exterior derivative of the maximizer $v$.} Before turning to the proof of Theorem~\ref{homogenizationtheorem} in the next section, we state and prove the following proposition, which is a consequence of Theorem~\ref{maintheorem}, and gives quantitative information about the flatness of the maximizers. Let us first recall that the form $ l_{p}^{m}$ is defined as the unique element of $C^{r-1}_\di \left( \cu_m \right)^\perp$ which satisfies $\di  l_{p}^{m} = p$. 

\begin{proposition} \label{weakcvggradientandflux}
There exist an exponent $\alpha := \alpha(d , \lambda) > 0$ and a constant $C := C(d, \lambda ) < \infty$ such that for each integer $r \in \{1 , \ldots, d \}$, each pair of forms $\left( p,q \right) \in B_1 \Lambda^r (\Rd) \times  B_1 \Lambda^{d-r} (\Rd)$ and each integer $m \in \N$,
\begin{multline} \label{gradientfluxcvg}
3^{-m} \left\|\di v \left( \cdot, \cu_m, p,q \right) - \left(\ahom^{-1} q - p \right) \right\|_{\underline{H}^{-1} \Lambda^r (\cu)} + 3^{-m} \left\|\a \di v \left( \cdot, \cu_m, p,q \right) - \left( q - \ahom p \right) \right\|_{\underline{H}^{-1} \Lambda^{d-r} (\cu)} \\ \leq \O_2 \left( C 3^{-m\alpha} \right),
\end{multline}
and
\begin{equation} \label{gradientfluxcvg2}
 \left\| v( \cdot , \cu_{m}, p ,q ) - l_{\ahom^{-1} q - p}^{m} \right\|_{\underline{L}^2 \Lambda^{r-1}(\cu_{m+1})} \leq \O_2 \left( C 3^{-\alpha m} \right).
\end{equation}

\end{proposition}

\begin{proof}
We first prove the estimate~\eqref{gradientfluxcvg}. The proof is split into 2 steps.
\begin{itemize}
\item \textit{Step 1.} We prove that, for each $q \in B_1 \Lambda^{d-r} (\Rd)$ and every $m,n \in \N$ such that $m \geq n$,
\begin{equation} \label{quantdvmeanvalue}
3^{d(n-m)}\sum_{y \in \mathcal{Z}_{n,m}} \left| \left( \di v \left(\cdot , \cu_m , 0 , q\right) - \ahom^{-1} q   \right)_{y + \cu_n} \right|^2 \leq \O_1 \left(C 3^{-\alpha n} \right).
\end{equation}
Similarly, 
for each $p \in B_1 \Lambda^{r} (\Rd)$ and every $m,n \in \N$ such that $m \geq n$,
\begin{equation} \label{quantdvmeanvaluebis}
3^{d(n-m)} \sum_{y \in \mathcal{Z}_{n,m}} \left| \left( \a \di v \left(\cdot , \cu_m , p , 0\right) + \ahom p  \right)_{y + \cu_n} \right|^2 \leq \O_1 \left( C 3^{-\alpha n} \right).
\end{equation}
\item \textit{Step 2.} We deduce the estimate~\eqref{gradientfluxcvg} from the previous step and the multiscale Poincar\'e inequality.
\end{itemize}
\medskip
\textit{Step 1.} We first deal with the case $m=n$. In this specific case, the estimate~\eqref{quantdvmeanvalue} reads
\begin{equation*}
 \left| \left( \di v \left(\cdot , \cu_n , 0 , q\right) - \ahom^{-1} q   \right)_{ \cu_n} \right|^2 \leq  \O_1 \left( C 3^{-\alpha n} \right).
\end{equation*}
To prove this inequality, we note that, by the first variation for the energy $J$,
\begin{equation*}
J(\cu_n, 0 , q) = \frac 12 \left( \di v(\cdot, \cu_n , 0 , q) \right)_{\cu_n} \wedge q.
\end{equation*}
Moreover, the map $q \mapsto \left( \di v(\cdot, \cu_m , 0 , q) \right)_{\cu_m}$ is linear and symmetric since, for each $q,q' \in \Lambda^{d-r} (\Rd)$,
\begin{align*}
\left( \di v \left( \cdot, \cu_m , 0 , q' \right) \right)_{\cu_m} \wedge q & = \int_{\cu_m} \di v \left( \cdot, \cu_m , 0 , q' \right)  \wedge \a \di v \left( \cdot, \cu_m , 0 , q \right) \\ & = \fint_{\cu_m} \di v \left( \cdot, \cu_m , 0 , q \right)  \wedge \a \di v \left( \cdot, \cu_m , 0 , q' \right) \\ & = \left( \di v \left( \cdot, \cu_m , 0 , q \right) \right)_{\cu_m} \wedge q'.
\end{align*}
A combination of the two previous arguments and Theorem~\ref{maintheorem} gives
\begin{align} \label{quantmeandvspecialcase}
\sup_{q \in B_1 \Lambda^{d-r}}   \left| \left( \di v \left(\cdot , \cu_n , 0 , q\right) - \ahom^{-1} q   \right)_{ \cu_n} \right|^2 & \leq C \sup_{q \in B_1 \Lambda^{d-r}} \left| J(\cu_n, 0 , q)  - \ahom^{-1} q \wedge q \right|^2 \\ & \leq \O_1 \left( C 3^{-n \alpha} \right). \notag
\end{align}
To obtain the general case $m \geq n$ from the specific case $m=n$, we compute
\begin{align*}
3^{d(n-m)}\lefteqn{ \sum_{y \in \mathcal{Z}_{n,m}} \left| \left(\di v \left(\cdot , \cu_m , 0,q\right) - \ahom^{-1} q  \right)_{y + \cu_n} \right|^2} \qquad &\\ & \leq C3^{d(n-m)} \sum_{y \in \mathcal{Z}_{n,m}} \left| \left( \di v \left(\cdot , \cu_m , 0,q\right) - \di v \left(\cdot , y + \cu_n , 0,q\right)  \right)_{y + \cu_n} \right|^2  \\ & \quad + C 3^{d(n-m)} \sum_{y \in \mathcal{Z}_{n,m}} \left| \left( \di v \left(\cdot , y + \cu_n , 0,q\right) - \ahom^{-1} q  \right)_{y + \cu_n} \right|^2 \\
& \\ & \leq C3^{d(n-m)} \sum_{y \in \mathcal{Z}_{n,m}} \left\| \di v \left(\cdot , \cu_m , 0,q\right) - \di v \left(\cdot , y + \cu_n , 0,q\right)  \right\|_{\underline{L}^2 \left(y + \cu_n \right) }^2  \\ & \quad + C 3^{d(n-m)} \sum_{y \in \mathcal{Z}_{n,m}} \left| \left( \di v \left(\cdot , y + \cu_n , 0,q\right) - \ahom^{-1} q  \right)_{y + \cu_n} \right|^2 \\ &
\leq C3^{d(n-m)} \sum_{y \in \mathcal{Z}_{n,m}} J( y + \cu_n, 0,q) -   J( \cu_m, 0,q) \\ & \quad + C 3^{d(n-m)} \sum_{y \in \mathcal{Z}_{n,m}} \left| \left( \di v \left(\cdot , y + \cu_n , 0,q\right) - \ahom^{-1} q  \right)_{y + \cu_n} \right|^2.
\end{align*}
To deal with the first term on the right-hand side, we note that, for each $y \in \mathcal{Z}_{m,n}$,
\begin{align*}
 J( y + \cu_n, 0,q) -   J( \cu_m, 0,q) &  \leq \left|  J( y + \cu_n, 0,q) -  \ahom^{-1} q \wedge q \right| + \left|  J( \cu_m, 0,q) -  \ahom^{-1} q \wedge q \right| \\ & \leq \O_2( C 3^{-\alpha n}) + \O_2( C 3^{-\alpha m}) \\& \leq \O_2 (C 3^{-\alpha n} ),
\end{align*}
by the stationarity assumption~\eqref{stationarityassumption}. Using the inequality~\eqref{e.0mean}, we obtain
\begin{equation*}
3^{d(n-m)} \sum_{y \in \mathcal{Z}_{n,m}} J( y + \cu_n, 0,q) -   J( \cu_m, 0,q)  \leq \O_2(C 3^{-n \alpha}).
\end{equation*}
To treat the second term on the right-hand side, we have by the stationarity assumption~\eqref{stationarityassumption} and the estimate~\eqref{quantmeandvspecialcase}, for each $y \in \mathcal{Z}_{m,n}$,
\begin{equation*}
\left| \left( \di v \left(\cdot , y + \cu_n , 0,q\right) - \ahom^{-1} q  \right)_{y + \cu_n} \right|^2 \leq \O_1 \left( C 3^{-\alpha n} \right).
\end{equation*}
Using the inequality~\eqref{e.0mean}, we obtain
\begin{equation*}
 3^{d(n-m)} \sum_{y \in \mathcal{Z}_{n,m}} \left| \left( \di v \left(\cdot , y + \cu_n , 0,q\right) - \ahom^{-1} q  \right)_{y + \cu_n} \right|^2 \leq \O_1 \left(C 3^{-\alpha n} \right).
\end{equation*}
The proof of~\eqref{quantdvmeanvalue} is thus complete. The proof of~\eqref{quantdvmeanvaluebis} is similar, the details are left to the reader.

\medskip

\textit{Step 2.} From Step 1 and (ii) of Proposition~\ref{propertiesbigO}, one has
\begin{equation*}
\sum_{n = 0}^{m-1} 3^{d(n-m)} \sum_{y \in \mathcal{Z}_{n,m}} \left| \left( \di v \left(\cdot , y + \cu_n , 0,q\right) - \ahom^{-1} q  \right)_{y + \cu_n} \right|^2  \leq  \O_1 \left(C 3^{(1-\alpha) m} \right)
\end{equation*}
and
\begin{equation*}
\sum_{n = 0}^{m-1} 3^{d(n-m)} \sum_{y \in \mathcal{Z}_{n,m}} \left| \left( \a \di v \left(\cdot , y + \cu_n , p,0\right) - \ahom p  \right)_{y + \cu_n} \right|^2  \leq  \O_1 \left(C 3^{(1-\alpha) m} \right).
\end{equation*}
By the multiscale Poincar\'e inequality stated in Proposition~\ref{multscalepoincare}, the bound on the $L^2$-norm of the exterior derivative $\di v$ stated in~\eqref{eq:171508}, and the previous estimates, one obtains for each $(p,q) \in B_1 \Lambda^r (\Rd) \times  B_1 \Lambda^{d-r} (\Rd)$,
\begin{equation*}
\left\|\di v \left( \cdot, \cu_m, p,q \right) - \left(\ahom^{-1} q - p \right) \right\|_{\underline{H}^{-1} \Lambda^r (\cu)} + \left\|\a \di v \left( \cdot, \cu_m, p,q \right) - \left( q - \ahom p \right) \right\|_{\underline{H}^{-1} \Lambda^{d-r} (\cu)} \\ \leq \O_2 \left( C 3^{(1-\alpha)m} \right).
\end{equation*}
Dividing both sides of the previous inequality by $3^m$ yields the inequality~\eqref{gradientfluxcvg}.

\smallskip

We then deduce the estimate~\eqref{gradientfluxcvg2} from the estimate~\eqref{gradientfluxcvg}. We first note that the form $ v( \cdot , \cu_{m}, p ,q ) - l_{\ahom^{-1} q - p}^{m}$ belongs to the space $C^{r-1}_\di \left( \cu_m \right)^\perp$. Applying Proposition~\ref{Poincareineq} and the estimate~\eqref{gradientfluxcvg} implies the estimate~\eqref{gradientfluxcvg2}. The proof of Proposition~\ref{weakcvggradientandflux} is complete.
\end{proof}

\section{Homogenization of the Dirichlet problem} \label{section6}

The goal of this section is to study the Dirichlet problem for the equation $\di \a \di u = 0$ and to establish Theorem~\ref{homogenizationtheorem}. We first prove existence and uniqueness of solution for this equation.

\begin{proposition} \label{exaunell.prop}
Let $U$ be a bounded smooth domain of $\Rd$ and $r$ be an integer in $\{1 , \ldots, d\}$. Let $f $ be an element of the space $H^1_\di \Lambda^{r-1} (U)$, then for any measurable map $ \a : \Rd \rightarrow \mathcal{L} \left( \Lambda^{r}(\Rd) , \Lambda^{(d-r)} \left( \Rd \right) \right)$ satisfying the assumptions~\eqref{ellipticityassumption} and~\eqref{symmetryassumption}, there exists a unique solution in $f + H^1_{\di , 0} \Lambda^{r-1} (U) \cap \left( C^{r-1}_{\di,0}(U) \right)^\perp$ of the equation
\begin{equation} 
\label{exaunell.2}
 \left\{ \begin{aligned}[l]
 	\di \left( \a \di u \right) = 0 & \hspace{3mm} \mbox{in} ~U, \\
       \mathbf{t} u = \mathbf{t} f & \hspace{3mm} \mbox{on} ~\partial U,
    \end{aligned}\right.
\end{equation}
in the sense that, for each form $v \in H^1_{\di , 0} \Lambda^{r-1} (U)$,
$
\int_U \di u \wedge \a \di v = 0.
$
The solution satisfies the estimate, for some constant $C := C(d , \lambda , U ) < \infty$,
\begin{equation*}
\left\| u \right\|_{H^1_\di\Lambda^{r-1}(U)} \leq C \left\| \di f \right\|_{L^2\Lambda^{r}(U)}.
\end{equation*}
Moreover if we enlarge the set of admissible solutions to the space $f + H^1_{\di , 0} \Lambda^{r-1} (U)$, one looses the uniqueness property, but if $v , w \in f + H^1_{\di , 0} \Lambda^{r-1} (U) $ are two solutions of~\eqref{exaunell.2}, then the difference $v - w$ belongs to the space $C^{r-1}_{\di , 0}(U).$
\end{proposition}

\begin{proof}
The existence and uniqueness of such a solution are obtained by minimizing the quantity
$
\mathcal{J} (v) := \left\langle \di f  + \di v , \di f +  \di v\right\rangle_U 
$
on the space $H^1_{\di , 0} \Lambda^{r -1} (U) \cap \left( C^r_{\di , 0}(U) \right)^\perp$ and requires to use the Poincar\'e inequality stated in Proposition~\ref{PoincareineqH10conv}. The techniques are standard, and we omit the details.
\end{proof}

The rest of this section is devoted to the proof of Theorem~\ref{homogenizationtheorem}.

\begin{proof}[Proof of Theorem~\ref{homogenizationtheorem}]
Without loss of generality, one can assume that the volume of the domain $U$ is equal to $1$. We fix an integer $l> 0$, such that the $0 \ll 3^l \ll \ep^{-1} $. The value $3^l \ep$ represents the thickness of a boundary layer we need to remove in the argument, it is chosen at the end of the proof and depends only on~$\epsilon$. For a radius $r > 0$, denote by $U_r := \left\{ x \in U \,:\, \dist \left( x , \partial U \right) > r \right\}$. For a subset $I \subseteq \left\{ 1 , \ldots , d \right\}$ of cardinality $r$, we recall the notation $\di x_I$ introduced in~\eqref{notationconventiondxI}. We denote by $ \phi_{m,I}$ the corrector in the cube $\cu_m$, which is defined by the following formula
\begin{equation*}
\phi_{m,I} := - v( \cdot , \cu_{m}, \di x_I) - l_{\di x_I}^{m}.
\end{equation*}
By Proposition~\ref{weakcvggradientandflux}, the corrector satisfies the following estimates
\begin{multline} \label{propphi1}
3^{-m}\left\| \phi_{m,I} \right\|_{\underline{L}^2 \Lambda^{r-1}(\cu_{m+1})}  + 3^{-m} \left\|\di \phi_{m,I} \right\|_{\underline{H}^{-1} \Lambda^r (\cu)} + 3^{-m} \left\|\a \left( \di x_I + \di\phi_{m,I} \right) - \ahom \di x_I\right\|_{\underline{H}^{-1} \Lambda^{d-r} (\cu)} \\ \leq \O_2 \left( C 3^{-m\alpha} \right).
\end{multline}
Let $m$ be the smallest integer such that
\begin{equation} \label{eq:defintegerm}
U \subseteq \varepsilon \cu_m,
\end{equation}
this implies that there exists a constant $C:= C(U) < \infty$ such that $\ep 3^m \leq C$. We then define the two-scale expansion, with the notation convention $\di u : =  \sum_{|I| = r } (\di u)_I  \di x_I$,
\begin{equation} \label{ad0}
w^\varepsilon_0 (x ) := u(x) + \varepsilon \zeta_l(x) \sum_{|I| = r } (\di u)_I (x) \phi_{m,I} \left( \frac{x}{\varepsilon} \right),
\end{equation}
where $\zeta_l \in C^\infty_c (U)$ is a smooth cutoff function satisfying, for every integer $k \in \N$,
\begin{equation} \label{definitionzetar} 
0 \leq \zeta_l \leq 1, \hspace{3mm} \zeta_l = 1 \mbox{ in } U_{2l}, \hspace{3mm}  \zeta_l = 0 \mbox{ in } U \setminus U_{l}, \hspace{3mm}  \left| \nabla^k \zeta_l \right| \leq C(k,d,U) l^{-k}.
\end{equation}
Note that $w^\varepsilon_0$ belongs to the space $f +H^1_{\di , 0} \Lambda^{r-1} (U) $. Since it is more convenient to work with an element of $f +H^1_{\di , 0} \Lambda^{r-1} (U) \cap \left( C^{r-1}_{\di,0}(U) \right)^\perp$ (to have the Poincar\'e inequality), we further define
\begin{equation*}
w^\varepsilon := f + \mathrm{Proj}_{\left( C^{r-1}_{\di,0}(U) \right)^\perp} \left(w^\varepsilon_0 - f  \right).
\end{equation*}
where $ \mathrm{Proj}_{\left( C^{r-1}_{\di,0}(U) \right)^\perp}$ denotes the $L^2$-orthogonal projection on the space $\left( C^{r-1}_{\di,0}(U) \right)^\perp$.
Note that $w^\varepsilon \in f +H^1_{\di , 0} \Lambda^{r-1} (U) \cap \left( C^{r-1}_{\di,0}(U) \right)^\perp$ by construction and that it satisfies
\begin{equation} \label{ad1}
\di w^\varepsilon_0 = \di w^\varepsilon ~\mbox{in} ~U.
\end{equation}
We then consider the map
\begin{equation*}
\di \left( \a \left( \frac{\cdot }{\varepsilon} \right) \di w^\varepsilon \right) : \left\{ \begin{aligned}
H^1_{\di , 0} \Lambda^{r-1} (U)  \rightarrow \R, \hspace{10mm}\\
v \rightarrow \int_{U} \di w^\varepsilon  \wedge \a \left( \frac{x}{\varepsilon} \right)  \di v .\\
\end{aligned} \right.
\end{equation*}
which belongs to the space $H^{-1}_{\di} \Lambda^{r-1}$.
The strategy of the proof is to compare $u^\varepsilon$ to the function $w^\varepsilon$. The proof is split into 6 steps which are summarized below:

\begin{itemize}
\item In Step 1, we prove that there exists a constant $C := C(U) < \infty$ such that
\begin{equation*}
\left\| \di \left( \a \left( \frac{\cdot }{\varepsilon} \right) \di w^\varepsilon \right) \right\|_{H^{-1}_{\di} \Lambda^{d - r+1}(U)}  \leq C \left\| \di \left( \a \left( \frac{\cdot }{\varepsilon} \right) \di w^\varepsilon \right) \right\|_{H^{-1} \Lambda^{d - r+1}(U)}.
\end{equation*}
\item In Step 2, we prove the estimate on the two-scale expansion $w^\ep$,
\begin{equation*}
\left\| \di \left( \a \left( \frac{\cdot }{\varepsilon} \right) \di w^\varepsilon \right) \right\|_{H^{-1} \Lambda^{r-1} (U)} \leq \left\{ \begin{aligned}
C \left\| \di f \right\|_{H^1 \Lambda^r (U)} \left( l^{\frac{1}{d-2}} + \O_2 \left(\frac{ \varepsilon^{\alpha}}{l^{3 + d/2}} \right) \right) & ~ \mbox{if} ~d \geq 3,\\
C \left\| \di f \right\|_{H^1 \Lambda^r (U)} \left( l^{\frac 14} + \O_2 \left( \frac{ \varepsilon^{\alpha}}{l^{3 + d/2}}  \right) \right) \hspace{3mm} & ~ \mbox{if} ~d =2 .\\
\end{aligned} \right.
\end{equation*}
\item In Step 3, we deduce from Steps 1 and 2 the estimate
\begin{equation*}
\left\| u^\varepsilon - w^\varepsilon \right\|_{H^1_\di \Lambda^{r-1} (U)} \leq \left\{ \begin{aligned}
C \left\| \di f \right\|_{H^1 \Lambda^r (U)} \left( l^{\frac{1}{d-2}} + \O_2 \left(\frac{ \varepsilon^{\alpha}}{l^{3 + d/2}} \right) \right) & ~ \mbox{if} ~d \geq 3,\\
C \left\| \di f \right\|_{H^1 \Lambda^r (U)} \left( l^{\frac 14} + \O_2 \left( \frac{ \varepsilon^{\alpha}}{l^{3 + d/2}}  \right) \right) \hspace{3mm} & ~ \mbox{if} ~d =2 .\\
\end{aligned} \right.
\end{equation*}
\item In Step 4, we prove that for each form $w \in H^1_{\di , 0} \Lambda^{r -1} (U) \cap \left( C^{r-1}_{\di , 0}(U) \right)^\perp$, one has the estimate
\begin{equation*}
\|  w \|_{L^2 \Lambda^{r-1} (U)} \leq \| \di w \|_{\underline{H}^{-1} \Lambda^r (U)}.
\end{equation*}
\item In Step 5, we deduce from Step 4 and the estimate~\eqref{propphi1} the inequality
\begin{equation*}
\left\| \di w^\varepsilon - \di u \right\|_{\underline{H}^{-1} \Lambda^r(U)} \leq  \O_2 \left( \frac{C\left\| \di f \right\|_{H^1\Lambda^r (U)}\varepsilon^\alpha}{l^{2 + d/2}}  \right).
\end{equation*}
\item In Step 6, we combine the results of Steps 3 and 5 and set the value  $l := \ep^{\frac{\alpha (3 + d/2)}{2}}$ to obtain the estimate
\begin{equation*}
\left\| u^{\varepsilon} - u \right\|_{L^2\Lambda^r(U)} + \left\| \di u^{\varepsilon} - \di u \right\|_{\underline{H}^{-1}\Lambda^r(U)} \leq \O_2 \left( C \ep^\alpha \right),
\end{equation*}
by reducing the size of the exponent $\alpha$ in the right side.
\end{itemize}

\smallskip

\textit{Step 1.} The main result of this step is a consequence of the following property. There exists a constant $C := C(d , \lambda, U) < \infty$, such that for each form $v \in H^1_{\di, 0} \Lambda^{r-1}(U)$, there exists a form $w \in H^1_0 \Lambda^{r-1}(U)$ such that
\begin{equation*}
\di w = \di v~ \mbox{and} ~ \left\| w \right\|_{H^1 \Lambda^r(U)} \leq C  \left\| v \right\|_{H^1_\di \Lambda^r(U)}.
\end{equation*}
To prove this, we follow the arguments of the proof of~\cite[Theorem 1.1]{MMM08}. Let $\left( O_j \right)_{1 \leq j \leq N}$ be a finite, open covering of $\bar{U}$ such that $O_j \cap U$ is a Lipschitz bounded star-shaped domain. Then let $\left( \chi_j \right)_{1 \leq j \leq N}$ be a smooth partition of unity such that $\supp \chi_j \subseteq O_j$ for each integer $j \in \{1 , \ldots , d \}$. Note that the form $\chi_j v $ belongs to the space $ H^{1}_{\di , 0} \Lambda^{r-1}\left(O_j \cap U\right) $. By Proposition~\ref{Poincarelemma}, there exists a form $w_j \in H^1_0 \Lambda^{r-1} \left(O_j \cap U\right)$ satisfying 
$
\di w_j = \chi_j v$ and $\left\| w_j  \right\|_{H^{1} \Lambda^{r-1}\left(O_j \cap U\right)} \leq C\left\| \chi_j v \right\|_{H^{1}_{\di} \Lambda^{r-1}\left(O_j \cap U\right)} .
$
We then extend the forms $ \chi_j v$ and $w_j$ by $0$ outside the set $O_j \cap U$, so that they belong to the spaces $H^{1}_{\di} \Lambda^{r-1}\left(\Rd\right) $ and $H^{1} \Lambda^{r-1}\left(\Rd\right)$ respectively and satisfy
$
\di w_j = \di \left( \chi_j v \right)$ in $\Rd.
$

We then define
$
w := \sum_{j=1}^{N}  w_j,
$
so that
$
w \in H^1_0 \Lambda^{r-1}\left( U\right) ~\mbox{and} ~ \di w = \sum_{j=1}^N \di \left( \chi_j v \right) = \di v. 
$
We also have the estimate
\begin{equation*}
\left\| w  \right\|_{H^1 \Lambda^{r-1}\left( U\right)} \leq C \left\| v  \right\|_{H^1_{\di} \Lambda^{r-1}\left( U\right)} .
\end{equation*}
This completes the proof of Step 1.

\medskip

\textit{Step 2.} We show the $H^{-1} \Lambda^{r-1} (U)$-estimate
\begin{equation} \label{homogthmstep1}
\left\| \di \left( \a \left( \frac{\cdot }{\varepsilon} \right) \di w^\varepsilon \right) \right\|_{H^{-1} \Lambda^{r-1} (U)} \leq \left\{ \begin{aligned}
C \left\| \di f \right\|_{H^1 \Lambda^r (U)} \left( l^{\frac{1}{d-2}} + \O_2 \left(\frac{ \varepsilon^{\alpha}}{l^{3 + d/2}} \right) \right) & ~ \mbox{if} ~d \geq 3,\\
C \left\| \di f \right\|_{H^1 \Lambda^r (U)} \left( l^{\frac 14} + \O_2 \left( \frac{ \varepsilon^{\alpha}}{l^{3 + d/2}}  \right) \right) \hspace{3mm} & ~ \mbox{if} ~d =2 .\\
\end{aligned} \right.
\end{equation}
We first compute the exterior derivative of $w^\varepsilon$, by~\eqref{ad0} and~\eqref{ad1},
\begin{align*}
\di w^\varepsilon & = \di u +  \zeta_l \sum_{|I| = r } (\di u)_I  \di \phi_{m,I} \left( \frac{\cdot }{\varepsilon} \right) + \varepsilon  \sum_{|I| = r } \di \left( \zeta_l  (\di u)_I \right) \wedge \phi_{m,I} \left( \frac{\cdot }{\varepsilon} \right)  \\
			& =  (1 - \zeta_l) \di u+ \sum_{|I| = r } \zeta_l  (\di u)_I  \left( \di x_I + \di \phi_{m,I} \left( \frac{\cdot }{\varepsilon} \right) \right) +  \varepsilon  \sum_{|I| = r } \di \left( \zeta_l  (\di u)_I \right) \phi_{m,I} \left( \frac{\cdot }{\varepsilon} \right).
\end{align*}
From this we deduce, in the weak sense,
\begin{multline*}
\di \left( \a\left( \frac{\cdot}{\varepsilon} \right) \di w^\varepsilon \right) = \di \left( \a  \left( \frac{\cdot }{\varepsilon} \right) \left(  (1 - \zeta_l) \di u + \varepsilon  \sum_{|I| = r } \di \left( \zeta_l  (\di u)_I \right) \phi_{m,I} \left( \frac{\cdot }{\varepsilon} \right) \right)\right)  \\+  \sum_{|I| = r } \di \left( \zeta_l  (\di u)_I \right) \wedge \a \left( \frac{\cdot }{\varepsilon} \right)  \left( \di x_I + \di \phi_{m,I} \left( \frac{\cdot }{\varepsilon} \right) \right).
\end{multline*}
Since $u$ satisfies the equation $\di ( \ahom \di u ) = 0$, one sees that
\begin{equation*}
\sum_{|I| = r}\di \left( \zeta_l \left( \di u \right)_I \right) \wedge \ahom \di x_I = \di \left( \zeta_l \ahom \di u \right) = - \di \left( ( 1- \zeta_l ) \ahom \di u \right).
\end{equation*}
Consequently one has the equality, in the weak sense,
\begin{multline*}
\di \left( \a\left( \frac{\cdot}{\varepsilon} \right) \di w^\varepsilon \right) = \di \left( \left( \a  \left( \frac{\cdot }{\varepsilon} \right) - \ahom \right) (1 - \zeta_l) \di u \right)+  \varepsilon  \sum_{|I| = r } \di \left( \a  \left( \frac{\cdot }{\varepsilon} \right) \di \left( \zeta_l  (\di u)_I \right) \wedge \phi_{m,I} \left( \frac{\cdot }{\varepsilon} \right) \right)  \\+  \sum_{|I| = r } \di \left( \zeta_l  (\di u)_I \right) \wedge \left( \a \left( \frac{\cdot }{\varepsilon} \right)  \left( \di x_I + \di \phi_{m,I} \left( \frac{\cdot }{\varepsilon} \right) \right) - \ahom \di x_I \right).
\end{multline*}
It follows that
\begin{align*}
\lefteqn{ \left\| \di \left( \a\left( \frac{\cdot}{\varepsilon} \right) \di w^\varepsilon \right)  \right\|_{H^{-1} \Lambda^{r-1} (U)} } & \\ &
			\leq \sum_{|I| = r} \left\| \di \left( \zeta_l  (\di u)_I \right) \right\|_{W^{1,\infty}(U)} \left\| \a \left( \frac{\cdot }{\varepsilon} \right)  \left( \di x_I + \di \phi_{m,I} \left( \frac{\cdot }{\varepsilon} \right) \right) - \ahom \di x_I \right\|_{H^{-1} \Lambda^{r} (U)} \\ &
 			\quad + \left\| \left( \a  \left( \frac{\cdot }{\varepsilon} \right) - \ahom \right) (1 - \zeta_l) \di u \right\|_{L^2\Lambda^r(U)} \\ &
			\quad + \varepsilon  \sum_{|I| = r } \left\|  \a  \left( \frac{\cdot }{\varepsilon} \right) \di \left( \zeta_l  (\di u)_I \right) \wedge \phi_{m,I} \left( \frac{\cdot }{\varepsilon} \right) \right\|_{L^2\Lambda^r(U)} \\ & 
			=: T_1 + T_2 + T_3.
\end{align*}
To bound the term $T_1$ on the right-hand side, we first appeal to the interior regularity estimate stated in Proposition~\ref{propellipticregularity} and the assumption~\eqref{definitionzetar} on the function $\zeta_l$, so that one has
\begin{equation} \label{ad3}
 \left\| \di \left( \zeta_l  (\di u)_I \right) \right\|_{W^{1,\infty}(U)}  \leq \frac{C}{l^{3+d/2}} \|\di f\|_{L^2 \Lambda^r (U)}.
\end{equation}
Then, by Proposition~\ref{weakcvggradientandflux}, the continuity of the inclusion $\underline{H}^{-1} \Lambda^{r} (U) \subseteq H^{-1} \Lambda^{r} (U) $, one has the estimate
\begin{equation*}
\left\| \a \left( \frac{\cdot }{\varepsilon} \right)  \left( \di x_I + \di \phi_{m,I} \left( \frac{\cdot }{\varepsilon} \right) \right) - \ahom \di x_I \right\|_{H^{-1} \Lambda^{r} (U)} \leq \O_2 \left( C 3^{-\alpha m} \right).
\end{equation*}
We then note that, by the definition of the integer $m$ given in~\eqref{eq:defintegerm}, we have the estimate $\ep 3^m \leq C$. This implies
\begin{equation*}
\left\| \a \left( \frac{\cdot }{\varepsilon} \right)  \left( \di x_I + \di \phi_{m,I} \left( \frac{\cdot }{\varepsilon} \right) \right) - \ahom \di x_I \right\|_{H^{-1} \Lambda^{r} (U)} \leq \O_2 \left( C \ep^{-\alpha} \right).
\end{equation*}
Combining the previous display and the estimate~\eqref{ad3}, we obtain the estimate for the term $T_1$
\begin{equation*}
 T_1  \leq  \frac{C}{l^{3+d/2}} \| \di f\|_{H^1(U)} \O_2 \left( \varepsilon^{ \alpha} \right).
\end{equation*}
The bound for $T_3$ is similar, by the estimate~\eqref{propphi1} and the Poincar\'e inequality stated in Proposition~\ref{PoincareineqH10conv}, one has
\begin{equation*}
\varepsilon \left\|   \phi_{m,I} \left( \frac{\cdot }{\varepsilon} \right) \right\|_{L^2\Lambda^r(U)}  \leq \O_2 \left( C \varepsilon^{ \alpha} \right).
\end{equation*}
By the estimate~\eqref{ad3}, one deduces
\begin{equation*}
T_3  \leq  \frac{C}{l^{3+d/2}} \| \di f\|_{H^1(U)} \O_2 \left( \varepsilon^{ \alpha} \right).
\end{equation*}
To estimate the second term $T_2$, the strategy is to apply the boundary regularity result proved in the appendix, Proposition~\ref{globalest}. Since the form $\di f$ is assumed to be in the space $H^1 \Lambda^{r} (U)$ and the set $U$ is assumed to be smooth, one has
\begin{equation*}
\left\| \di u \right\|_{H^1 \Lambda^r (U)} \leq C \left\| \di \ahom \di f \right\|_{L^2 \Lambda^{d -r+1} (U)} \leq  C \left\| \di f \right\|_{H^1 \Lambda^{r} (U)}.
\end{equation*}
This implies, via the Sobolev imbedding theorem, that the exterior derivative $\di u$ belongs to the space $ L^{\frac{2d}{d-2}} \Lambda^{r} (U)$ if $d \geq 3$ and to the space $\bigcap_{p\geq 1} L^p \Lambda^{r} (U)$ if $d = 2$, with the estimates
\begin{equation} \label{eq:est2507}
\left\{ \begin{aligned}
\left\| \di u \right\|_{ L^{\frac{2d}{d-2}} \Lambda^r (U)} & \leq C& \left\| \di f \right\|_{H^1 \Lambda^{r} (U)} &\hspace{5mm} \mbox{if} ~~ d \geq 3,\\
\left\| \di u \right\|_{ L^{p} \Lambda^r (U)} \hspace{3.7mm}& \leq C_p& \left\| \di f \right\|_{H^1 \Lambda^{r} (U)}& \hspace{5mm}  \mbox{if} ~~ d =2,\\
\end{aligned} \right.
\end{equation}
for some constants $C:= C(U ) < \infty$ and $C_p := C(p , U ) < \infty$. We now set $p = 4$ (but any exponent $p$ strictly larger than $2$ would be sufficient). Using the estimate~\eqref{eq:est2507} and the fact that the function $(1 - \zeta_l)$ is supported in the set $U \setminus U_{2l}$, gives, by the H\"older inequality,
\begin{equation*}
T_2 \leq C \left\| \di u \right\|_{L^2\Lambda^r\left(U \setminus U_{2l} \right) } \leq \left\{ \begin{aligned}
 C \left| U \setminus U_{2l} \right|^{\frac 1{d-2}}  \left\| \di u \right\|_{L^{\frac{2d}{d-2}} \Lambda^r \left(U \right)} & ~ \mbox{if} ~d \geq 3,\\
 C \left| U \setminus U_{2l} \right|^{\frac 1{4}}  \left\| \di u \right\|_{L^4  \Lambda^r \left(U \right)} \hspace{6.4mm} & ~ \mbox{if} ~d =2 .\\
\end{aligned} \right.
\end{equation*}
Combining the few previous results completes the proof of the estimate~\eqref{homogthmstep1}.

\medskip

\textit{Step 3.} We deduce from Steps 1 and 2 the following $H^1 \Lambda^{r-1} (U)$-estimate
\begin{equation} \label{homogthmstep2}
\left\| u^\varepsilon - w^\varepsilon \right\|_{H^1_\di \Lambda^{r-1} (U)} \leq \left\{ \begin{aligned}
C \left\| \di f \right\|_{H^1 \Lambda^r (U)} \left( l^{\frac{1}{d-2}} + \O_2 \left(\frac{ \varepsilon^{\alpha}}{l^{3 + d/2}} \right) \right) & ~ \mbox{if} ~d \geq 3,\\
C \left\| \di f \right\|_{H^1 \Lambda^r (U)} \left( l^{\frac 14} + \O_2 \left( \frac{ \varepsilon^{\alpha}}{l^{3 + d/2}}  \right) \right) \hspace{3mm} & ~ \mbox{if} ~d =2 .\\
\end{aligned} \right.
\end{equation}
Testing the estimate~\eqref{homogthmstep1} with the form $u^\varepsilon - w^\varepsilon \in H^1_{\di, 0}\Lambda^{r-1} (U)$, and using the main result of Step 1, one obtains
\begin{align*}
\left| \int_U \di \left( u^\varepsilon - w^\varepsilon \right) (x) \wedge \a \left( \frac{x}{\varepsilon} \right) \di w^\varepsilon (x) \right| & \leq \left\| u^\varepsilon - w^\varepsilon \right\|_{H^1_{\di} \Lambda^{r-1}(U)} \left\| \di \left( \a \left( \frac{\cdot}{\varepsilon} \right) \di w^\varepsilon \right) \right\|_{H^{-1}_{\di} \Lambda^{r-1} (U)} \\
											& \leq \left\| u^\varepsilon - w^\varepsilon \right\|_{H^1_{\di} \Lambda^{r-1}(U)} \left\| \di \left( \a \left( \frac{\cdot}{\varepsilon} \right) \di w^\varepsilon \right) \right\|_{H^{-1} \Lambda^{r-1} (U)}.
\end{align*}
Using that $u^\varepsilon$ is a solution of the equation $\di \left( \a \left( \frac{x}{\ep} \right) \di u^\ep \right) = 0$, we obtain
\begin{equation*}
\int_U \di \left( u^\varepsilon - w^\varepsilon \right) (x) \wedge \a \left( \frac{x}{\varepsilon} \right) \di u^\varepsilon (x) = 0.
\end{equation*}
Combining the two previous displays with the Poincar\'e inequality yields
\begin{align*}
\left\| \di u^\varepsilon - \di w^\varepsilon \right\|_{L^2\Lambda^r(U)}^2  & \leq C \int_U \di \left( u^\varepsilon - w^\varepsilon \right) (x) \wedge \a \left( \frac{x}{\varepsilon} \right) \di \left( u^\varepsilon - v^\varepsilon \right) (x) \\
										& \leq C \left\| u^\varepsilon - w^\varepsilon \right\|_{H^1_{\di} \Lambda^r(U)} \left\| \di \left( \a \left( \frac{\cdot}{\varepsilon} \right) \di w^\varepsilon \right) \right\|_{H^{-1}_{\di}\Lambda^{r-1} (U)} \\
										& \leq C \left\| \di u^\varepsilon - \di w^\varepsilon \right\|_{L^2\Lambda^r(U)} \left\| \di \left( \a \left( \frac{\cdot}{\varepsilon} \right) \di w^\varepsilon \right) \right\|_{H^{-1}\Lambda^{r-1} (U)}.
\end{align*}
Thus
\begin{equation*}
\left\| \di u^\varepsilon - \di w^\varepsilon \right\|_{L^2\Lambda^r(U)} \leq C \left\| \di \left( \a \left( \frac{\cdot}{\varepsilon} \right) \di w^\varepsilon  \right) \right\|_{H^{-1} \Lambda^{r-1}(U)}.
\end{equation*}
Using the estimate~\eqref{homogthmstep1} and another application of the Poincar\'e inequality completes the proof of~\eqref{homogthmstep2}.

\medskip

\textit{Step 4.} Recall that at the beginning of the proof, we assumed that the volume of the set $U$ is equal to $1$.
The objective of this step is to prove that, for each form $w \in H^1_{\di , 0} \Lambda^{r -1} (U) \cap \left( C^{r-1}_{\di , 0}(U) \right)^\perp$, one has the estimate
\begin{equation} \label{mainest.step4homthm}
\|  w \|_{L^2 \Lambda^{r-1} (U)} \leq \| \di w \|_{\underline{H}^{-1} \Lambda^r (U)}.
\end{equation}
To this end, we let $v$ be the unique solution in the space $H^1_{\di , 0} \Lambda^{r -1} (U) \cap \left( C^{r-1}_{\di , 0}(U) \right)^\perp$ of the problem
\begin{equation} \label{eqdirhom}
\left\{ \begin{aligned}
\delta \di v = w &~ \mbox{in} ~U,\\
\mathbf{t} v = 0 & ~\mbox{on} ~\partial U.
\end{aligned} \right.
\end{equation}
The existence and uniqueness of such a solution are obtained by minimizing the quantity
$
\mathcal{J} (v) := \left\langle \di v , \di v\right\rangle_U - \left\langle w ,  v\right\rangle_U
$
on the space $H^1_{\di , 0} \Lambda^{r -1} (U) \cap \left( C^r_{\di , 0}(U) \right)^\perp$ and requires to use the Poincar\'e inequality stated in Proposition~\ref{PoincareineqH10conv}; the details are omitted.

\smallskip

We note that the form $v$ satisfies the following properties:
$
\di \di v =0 \in L^2 \Lambda^{r+1} (U), ~\delta \di v =w \in L^2 \Lambda^{r-1} (U)$ and $\mathbf{t} \di v = 0.
$
 The third property $\mathbf{t} \di v = 0$ is implied by the condition $\mathbf{t} v = 0 $ and can be deduced from a direct computation. As a consequence, the Gaffney-Friedrichs inequality stated in Proposition~\ref{GFCDsmooth} implies that the exterior derivative $\di v$ belongs to the space $H^1 \Lambda^{r} (U)$, together with the estimate
\begin{equation*}
\left\|\di v \right\|_{H^1 \Lambda^{r} (U)} \leq C \left( \left\|w \right\|_{L^2 \Lambda^{r-1} (U)} +   \left\|\di v \right\|_{L^2 \Lambda^{r} (U)}  \right).
\end{equation*}
Testing the equation~\eqref{eqdirhom} with the form $v$ and using the Poincar\'e inequality shows
\begin{equation*}
\left\|\di v \right\|_{L^2 \Lambda^{r} (U)} \leq C  \left\|w \right\|_{L^2 \Lambda^{r-1} (U)}.
\end{equation*}
Combining the two previous displays implies
\begin{equation} \label{h11divv}
\left\|\di v \right\|_{H^1 \Lambda^{r} (U)} \leq C \left\|w \right\|_{L^2 \Lambda^{r-1} (U)}.
\end{equation}
Testing the equation~\eqref{eqdirhom} with $w$ then shows
\begin{align*}
\left\langle \di w , \di v \right\rangle_U  =\left\langle  w ,  w \right\rangle_U =  \|  w \|_{L^2 \Lambda^{r-1} (U)}.
\end{align*}
On the other hand, by the definition of the $\underline{H}^{-1} \Lambda^r (U)$-norm and the estimate~\eqref{h11divv}, one has
\begin{equation*}
\left\langle \di w , \di v \right\rangle_U  \leq \| \di w \|_{\underline{H}^{-1} \Lambda^r (U)} \| \di v \|_{H^{1} \Lambda^r (U)}
							\leq \| \di w \|_{\underline{H}^{-1} \Lambda^r (U)} \| w \|_{L^2 \Lambda^{r-1} (U)}.
\end{equation*}
Combining the two previous displays completes the proof of Step 4.

\medskip

\textit{Step 5.} The objective of this step is to prove the following estimate
\begin{equation} \label{mainestimateStep3homog}
\left\| \di w^\varepsilon - \di u \right\|_{\underline{H}^{-1} \Lambda^r(U)} \leq  \O_2 \left( \frac{C\left\| \di f \right\|_{H^1\Lambda^r (U)}\varepsilon^\alpha}{l^{2 + d/2}}  \right).
\end{equation}
First, one has the equality
\begin{equation*}
\di w^\varepsilon  - \di u = \di \left(  \varepsilon \zeta_l \sum_{|I| = r } (\di u)_I \phi_{m,I} \left( \frac{\cdot}{\varepsilon} \right) \right)
\end{equation*}
and therefore, since $w^\varepsilon - u \in H^1_{\di , 0} \Lambda^{r-1} (U)$, one has
\begin{equation*}
\left\| \di w^\varepsilon - \di u \right\|_{\underline{H}^{-1} \Lambda^r(U)} \leq C \left\| \di u \right\|_{L^\infty(U_l)} \sum_{|I| = r } \varepsilon  \left\| \phi_{m,I} \left( \frac{\cdot}{\varepsilon} \right) \right\|_{L^2\Lambda^r(U)}.
\end{equation*}
But with the same proof as in Step 4, with the cube $\varepsilon \cu_m$ instead of the set $U$, one has
\begin{equation*}
 \varepsilon  \left\| \phi_{m,I} \left( \frac{\cdot}{\varepsilon} \right) \right\|_{L^2\Lambda^r(U)}  \leq  \varepsilon  \left\| \phi_{m,I} \left( \frac{\cdot}{\varepsilon} \right) \right\|_{L^2\Lambda^r\left(\varepsilon \cu_m \right)} 
		 \leq  C \varepsilon  \left\| \phi_{m,I} \right\|_{\underline{L}^2\Lambda^r\left( \cu_m \right)}
		  \leq   C \varepsilon  \left\| \di \phi_{m,I}  \right\|_{\underline{H}^{-1}\Lambda^r\left( \cu_m \right)}.
\end{equation*}
We then apply  the estimate~\eqref{propphi1} and note that, by the definition of the integer $m$ given in~\eqref{eq:defintegerm}, one has the estimate $\ep 3^m \leq C$. This implies
\begin{equation*}
\varepsilon  \left\| \phi_{m,I} \left( \frac{\cdot}{\varepsilon} \right) \right\|_{L^2\Lambda^r(U)} \leq C \varepsilon  \left\| \di \phi_{m,I} \right\|_{\underline{H}^{-1}\Lambda^r\left( \cu_m \right)} \leq  \O_2  \left( C \varepsilon^\alpha \right).
\end{equation*}
Then, by Proposition~\ref{propellipticregularity}, one obtains
\begin{equation*}
\left\| \di w^\varepsilon - \di u \right\|_{\underline{H}^{-1} \Lambda^r(U)} \leq \| \di f \|_{L^2\Lambda^r (U)} \O_2 \left(\frac{C\varepsilon^{\alpha}}{l^{1+d/2}} \right). 
\end{equation*}
This completes the proof of the estimate~\eqref{mainestimateStep3homog}.

\smallskip

\textit{Step 6.} The conclusion. We first use the triangle inequality
\begin{align*}
\left\| \di u^{\varepsilon} - \di u \right\|_{\underline{H}^{-1}\Lambda^r(U)} & \leq \left\| \di u^{\varepsilon} - \di w^{\varepsilon} \right\|_{\underline{H}^{-1}\Lambda^r(U)} + \left\| \di w^{\varepsilon} - \di u \right\|_{\underline{H}^{-1}\Lambda^r(U)} \\
										& \leq \left\| \di u^{\varepsilon} - \di w^{\varepsilon} \right\|_{\underline{L}^2 \Lambda^{r}(U)}  + \left\| \di w^{\varepsilon} - \di u \right\|_{\underline{H}^{-1}\Lambda^r(U)}.
\end{align*}
By the main estimate~\eqref{homogthmstep2} of Step 3 and~\eqref{mainestimateStep3homog} of Step 5, we obtain
\begin{equation} \label{eq:mainest.step6}
\left\| \di u^{\varepsilon} - \di u \right\|_{\underline{H}^{-1}\Lambda^r(U)} \leq \left\{ \begin{aligned}
C \left\| \di f \right\|_{H^1 \Lambda^r (U)} \left( l^{\frac{1}{d-2}} + \O_2 \left(\frac{ \varepsilon^{\alpha}}{l^{3 + d/2}} \right) \right) & ~ \mbox{if} ~d \geq 3,\\
C \left\| \di f \right\|_{H^1 \Lambda^r (U)} \left( l^{\frac 14} + \O_2 \left( \frac{ \varepsilon^{\alpha}}{l^{3 + d/2}}  \right) \right) \hspace{3mm} & ~ \mbox{if} ~d =2 .
\end{aligned} \right.
\end{equation}
Finally, the bound for $\left\|  u^{\varepsilon} - u \right\|_{L^2\Lambda^{r-1}(U)} $ is obtained from the previous inequality and the main estimate~\eqref{mainest.step4homthm} of Step 4. Indeed, since $u - u^\varepsilon$ belongs to the space $H^1_{\di , 0} \Lambda^{r -1} (U) \cap \left( C^{r-1}_{\di , 0}(U) \right)^\perp$, one has
\begin{equation*}
\left\| u^{\varepsilon} - u \right\|_{L^2\Lambda^{r-1}(U)} \leq C \left\| \di u^{\varepsilon} - \di u \right\|_{\underline{H}^{-1}\Lambda^r(U)}.
\end{equation*}
The estimate~\eqref{eq:mainest.step6} is valid for any value of $l \in (0,1]$, in particular we can choose $l := \ep^{\frac{\alpha (3 + d/2)}{2}}$ so that $\frac{ \varepsilon^{\alpha}}{l^{3 + d/2}} = \ep^{\frac{\alpha}{2}}$. By reducing the size of the exponent $\alpha$, one obtains
\begin{equation*}
\left\| u^{\varepsilon} - u \right\|_{L^2\Lambda^r(U)} + \left\| \di u^{\varepsilon} - \di u \right\|_{\underline{H}^{-1}\Lambda^r(U)} \leq \O_2 \left( C \ep^\alpha \right).
\end{equation*}
The proof of Theorem~\ref{homogenizationtheorem} is complete.
\end{proof}

\section{Duality} \label{section7}
The goal of this section is to study a duality property between the homogenization of $r$-forms and $(d-r)$-forms. We note that similar results were obtained independently by Serre ~\cite{serre2017periodic} in the case of periodic coefficients. For each environment $\a \in \Omega_r$ and each $x \in \Rd$, the operator $\a (x) \in \mathcal{L} \left( \Lambda^{r}(\Rd),   \Lambda^{(d-r)}(\Rd) \right) $ satisfies the ellipticity assumption~\eqref{ellipticityassumption}, so it is invertible and one can define the inverse operator $(\a (x))^{-1} \in \mathcal{L} \left( \Lambda^{(d-r)}(\Rd),   \Lambda^{r}(\Rd) \right)$, which satisfies the symmetry assumption~\eqref{symmetryassumption} and the following ellipticity condition
\begin{equation} \label{newellipticityassumption}
\frac 1 \lambda |p|^2 \leq \a(x)^{-1} p \wedge p \leq  \lambda |p|^2, \, \forall p \in \Lambda^{(d-r)}(\Rd).
\end{equation}
We denote by
\begin{multline*}
\Omega'_{d-r} := \Big\{ \a (\cdot) \, : \, \a : \Rd \rightarrow \mathcal{L} \left( \Lambda^{(d-r)}(\Rd) , \Lambda^{r} \left( \Rd \right) \right)  \mbox{ is Lebesgue measurable}  \\   \mbox{ and satisfies \eqref{symmetryassumption} and \eqref{newellipticityassumption}}\Big\}.
\end{multline*}
We equip this set with a family of sigma-algebras, for each open set $U \subseteq \Rd$,
\begin{multline*}
\mathcal{F}_r'(U) := \bigg\{ \sigma \mbox{-algebra on } \Omega_r' \mbox{ generated by the family of maps } \\
						 \a \rightarrow \int_U p \wedge \a(x) q \phi(x), \, p,q \in \Lambda^{r}(\Rd), \, \phi \in C^\infty_c (U) \bigg\}.
\end{multline*}
We also define the operator inv to be the mapping
\begin{equation*}
\mbox{inv}  : \left\{ \begin{aligned}
\Omega_{r} & \rightarrow \Omega'_{d-r}, \\
\a  & \rightarrow \a^{-1}.
\end{aligned}  \right.
\end{equation*}
We then define $\mbox{inv}_* \P$ the probability measure defined on the measured space $(\Omega_{d-r}', \F_{d-r}')$ by, for each $A \in \F_{d-r}'$,
$
\mbox{inv}_* \P_r (A) := \P_r \left( \mbox{inv}^{-1} \, A \right).
$
The probability space $(\Omega_{d-r}', \F_{d-r}', \mbox{inv}_* \P_r )$ satisfies the stationarity assumption~\eqref{stationarityassumption} and the finite range dependence assumption~\eqref{independenceassumption}. We then define, for each $(p,q) \in  \Lambda^{(d-r)}(\Rd) \times  \Lambda^{r}(\Rd)$ and each $m \in \N$,
\begin{equation*}
J_{\mathrm{inv}}(\cu_m ,p ,q) := \sup_{u \in \A^{\mathrm{inv}}\left(\cu_m \right) } \fint_{\cu_m} \left( -\frac 12 \a^{-1} \di u \wedge \di u - \a^{-1} \di u  \wedge p +   q \wedge  \di u  \right),
\end{equation*}
where $\A^{\mathrm{inv}}\left(\cu_m \right)  $ is the set of solutions under the environment $\a^{-1}$, i.e.,
\begin{equation} \label{def.Ainv}
 \A_{\mathrm{inv}}\left(\cu_m \right) := \left\{ u \in H^1_\di \Lambda^{(d-r-1)} \left(\cu_m \right) \, :\, \forall v \in C^\infty_c\Lambda^{r-1} \left(\cu_m \right), \int_{\cu_m} du \wedge \a^{-1} dv = 0 \right\}.
\end{equation}
This quantity satisfies the conclusions of Proposition~\ref{basicpropJ} and Theorem~\ref{maintheorem}. In particular, there exist a constant $C(d,\lambda) < \infty$, an exponent $\alpha(d,\lambda) > 0$ and a linear operator
$
\ahominv \in \mathcal{L} \left( \Lambda^{(d-r)}(\Rd), \Lambda^{r}(\Rd) \right)
$
such that, for each $m \in \N$,
\begin{equation*}
\sup_{p \in B_1\Lambda^{(d-r)}(\Rd)} \E \left[ J_{\mathrm{inv}}(\cu_m ,p ,\ahominv \, p) \right] \leq C 3^{-m \alpha}.
\end{equation*}
We can now present the proof of Theorem~\ref{dualityprop}.

\begin{proof}[Proof of Theorem~\ref{dualityprop}]
First we need to prove the following result, for each integer $r \in \{0 , \ldots, d\}$ and each integer $m \in \N,$
\begin{equation} \label{eq:identtity6.3}
\A_{\mathrm{inv}} \left(\cu_m \right) = \left\{ v \in H^1_\di \Lambda^{(d-r-1)}  \left(\cu_m \right) ~:~ \di v = \a \di u \mbox{ with } u \in \A  \left(\cu_m \right)  \right\}.
\end{equation}
We split the argument into 2 steps:
\begin{itemize}
\item We prove that each form $ v \in H^1_\di \Lambda^{(d-r-1)} \left(\cu_m \right)$ satisfying the equality $\di v = \a \di u$ for some solution $u \in \A \left(\cu_m \right)$ belongs to $\A_{\mathrm{inv}}\left(\cu_m \right)$. Indeed, for each $w \in C^\infty_c \Lambda^{r-1}\left(\cu_m \right)$, one has, by the symmetry assumption~\eqref{symmetryassumption} and the equality~\eqref{IPPdi},
\begin{equation} \label{ad09}
\int_{\cu_m} \di v \wedge \a^{-1} \di w = \int_{\cu_m} \di w \wedge \a^{-1} \di v =  \int_{\cu_m}  \di w \wedge \di u = 0.
\end{equation}
\item We prove that for each $v \in \A_{\mathrm{inv}} \left(\cu_m \right)$, there exists $u \in \A \left(\cu_m \right)$ such that $\di v = \a \di u$. Indeed, if $v \in \A_{\mathrm{inv}} \left(\cu_m \right)$, then $\a^{-1} \di v$ belongs to  $ L^2 \Lambda^{r} \left(\cu_m \right)$ and satisfies
\begin{equation*}
d\left(  \a^{-1} \di v \right) = 0 ~\mbox{in}~\cu_m.
\end{equation*}
Consequently $\a^{-1} \di v$ belongs to the space $H^1_\di \Lambda^{r} \left(\cu_m \right)$. We can apply Proposition~\ref{Poincarelemma}, to prove that there exists a form $u \in H^1_\di \Lambda^{r} \left(\cu_m \right)$ such that
\begin{equation*}
\a^{-1} \di v = \di u  ~\mbox{in}~\cu_m.
\end{equation*}
There only remains to prove that $u \in \A \left(\cu_m \right)$, it is a consequence of the following computation: for each form $w \in C^\infty_c \Lambda^{(d-r-1)} \left(\cu_m \right)$, one has, by the symmetry assumption~\eqref{symmetryassumption} and the identity~\eqref{IPPdi},
\begin{equation*}
\int_{\cu_m} \di u \wedge \a\di w = \int_{\cu_m} \di w \wedge \a \di u =  \int_{\cu_m} \di w \wedge \di v = 0.
\end{equation*}
This completes the proof of the identity~\eqref{eq:identtity6.3}.
\end{itemize}
From the identity~\eqref{eq:identtity6.3}, we deduce the equality, for each pair $(p,q) \in \Lambda^r \left( \Rd\right) \times  \Lambda^{d-r} \left( \Rd\right)$,
\begin{equation} \label{eq:eq1742}
J_{\mathrm{inv}}(\cu_m ,p ,q) =  J\left(\cu_m ,q,p\right).
\end{equation}
The identity~\eqref{eq:eq1742} is a consequence of the following computation
\begin{align*}
J_{\mathrm{inv}}(\cu_m ,p ,q) & = \sup_{u \in \A^{\mathrm{inv}}\left(\cu_m \right) } \fint_{\cu_m} \left( -\frac 12  \a^{-1}  \di u \wedge \di u -   \a^{-1} \di u \wedge p  +   q \wedge  \di u  \right) \\
& = \sup_{v \in \A\left(\cu_m \right)} \fint_{\cu_m} \left( -\frac 12   (\a \di v) \wedge (\a \di v) -  \a^{-1}  (\a \di v) \wedge p  +   q \wedge  (\a \di v)  \right)  \\
& =  \sup_{v \in \A\left(\cu_m \right)} \fint_{\cu_m} \left( -\frac 12 \di v \wedge \a \di v -   \di v  \wedge p+   q \wedge \a \di v \right)  \\
& = J\left( \cu_m,-q,-p\right) \\
& =  J\left(\cu_m ,q,p\right).
\end{align*}
Thus, by Theorem~\ref{maintheorem}, for each integer $m \in \N$,
\begin{equation*}
\sup_{p \in B_1\Lambda^{r}(\Rd)} \E \left[ J_{\mathrm{inv}}(\cu_m ,\ahom p ,p) \right] = \sup_{p \in B_1\Lambda^{r}(\Rd)} \E \left[ J(\cu_m , p ,\ahom p) \right] \leq  C 3^{- \alpha m}.
\end{equation*}
The previous inequality can be rewritten
\begin{equation*}
\sup_{q \in B_1\Lambda^{(d-r)}(\Rd)} \E \left[ J_{\mathrm{inv}}(\cu_m ,q ,\ahom^{-1} q) \right] \leq  C 3^{- \alpha m}.
\end{equation*}
Since the homogenized matrix is unique, we have $\ahom^{-1} = \ahominv$. This gives the expected result.
\end{proof}

\appendix

\section{Regularity estimates for differential forms} \label{appA}
In this appendix, we record some properties about the regularity of the solutions of the constant coefficient equation $\di \ahom \di u$. The two main results are the pointwise interior estimate, Proposition~\ref{propellipticregularity} and the $H^2$ boundary estimate, Proposition~\ref{globalest}. Both results are used in the proof of Theorem~\ref{homogenizationtheorem}. Most of these proofs are an adaptation of the classical proofs of the regularity theory of uniformly elliptic equations (cf.~\cite{gilbarg2015elliptic}). 

We first state two propositions, Proposition~\ref{GFint.prop}, an interior Gaffney-Friedrich inequality, and Proposition~\ref{H2interiorregularityestimate}, an interior $H^2$ regularity estimate. We then use these two ingredients to prove the pointwise interior estimate, Proposition~\ref{propellipticregularity}. We finally prove a global $H^2$ regularity result for the solutions of $\di \ahom \di u = 0$, Proposition~\ref{globalest}.

The following proposition is an interior version of the Gaffney-Friedrich inequalities stated Proposition~\ref{GFCDsmooth}.
\begin{proposition}[Interior Gaffney-Friedrich inequality] \label{GFint.prop}
There exists a constant $C := C(d) < \infty$ such that, for every $0 \leq r \leq d$, every open bounded subsets $V , U \subseteq \Rd$ satisfying $\bar{V} \subseteq U$, and every form $u \in L^2 \Lambda^r \left(U \right)$ such that $\di u \in L^2 \Lambda^{r+1} \left(U \right)$ and $\delta u \in L^2 \Lambda^{r-1} \left(U \right)$, one has $u \in H^1 \Lambda^r  \left(V \right)$ with the estimate
\begin{equation} \label{GFint.eq}
\left\| \nabla u  \right\|_{L^2 \Lambda^r  \left(V \right)} \leq C \left( \left\| \di u  \right\|_{L^2 \Lambda^r  \left(U \right)} + \left\| \delta u  \right\|_{L^2 \Lambda^r  \left(U \right)} + \frac{1}{\dist (V, \partial U) } \left\| u  \right\|_{L^2 \Lambda^r  \left(U \right)} \right).
\end{equation}
\end{proposition}

\begin{proof}
The proof relies on the following observation: given a form $u = \sum_{|I| = r} u_I \di x_I \in C^\infty \Lambda^r (U) $, one has the equality
$
(\di \delta + \delta \di ) u = \sum_{|I| =r}  \Delta u_I  \di x_I.
$
We select a cutoff  function $\eta \in C_c^\infty (U)$ such that
\begin{equation} \label{prop.eta.appA}
\indc_V \leq \eta \leq 1, ~~ |\nabla \eta| \leq \frac{C}{\dist(V, \partial U)},
\end{equation} 
and compute
\begin{equation*}
\left\| \nabla u  \right\|_{L^2 \Lambda^r  \left(V \right)}^2 \leq  \sum_I \int_U \left| \nabla \left( u_I \eta \right)  \right|^2(x) \, dx = \left\langle \di (u \eta) , \di ( u \eta) \right\rangle_U + \left\langle \delta (u \eta) , \delta ( u \eta) \right\rangle_U.
\end{equation*}
By the formulas~\eqref{propd} and the properties on the cutoff function $\eta$ stated in~\eqref{prop.eta.appA}, one has
\begin{align*}
 \left\langle \di (u \eta) , \di ( u \eta) \right\rangle_U + \left\langle \delta (u \eta) , \delta ( u \eta) \right\rangle_U
									& \leq C \left( \left\| \di u  \right\|_{L^2 \Lambda^{r+1}  \left(U \right)}^2 +  \left\| \delta u  \right\|_{L^2 \Lambda^{r+1}  \left(U \right)}^2 + \frac{1}{\dist (V, \partial U)^2 } \left\| u  \right\|_{L^2 \Lambda^r  \left(U \right)}^2 \right).
\end{align*}
Combining the three previous displays completes the proof of~\eqref{GFint.eq}.
\end{proof}

We then use the previous interior Gaffney-Friedrich inequality to prove the following interior $H^2$ estimate. The proof of the following proposition is an adaptation of the standard interior $H^2$ estimate for the solutions of uniformly elliptic equations, cf.~\cite[Theorem 8.8]{gilbarg2015elliptic}.

\begin{proposition}[Interior $H^2$ regularity estimate] \label{H2interiorregularityestimate}
There exists a constant $C := C( d , \lambda) < \infty$ such that for every open bounded subsets $U,V \subseteq \Rd$ such that $\bar{V} \subseteq U$, every $1 \leq r \leq d$ and every form $u \in H^1_\di \Lambda^{(r-1)} (U)$ solution of the equation
\begin{equation} \label{eqellH2int}
\di \left( \ahom \di u \right) = 0 \mbox{ in } U,
\end{equation}
the form $\di u$ belongs to the space $H^1 \Lambda^r (V)$ and the following estimate holds
\begin{equation*}
\| \nabla \di u \|_{ L^2 \Lambda^{r} (V)} \leq C \left(  \frac{1}{\dist (V, \partial U )} \| \di u \|_{ L^2 \Lambda^{r} (U)} +  \frac{1}{\dist (V, \partial U )^2} \| u \|_{ L^2 \Lambda^{(r-1)} (U)} \right).
\end{equation*}
\end{proposition}

\begin{proof}
The main idea of this proof is to follow the proof of~\cite[Theorem 8.8]{gilbarg2015elliptic} for harmonic functions and combine it with the interior Gaffney-Friedrich inequality. 

First note that without loss of generality, one can assume that the form $u$ belongs to the space $C^{r-1}_{\di}(U)^\perp$. We select two open sets $W, W_1$ such that $V \subseteq W \subseteq \bar{W} \subseteq W_1 \subseteq \bar{W_1} \subseteq U$ and such that
\begin{equation} \label{VWW1U}
\dist(V, \partial W)  = \dist(W, \partial W_1) = \dist (W_1 , \partial U) = \frac{\dist(V, \partial W)}{3}.
\end{equation}
Additionally, we select a cutoff function $\eta \in C_c^\infty (U)$ such that
\begin{equation} \label{defetaUVW}
\indc_V \leq \eta \leq \indc_W, ~\| \nabla \eta \| \leq \frac{C}{\dist(V, \partial U)}.
\end{equation}
Let $h > 0$ be small, choose an integer $k \in \{ 1 , \ldots, d \}$ and denote by
$
v := D_k^{-h} \left( \eta^2 D_k^{h} u \right),
$
where $D^h_k$ is the difference quotient, defined by the formula
$
D^h_k u(x) = \frac{u(x+ he_k) - u(x)}{h}.
$
If the parameter $h$ is small enough then the form $v$ belongs to the space $H^1_{\di , 0} \Lambda^{r-1}(U)$ and can be used as a test function in~\eqref{eqellH2int}. We obtain the equality
$
\left\langle \di u , \ahom \di v \right\rangle_{U} = 0.
$
Standard computations on the previous display lead to the inequality
\begin{equation*}
\left\langle D_k^{h} \di u , D_k^{h} \di u \right\rangle_{V} \leq \frac{C}{\dist (V , \partial U)^2}  \left\langle  D_k^{h} u , D_k^{h} u  \right\rangle_{W}.
\end{equation*}
Since we assume $u \in C^{r-1}_d(U)^\perp$, one has $\delta u = 0$ in $U$ and in particular $\delta u \in L^2 \Lambda^{r-2}(V)$. From this we deduce that the form $u$ satisfies the assumptions of Proposition~\ref{GFint.prop} and consequently it belongs to the space $H^1 \Lambda^{r-1}(W_1)$ and satisfies the estimate
\begin{equation*}
\| \nabla u \|_{ L^2 \Lambda^{r} (W_1)}  \leq C \left( \| \di u \|_{ L^2 \Lambda^{r} (U)} +  \frac{1}{\dist (V, \partial U )} \| u \|_{ L^2 \Lambda^{(r-1)} (U)} \right),
\end{equation*}
where we used the estimate~\eqref{defetaUVW}. Additionally, according to \cite[Lemma 7.23]{gilbarg2015elliptic}, one has the inequality
$
\left\|  D_k^{h} u \right\|_{L^2 (W)} \leq C \| \nabla u \|_{ L^2 \Lambda^{r} (W_1)}
$
for $h > 0$ small enough. Combining the three previous displays and sending $h$ to $0$ gives, according to~\cite[Lemma 7.24]{gilbarg2015elliptic},
\begin{equation*}
\| \nabla \di u \|_{ L^2 \Lambda^{r} (V)}^2 \leq \frac{C}{\dist (V , \partial U)^2} \left(  \| \di u \|_{ L^2 \Lambda^{r} (U)}^2 + \frac{1}{\dist (V, \partial U )^2}  \| u \|_{ L^2 \Lambda^{(r-1)} (U)}^2 \right).
\end{equation*}
The proof is complete.
\end{proof}

\begin{proposition}[Elliptic regularity] \label{propellipticregularity}
There exists a constant $C := C(d,k,\lambda) < \infty$ such that for every open bounded subset $U \subseteq \Rd$, every integer $r \in \{ 0, \ldots, d \}$,  every integer $k \in \N$, every $R > 0$, and every solution of the equation
\begin{equation*}
\di \left( \ahom \di u \right) = 0 \mbox{ in } U,
\end{equation*}
the following pointwise estimate holds
\begin{equation} \label{pointwiseellipticestimatedu}
\left\| \nabla^k \di u\right\|_{L^\infty \Lambda^r \left(U_R\right)} \leq \frac{C}{R^{k + d/2}} \left\| \di u \right\|_{L^2 \Lambda^r \left(U\right)},
\end{equation}
where we used the notation $U_R := \left\{ x \in U  \, : \, \dist(x, \partial U) > R \right\}$.
\end{proposition}

\begin{proof}
We select an integer $k \in \N$, a non-negative real number $R > 0$, and a point $x \in U_R$. It is sufficient to prove~\eqref{pointwiseellipticestimatedu}, to show the estimate
\begin{equation} \label{pointwiseellipticestimatedu0}
 \left| \nabla^k \di u (x) \right| \leq  \frac{C}{R^{k + d/2}} \left\| \di u \right\|_{L^2 \Lambda^r \left(B_R(x)\right)},
\end{equation}
for some constant $C = C(d,k,\Lambda ) < \infty$. We split the proof into two steps.

\textit{Step 1.} We prove that there exists a constant $C = C(d,\lambda ) < \infty$, such that for every $l \in \N$, $\di u \in H^l \Lambda^r \left( B_{R/2^l}(x) \right)$ and
\begin{equation} \label{inductionstepl}
\left\| \nabla^l \di u\right\|_{L^2 \Lambda^r \left(B_{R/2^l}(x) \right)} \leq  \frac{C^l 2^{l^2/2} }{R^{l }} \left\| \di u \right\|_{L^2 \Lambda^r \left( B_R(x) \right)}.
\end{equation}
This inequality can be proved by induction on $l$. It is true for $l = 0$. We can use Proposition~\ref{H2interiorregularityestimate} to go from $l$ to $l+1$. Assume that the estimate~\eqref{inductionstepl} holds with the integer $l$. In that case, one has $\nabla^l \di u \in L^2 \Lambda^r \left( B_{R/2^l}(x) \right)$. It is straightforward to check the identity
$
\di \left( \nabla^l \di u \right) = 0.
$
Thus by Proposition~\ref{Poincarelemma}, there exists a form $v_l \in H^1_\di \Lambda^{r-1} \left( B_{R/2^l}(x) \right)$ such that 
$
v_l \in  C^{r-1}_\di \left( B_{R/2^l}(x ) \right)^\perp \mbox{ and } \di v_l = \nabla^l \di u.
$
It is moreover a straightforward to check that it satisfies the equation
$
\di \left( \ahom \di v_l \right) = 0.
$

Consequently, one can apply Proposition~\ref{H2interiorregularityestimate} to the form $v_l$ with the sets $U = B_{R/2^l}(x)$ and $V = B_{R/2^{l+1}}(x)$. This gives that the form $\nabla^{l+1} \di u$ belongs to the space $H^1 \Lambda^r \left( B_{R/2^{l+1}}(x) \right) $, and thus $\di u \in H^{l+1} \Lambda^r \left( B_{R/2^{l+1}}(x) \right)$ with the estimate
\begin{align*}
\left\| \nabla^{l+1} \di u \right\|_{L^2 \Lambda^r \left( B_{R/2^{l+1}}(x) \right)} \leq  \frac{C2^{l+1}}{R} \left(  \left\|  \nabla^{l} \di u \right\|_{ L^2 \Lambda^{r} \left( B_{R/2^l}(x)) \right)} + \frac{2^{l+1}}{R}  \| v_l \|_{ L^2 \Lambda^{(r-1)} \left(B_{R/2^l}(x) \right)} \right).
\end{align*}
By Proposition~\ref{PoincareineqH10conv}, the form $v_l$ satisfies the Poincar\'e inequality
\begin{equation*}
\| v_l \|_{ L^2 \Lambda^{(r-1)} (B_{R/2^l}(x))} \leq C \frac{2^l}{R} \| \nabla^l u \|_{ L^2 \Lambda^{(r-1)} \left(B_{R/2^l}(x)\right)}.
\end{equation*}
Combining the two previous displays yields
\begin{equation*}
\left\| \nabla^{l+1} \di u \right\|_{L^2 \Lambda^r \left( B_{R/2^{l+1}}(x) \right)} \leq  \frac{C2^{l+1}}{R} \|  \nabla^{l} \di u \|_{ L^2 \Lambda^{r} \left( B_{R/2^l}(x)\right)}.
\end{equation*}
Applying the induction hypothesis completes the proof.

\smallskip

\textit{Step 2.} From the first step, we get that for every integer $l \in \N$, the form $\nabla^{k+l} \di u$ belongs to the space $L^2 \Lambda^r \left( B_{R/2^{k+l}}(x) \right)$. In particular, by the Sobolev injection (see ~\cite[Chapter 4]{adams2003sobolev}), the form $\nabla^{k} \di u$ belongs to the space $L^\infty \Lambda^r \left( B_{R/2^{k+  d/2 + 1}}(x) \right),$ and we have the estimate
\begin{equation*}
\left\| \nabla^{k} \di u \right\|_{L^\infty \Lambda^r \left( B_{R/2^{k+ d/2 + 1}}(x) \right)} \leq \frac{C}{R^{k + d/2}} \left\| \di u \right\|_{L^2 \Lambda^r \left(B_R(x)\right)}.
\end{equation*}
This completes the proof of~\eqref{pointwiseellipticestimatedu0}.
\end{proof}

We then establish the following global $H^2$ estimate for the solutions of the equation $\di \ahom \di u = 0$.
\begin{proposition}[Global $H^2$ regularity] \label{globalest}
Let $U \subseteq \Rd$ be a smooth bounded domain of $\Rd$ and $r$ be an integer in the set $\{1, \ldots , d \}$. Let $f \in H^1_{\di} \Lambda^{r-1} \left( U \right)$ be such that $\di f  \in H^1 \Lambda^{r} \left( U \right)$. Let $u \in H^1_{d, 0} \Lambda^{r-1}(U)$ be a solution of the equation
\begin{equation} \label{globalestpb}
\left\{ \begin{aligned}
\di \left( \ahom \di u \right) = 0 \mbox{ in } U, \\
\mathbf{t} u = f \mbox{ on } \partial U,
\end{aligned}  \right.
\end{equation}
then the form $\di u$ belongs to the space $H^1 \Lambda^{r}(U)$ and one has the estimate
\begin{equation}\label{globalestres}
\left\|  \di u \right\|_{H^1 \Lambda^{r}(U)} \leq C  \left\| \di f \right\|_{H^1 \Lambda^r (U)} .
\end{equation}
\end{proposition}

\begin{proof}
We first note that two solutions of the equation~\eqref{globalestpb} differ by a form in the space $C^{r-1}_{\di , 0}$. This implies that two solutions of the equation~\eqref{globalestpb} have the same exterior derivative. Thus to prove the estimate~\eqref{globalestres}, it is enough to prove it for any solution of the equation~\eqref{globalest}.

The strategy of the proof is the following. One wants to apply a result from the regularity theory of strongly elliptic operators to the differential form $u$, see~\eqref{defstrongell} for a definition and~\cite{mclean2000strongly} for a reference on the topic of strongly elliptic differential operators. Unfortunately the operator $\di \ahom \di$ is not strongly elliptic. The strategy is then to solve the problem $\di \ahom \di + (-1)^r \star \di \delta u = 0$ with appropriate boundary conditions so that one has the equality $\star d \delta u=0 $ and the form $u$ is in fact a solution of the equation~\eqref{globalestpb}. Contrary to the operator $\di \ahom \di$, the operator $\di \ahom \di + (-1)^r \star \di \delta$ is strongly elliptic and a regularity theory exists for these operators. Thanks to this argument, one is able to derive $H^2$ boundary regularity estimate for the form $u$. This implies the regularity estimate~\eqref{globalestres}.

The main ideas of the proof can be found in~\cite[Chapter 2]{Sch95} and~\cite[Chapter 4]{mclean2000strongly}. We recall the notation for the set of harmonic forms with Dirichlet boundary condition,
\begin{equation*}
\mathcal{H}^{r-1}_D(U) := \mathcal{H}^{r-1}(U) \cap H^1_{\di,0}(U) :=  \left\{ u \in L^2\Lambda^{r-1}(U) \, : \, \di u = 0, ~ \delta u = 0 ~\mbox{in}~ U ~ \mbox{and} ~ \mathbf{t} u = 0 ~ \mbox{on } \partial U \right\}.
\end{equation*}
We split the proof into 8 steps:
\begin{itemize}
\item In Step 1, we show that there exists a unique solution denoted by $u$ in the space $\mathcal{H}^{r-1}_D(U)^\perp$ to the elliptic system
\begin{equation} \label{strongellpb}
\left\{ \begin{aligned}
\di \ahom \di u +  (-1)^r  \star \di \delta u = \di \ahom \di f &~\mbox{in} ~ U, \\
\mathbf{t} u = 0 &~\mbox{on} ~ \partial U, \\
 \mathbf{t} \delta  u = 0 &~ \mbox{on} ~ \partial U.
\end{aligned}  \right.
\end{equation}
\item In Step 2, we show that the form $u$ defined in Step 1 satisfies $\di \delta u = 0$ and is a solution of~\eqref{globalestpb}.
\item Steps 3 to 6 are the technical steps, we show, using the regularity theory for strongly elliptic operators developed in~\cite{mclean2000strongly}, the $H^2$ boundary regularity for the solution of the system 
\begin{equation*}
\left\{ \begin{aligned}
\di \ahom \di u +  (-1)^r  \star \di \delta u = \di \ahom \di f &~\mbox{in} ~ U, \\
\mathbf{t} u = 0 &~\mbox{on} ~ \partial U, \\
 \mathbf{t} \delta  u = 0 &~ \mbox{on} ~ \partial U.
\end{aligned}  \right.
\end{equation*}
\item In Steps 7 and 8, we combine the results of the previous steps to prove the regularity estimate~\eqref{globalestres}.
\end{itemize}

\textit{Step 1.} 
First, we prove that there exists a unique solution $u \in \mathcal{H}^{r-1}_D(U)^\perp$ of the system
\begin{equation*}
\left\{ \begin{aligned}
\di \ahom \di u +  (-1)^r  \star \di \delta u = \di \ahom \di f &~\mbox{in} ~ U, \\
\mathbf{t} u = 0 &~\mbox{on} ~ \partial U, \\
 \mathbf{t} \delta  u = 0 &~ \mbox{on} ~ \partial U.
\end{aligned}  \right.
\end{equation*}
This equation can be rewritten variationally the following way, there exists a form $u$ which belongs to the space $H^1 \Lambda^{r-1} (U) \cap  \mathcal{H}^{r-1}_D(U)^\perp$ such that $\mathbf{t}u = 0 $ on $\partial U$ and for each $v \in H^1 \Lambda^{r-1} (U)$ satisfying $\mathbf{t} v = 0$ on $\partial U$,
\begin{equation} \label{varregb}
 \int_{U} \di u \wedge \ahom \di v + \int_{U} \delta u \wedge  \star \delta v = \int_{U} \di f \wedge \ahom \di v.
\end{equation}
To solve this equation, we consider the associated energy: for $v \in H^1 \Lambda^{r-1} (U) \cap \mathcal{H}^{r-1}_D(U)^\perp$ satisfying the boundary condition $\mathbf{t} v = 0$, we define
\begin{equation*}
\mathcal{J}(v) :=  \int_{U} \di v \wedge \ahom \di v + \int_{U} \delta v \wedge  \star \delta v - \int_{U} \di f \wedge \ahom \di v.
\end{equation*}
Since the homogenized environment $\ahom $ satisfies the ellipticity assumption~\eqref{ellipticityassumption}, one has
\begin{equation*}
\int_{U} \di v \wedge \ahom \di v + \int_{U} \delta v \wedge  \star \delta v \geq \lambda \left\| \di v \right\|_{L^2 \Lambda^{r}(U)} +  \left\| \delta v \right\|_{L^2 \Lambda^{r}(U)}.
\end{equation*}
Moreover, by the Gaffney-Friedrich inequality stated in Proposition~\ref{GFCDsmooth}, we have
\begin{align*}
\int_{U} \di v \wedge \ahom \di v + \int_{U} \delta v \wedge  \star \delta v & \geq \lambda \left\| \di v \right\|_{L^2 \Lambda^{r}(U)} +  \left\| \delta v \right\|_{L^2 \Lambda^{r}(U)} \\
															& \geq c \left\| \nabla v \right\|_{L^2 \Lambda^{r-1}(U)},
\end{align*}
for some constant $c := c(d , \lambda, U) > 0$.
Arguing by contradiction, it is straightforward to prove the following Poincar\'e inequality: there exists a constant $C:= C(d, U) < \infty$ such that for each form $u$ belonging to the space $H^1 \Lambda^{r-1} (U) \cap \mathcal{H}^{r-1}_D(U)^\perp $ satisfying $\mathbf{t}u = 0$ on $\partial U$, on has the estimate
\begin{equation} \label{finalpoincareineqothharm}
\left\| u \right\|_{L^2 \Lambda^{r-1}(U)} \leq C \left\| \nabla u \right\|_{L^2 \Lambda^{r-1}(U)}.
\end{equation}
The previous inequality implies that the functional $\mathcal{J}$ is coercive on the space
$
\left\{ u \in H^1 \Lambda^{r-1} (U) \cap \mathcal{H}^{r-1}_D(U)^\perp ~:~ \mathbf{t}u = 0 ~\mbox{on} ~\partial U \right\},
$
equipped with the $H^1 \Lambda^{r-1} (U)$-norm. Moreover, the functional $\mathcal{J}$ is uniformly convex. The standard techniques of the calculus of variations then show that there exists a unique minimizer of the functional $\mathcal{J}$ denoted by $u$. By the first variation formula, one has, for each form $v \in H^1 \Lambda^{r-1} (U) \cap \mathcal{H}^{r-1}_D(U)^\perp$ satisfying the boundary condition $\mathbf{t} v = 0$ on $\partial U$,
\begin{equation*}
\int_{U} \di u \wedge \ahom \di v + \int_{U} \delta u \wedge  \star \delta v = \int_{U} \di f \wedge \ahom \di v.
\end{equation*}
Additionally, for each form $v \in \mathcal{H}^{r-1}_D(U)$, one has
\begin{equation*}
\int_{U} \di u \wedge \ahom \di v + \int_{U} \delta u \wedge  \star \delta v = \int_{U} \di f \wedge \ahom \di v = 0.
\end{equation*}
Thus for each form $v \in H^1 \Lambda^{r-1} (U)$ satisfying the boundary condition $\mathbf{t}v = 0$ on $\partial U$, we have
\begin{equation*}
 \int_{U} \di u \wedge \ahom \di v + \int_{U} \delta u \wedge  \star \delta v = \int_{U} \di f \wedge \ahom \di v
\end{equation*}
and the proof of Step 1 is complete. As a remark, we note that since the form $\di f$ belongs to the space $H^1 \Lambda^{r} (U)$, the form $\di \ahom \di f$ belongs to the space $L^2 \Lambda^{d-r+1} (U)$. Thus, if we denote by $g := \di \ahom \di f \in L^2 \Lambda^{d-r+1} (U)$, one has the identity
\begin{equation} \label{varregb2}
\int_{U} \di u \wedge \ahom \di v + \int_{U} \delta u \wedge  \star \delta v = \int_{U} g \wedge  v,
\end{equation}
for each form $v \in H^1 \Lambda^{r-1} (U)$ satisfying the boundary condition $\mathbf{t}v = 0$ on $\partial U$.

\smallskip

\textit{Step 2.} We show that the form $u$ constructed in the previous step is a solution of the equation
\begin{equation*}
\left\{ \begin{aligned}
\di \ahom \di u = \di \ahom \di f ~\mbox{in}~U, \\
\mathbf{t} u = 0 ~\mbox{on} ~ \partial U.
\end{aligned}  \right.
\end{equation*}
To prove this result, it is enough, by Proposition~\ref{propdensity}, to show that for each form $v \in H^1 \Lambda^{r-1} (U)$ satisfying the boundary condition $\mathbf{t}v = 0$ on $\partial U$,
\begin{equation*}
\int_{U} \di u \wedge \ahom \di v  = \int_{U} \di f \wedge \ahom \di v.
\end{equation*}
To this end, we fix a form $v \in H^1 \Lambda^{r-1} (U) \cap  H^1_{\di , 0} \Lambda^{r-1} (U) $. We denote by $\alpha_v$ the unique form of~$ C^{r-1}_{\di,0}(U)$ which satisfies
\begin{equation*}
\alpha_v = \underset{\alpha \in C^{r-1}_{\di,0}(U)}{\argmin} \left\| v - \alpha \right\|_{L^2 \Lambda^{r-1} (U)},
\end{equation*}
and we set $w := v - \alpha_v$. In particular, this form satisfies, for each smooth form $\gamma \in C^\infty_c \Lambda^{r-2}(U)$,
$
\left\langle w , \di \gamma \right\rangle_{U} = 0.
$
This implies $\delta w = 0$. Moreover it is clear that $\di w = \di v$ and that $\mathbf{t}w = 0$ on the boundary $\partial U$. Thus, by the Gaffney-Friedrich inequality, the form $w$ belongs to the space $H^1 \Lambda^{r-1} (U)$ and it can be tested in the equation~\eqref{varregb}. This gives
\begin{equation*} 
 \int_{U} \di u \wedge \ahom \di w  = \int_{U} \di f \wedge \ahom \di w.
\end{equation*}
Since $\di w = \di v$, the previous equality can be rewritten
\begin{equation*}
 \int_{U} \di u \wedge \ahom \di v  = \int_{U} \di f \wedge \ahom \di v,
\end{equation*}
which is the desired result. The proof of Step 2 is complete.

\smallskip

\textit{Step 3.} In this step, we follow the arguments of the proofs of~\cite[Section 2.3]{Sch95}. Let $X$ be a smooth vector field supported in $U$, tangent to the boundary of $U$. This vector field generates a global flow $\psi^X_t$ such that for every $t \in \R$, $\psi^X_t$ is a smooth diffeomorphism of $ U$. The pullback $\left( \psi^X_t\right)^*$ gives rise to the following linear mapping, for $t \in \R \setminus \{ 0 \}$,
\begin{equation*}
\Sigma_t^X :=  \left\{ \begin{aligned}[l]
 	L^2 \Lambda^{r-1}(U) & \rightarrow L^2 \Lambda^{r-1}(U)  \\
       \omega  & \mapsto \frac 1t \left( \left( \psi^X_t\right)^*  \omega - \omega  \right).
    \end{aligned}\right.
\end{equation*}
This operator satisfies a number of convenient properties which are listed below. Most of these properties can be found in~\cite[Section 2.3 and Lemma 2.3.1]{Sch95}.

\begin{lemma}[Properties of $\Sigma_t^X$] \label{propertiesSigmatX}
The operator $\Sigma_t^X$ satisfies the following properties:
\begin{itemize}
\item Since the pullback $\left( \psi^X_t\right)^*$ commutes with the exterior derivative $\di$, so does the mapping $\Sigma_t^X$,
\begin{equation*}
\di \Sigma_t^X \omega =  \Sigma_t^X \di \omega.
\end{equation*}
\item Since the pullback $\left( \psi^X_t\right)^*$ commutes with the projection to the tangential component, so does the mapping $\Sigma_t^X$,
\begin{equation*}
\mathbf{t} \omega = 0 \implies \mathbf{t} \Sigma_t^X \omega = 0.
\end{equation*}
\item There exists a constant $C := C(d , U, X) <\infty$ such that for each $t \in [-1,1]$ and each form $\omega \in H^1 \Lambda^{r-1}(U)$,
\begin{equation*}
\left\| \Sigma_t^X \omega \right\|_{L^2 \Lambda^{r-1}(U)} \leq C \left\|  \omega \right\|_{H^1 \Lambda^{r-1}(U)}.
\end{equation*}
\item Let $\Theta_t^X : L^2 \Lambda^{r-1}(U)  \rightarrow L^2 \Lambda^{r-1}(U)$ be the operator defined according to the formula
\begin{equation*}
\Sigma_t^X \star \omega =\star  \Sigma_t^X  + \star \Theta_t^X \omega.
\end{equation*}
Then there exists a constant $C := C(d , U, X) <\infty$ such that for each $t \in [-1,1],$
\begin{equation*}
\left\| \Theta_t^X \omega \right\|_{L^2 \Lambda^{r-1}} \leq C \left\| \omega \right\|_{L^2 \Lambda^{r-1}}.
\end{equation*}
\item Let $\Theta_{t}^{\ahom, X} : L^2 \Lambda^{r}(U)  \rightarrow L^2 \Lambda^{d-r}(U)$ be the operator defined according to the formula
\begin{equation*}
\Sigma_t^X \ahom  =\ahom  \Sigma_t^X  - \left(\psi_X^t \right)^*  \Theta_{t}^{\ahom, X} \omega,
\end{equation*}
then there exists a constant $C := C(d , U, X) <\infty$ such that for each $t \in [-1,1],$
\begin{equation*}
\left\| \Theta_{t}^{\ahom, X} \omega \right\|_{L^2 \Lambda^{d-r}(U)} \leq C \left\| \omega \right\|_{L^2 \Lambda^{r}(U)}.
\end{equation*}
\end{itemize}
\end{lemma}

\begin{proof}
All these properties are proved in~\cite[Section 2.3 and Lemma 2.3.1]{Sch95} except for the last one, which we now prove.

An explicit computation gives the following formula for the operator $\Theta_{t}^{\ahom, X}$,
\begin{equation*}
\Theta_{t}^{\ahom, X} = \frac{\left(\psi_X^{-t} \right)^* \ahom  \left(\psi_X^{t} \right)^* - \ahom }{t}.
\end{equation*}
Then, using this formula, we note that there exist smooth functions $\phi_{I,J} : [-1,1] \times U \rightarrow \R$, with $I,J \subseteq \{1, \ldots, d  \}$ satisfying $|I|=r$ and $|J |= d-r$, such that, for each form $\omega \in L^2 \Lambda^r(U)$,
\begin{equation*}
\left(\psi_X^{-t} \right)^* \ahom  \left(\psi_X^{t} \right)^* \omega = \sum_{I,J} \omega_I(x) \phi_{I,J}(t,x) \di x_J.
\end{equation*}
Using that all the functions $\phi_{I,J}$ are smooth and that $\left(\psi_X^0 \right)^* \ahom \left(\psi_X^{0} \right)^* = \ahom $, we obtain the result.
\end{proof}

With these properties, one can prove the following estimate, which allows to perform integration by parts: there exists a constant $C := C(d , \lambda, X) <\infty$ such that for each $t \in [-1, 1]$, and for each pair of forms $v , w \in H^1 \Lambda^{r-1} (U)$,
\begin{equation} \label{formipp}
\left| \int_{U} \di v \wedge \ahom \di \Sigma_{-t}^X w - \int_{U} \di \Sigma_t^X v \wedge \ahom \di  w \right| \leq C  \left\| \di v \right\|_{L^2 \Lambda^{r}(U)}\left\| \di w\right\|_{L^2 \Lambda^{r}(U)}
\end{equation}
and
\begin{equation} \label{formipp2}
\left| \int_{U} \delta \Sigma_t^X v \wedge \star \delta  w - \int_{U} \delta v \wedge \star \delta \Sigma_{-t}^X w \right| \\ \leq C  \left\| \delta v \right\|_{L^2 \Lambda^{r-2}(U)}\left\| \delta w\right\|_{L^2 \Lambda^{r-2}(U)}.
\end{equation}
We first prove the estimate~\eqref{formipp}. We note that, for each pair of forms $\omega, \xi \in L^2 \Lambda^r(U)$,
\begin{equation*}
\Sigma_{t}^X ( \omega \wedge \ahom \xi )=  \Sigma_{t}^X \omega \wedge \ahom \xi  - \left( \psi^X_t \right)^* \left( \omega \wedge \ahom \Sigma_{-t}^X \xi \right) +  \left( \psi^X_t \right)^* \left( \omega \wedge \left( \psi^X_{-t} \right)^* \Theta_{-t}^{\ahom, X} \xi \right) .
\end{equation*}
Integrating this equation over $U$ and using the formula~\eqref{changevar} gives
\begin{equation*}
\int_U  \Sigma_{t}^X  w \wedge \ahom \xi  - \int_U \left( w \wedge \ahom \Sigma_{-t}^X \xi \right) + \int_U \left( w \wedge \left( \psi^X_{-t} \right)^* \Theta_{-t}^{\ahom, X} \xi \right) = 0.
\end{equation*}
Applying the previous formula with $\omega = \di v$ and $\xi = \di w$ and using Lemma~\ref{propertiesSigmatX} gives
\begin{equation*}
\left| \int_{U} \di v \wedge \ahom \di \Sigma_{-t}^X w - \int_{U} \di \Sigma_t^X v \wedge \ahom \di  w \right|  \leq C    \left\| \di v \right\|_{L^2 \Lambda^{r}(U)}\left\| \di w\right\|_{L^2 \Lambda^{r}(U)}.
\end{equation*}
The proof of the inequality~\eqref{formipp2} is similar and we omit the details.

We then apply the estimate~\eqref{formipp} and~\eqref{formipp2} with the forms $v = u$ and $w = \Sigma_t^X u$. This gives
\begin{multline*}
\left| \int_{U} \di u \wedge \ahom \di \Sigma_{-t}^X \Sigma_{t}^X u - \int_{U} \di \Sigma_t^X u \wedge \ahom \di \Sigma_{t}^X u \right| + \left| \int_{U} \delta \Sigma_t^X u \wedge \star \delta  \Sigma_t^X u - \int_{U} \delta u \wedge \star \delta \Sigma_{-t}^X \Sigma_t^X u  \right| \\ \leq C  \left\| u \right\|_{H^1 \Lambda^{r-1}(U)}\left\| \Sigma_t^X u  \right\|_{H^1 \Lambda^{r-1}(U)},
\end{multline*}
where we used the estimate $ \left\| \di u \right\|_{L^2 \Lambda^{r}(U)} +  \left\| \delta u \right\|_{L^2 \Lambda^{r-2}(U)} \leq \left\| u \right\|_{H^1 \Lambda^{r-1}(U)}$. But, by the definition of the form $u$ given in Step 1, one has
\begin{equation*}
 \int_{U} \di u \wedge \ahom \di \Sigma_{-t}^X \Sigma_{t}^X u + \int_{U} \delta u \wedge \star \delta \Sigma_{-t}^X \Sigma_t^X u = \int_{U} \di f \wedge \ahom \di \Sigma_{-t}^X \Sigma_t^X u.
\end{equation*}
The term on the right-hand side can be estimated by the estimate~\eqref{formipp},
\begin{align*}
\left| \int_{U} \di f \wedge \ahom \di \Sigma_{-t}^X \Sigma_t^X u \right| & \leq  \left| \int_{U} \di  \Sigma_{t}^X f \wedge \ahom \di \Sigma_t^X u \right| + C  \left\| \di f \right\|_{L^2 \Lambda^{r}(U)}\left\| \Sigma_t^X u  \right\|_{H^1 \Lambda^{r-1}(U)} \\
													& \leq C  \left\| \di f \right\|_{H^1 \Lambda^{r}(U)}\left\| \Sigma_t^X u  \right\|_{H^1 \Lambda^{r-1}(U)}.
\end{align*}
Combining the few previous displays implies
\begin{equation*}
\int_{U} \di \Sigma_t^X u \wedge \ahom \di \Sigma_t^X u + \int_{U} \delta \Sigma_t^X u \wedge \star \delta  \Sigma_t^X u \leq C  \left( \left\| u \right\|_{H^1 \Lambda^{r-1}(U)}+  \left\| \di f \right\|_{H^1 \Lambda^{r}(U)} \right) \left\| \Sigma_t^X u  \right\|_{H^1 \Lambda^{r-1}(U)}.
\end{equation*}
By the version of the Gaffney-Friedrich inequality stated in Proposition~\ref{GFCDsmooth}, one obtains
\begin{equation*}
\left\| \Sigma_t^X u \right\|_{H^1 \Lambda^{r-1} (U)}^2 \leq C  \left( \left\| u \right\|_{H^1 \Lambda^{r-1}(U)} +  \left\| \di f \right\|_{H^1 \Lambda^{r}(U)} \right) \left\| \Sigma_t^X u  \right\|_{H^1 \Lambda^{r-1}(U)} + C  \left\| \Sigma_t^X u \right\|_{L^2 \Lambda^{r-1}(U)}^2.
\end{equation*}
The previous inequality can be further refined
\begin{equation*}
\left\| \Sigma_t^X u \right\|_{H^1 \Lambda^{r-1} (U)}^2 \leq C  \left( \left\| u \right\|_{H^1 \Lambda^{r-1}(U)} +  \left\| \di f \right\|_{H^1 \Lambda^{r-1}(U)} \right) \left\| \Sigma_t^X u  \right\|_{H^1 \Lambda^{r-1}(U)} +C  \left\| u \right\|_{H^1 \Lambda^{r-1}(U)}^2.
\end{equation*}
Consequently
\begin{equation} \label{step3mqiinboundaryH2}
\left\| \Sigma_t^X u \right\|_{H^1 \Lambda^{r-1} (U)} \leq C  \left( \left\| u \right\|_{H^1 \Lambda^{r-1}(U)} +  \left\| \di f \right\|_{H^1 \Lambda^{r-1}(U)} \right).
\end{equation}
\medskip

\textit{Step 4.} Interior regularity. Using the result of Step 3, we prove that $u$ is locally $H^2$ in the set $U$. To this end, we fix $x \in U$ and consider an open subset $V_x$ such that $x \in V_x \subseteq \bar V_x \subseteq U. $
Consider $d$ vector fields $X_1, \ldots, X_d$ compactly supported in $U$ such that, for each $k \in \{ 1 , \ldots, d \}$, and each point $y \in V$, $X_k (y) = e_k .$
We then recall the notation for the finite difference operator used in Proposition~\ref{H2interiorregularityestimate}: for $h > 0$ small, and $k \in \{ 1 , \ldots, d \}$, we denote by
\begin{equation*}
D^h_k u(x) = \frac{u(x+ he_k) - u(x)}{h}.
\end{equation*}
By the estimate~\eqref{step3mqiinboundaryH2}, we deduce that for $t > 0$ small enough and each $k \in \{ 1 , \ldots, d \}$,
\begin{equation*}
\left\| D^h_k u \right\|_{H^1 \Lambda^{r-1} (V)} \leq C  \left( \left\| u \right\|_{H^1 \Lambda^{r-1}(U)} +  \left\| \di f \right\|_{H^1 \Lambda^{r-1}(U)} \right).
\end{equation*}
Thus, according to~\cite[Lemma 7.24]{gilbarg2015elliptic}, the form $u$ belongs to the space $H^2 \Lambda^r (V)$ and we have the estimate
\begin{equation*}
\left\| u \right\|_{H^2 \Lambda^{r-1} (V)} \leq C  \left( \left\| u \right\|_{H^1 \Lambda^{r-1}(U)} +  \left\| \di f \right\|_{H^1 \Lambda^{r-1}(U)} \right).
\end{equation*}

\medskip

\textit{Step 5.} Boundary regularity I. The first part of this step is to reduce the problem to the half-ball denoted by $B^+ : = \left\{ x \in B(0,1) \, : \, x_n \geq 0\right\}$. We introduce the notation $B^+_{\frac 12} : = \left\{ x \in B\left(0,\frac12 \right) \, : \, x_n \geq 0\right\}$. 

Select $x \in \partial U$. Since the boundary $\partial U$ is assumed to be smooth there exists an open set $V \subseteq \Rd$ such that $x \in V$ and a smooth positively oriented diffeomorphism $\Phi : B(0,1) \rightarrow V$ such that
\begin{equation*}
\Phi \left( B^+ \right) =  V \cap \bar{U} \mbox{ and } \Phi(0) = x.
\end{equation*} 
Using the change of variables formula~\eqref{changevar} and the definition of the tangential trace~\eqref{tangentialtrace}, the following implication holds
\begin{equation*}
v \in H^{1}_{\di , 0} \Lambda^{r-1} (U) \implies  \mathbf{t} \Phi^* v =0~ \mbox{on} ~ \left\{ x \in B(0,1) \, : \, x_n = 0 \right\}.
\end{equation*}
To ease the notation, we denote by $u_\Phi :=  \left( \Phi \right)^* u$. It is a form defined on the set $B^+$. The purpose of this step is to prove that, for each $k \in \{ 1 , \ldots, d-1 \}$, the derivative $\partial_k \nabla u_\Phi$ belongs to the space $L^2 \Lambda^r \left(B^+\right)$.
\smallskip

As in the previous step, consider $(d-1)$ vector fields $X_1 , \ldots , X_{d-1}$ compactly supported in the half-ball $B^+$, tangent to the boundary of $B^+$ and satisfying
$
X_k (y) = e_k ,$ for each $y \in  B^+_{\frac 12}.
$
Note then that one has the identity, for each $k \in \{ 1 , \ldots, d-1 \}$,
\begin{equation} \label{est:1324jul}
 \left( \psi^{X_k}_t\right)^*u_\Phi = \left( \Phi \right)^* \left( \psi^{\tilde X_k}_t\right)^*  u ,
\end{equation}
where $\tilde X_k$ is the vector field defined on $V$ according to the formula
\begin{equation*}
\tilde X_k (\Phi(x) ) := \di \Phi (x) (X_k(x)), \quad \forall x \in B^+.
\end{equation*}
From the estimate~\eqref{est:1324jul}, we deduce
$
 \left( \Sigma^{X_k}_t\right)^*u_\Phi = \left( \Phi \right)^* \left( \Sigma^{\tilde X_k}_t\right)^*  u.
$
Thanks to the estimate~\eqref{step3mqiinboundaryH2}, one has
\begin{equation*}
\left\| \left( \Sigma^{\tilde X_k}_t\right)^*  u \right\|_{H^1 \Lambda^{r-1}(V)} \leq C  \left( \left\| u \right\|_{H^1 \Lambda^{r-1}(U)} +  \left\| \di f \right\|_{H^1 \Lambda^{r-1}(U)} \right),
\end{equation*}
for some constant $C := C(d , \lambda , \Phi, X_k) < \infty$. By the estimate~\eqref{phistarhscont}, the previous display can be further refined
\begin{equation*}
\left\|\left( \Phi \right)^* \left( \Sigma^{\tilde X_k}_t\right)^*  u \right\|_{H^1 \Lambda^{r-1}(V)} \leq C  \left( \left\| u \right\|_{H^1 \Lambda^{r-1}(U)} +  \left\| \di f \right\|_{H^1 \Lambda^{r-1}(U)} \right),
\end{equation*}
for some constant $C := C(d ,  \lambda , \Phi, X_k) < \infty$. This can be rewritten
\begin{equation} \label{1328eq:jul}
\left\| \left( \Sigma^{X_k}_t\right)^*u_\Phi \right\|_{H^1 \Lambda^{r-1}(V)} \leq C  \left( \left\| u \right\|_{H^1 \Lambda^{r-1}(U)} +  \left\| \di f \right\|_{H^1 \Lambda^{r-1}(U)} \right),
\end{equation}
for some constant $C := C(d , \lambda , \Phi, X_k) < \infty$. We then take the limit $t$ tends to $0$ in the estimate~\eqref{1328eq:jul} and apply~\cite[Lemma 7.24]{gilbarg2015elliptic} to get that, for each integer $k \in \{ 1 , \ldots, d-1 \}$, the derivative $\partial_k \nabla u_\Phi$ belongs to the space $L^2 \Lambda^r \left(B^+\right)$, with the estimate
\begin{equation*}
\left\| \partial_k \nabla u_\Phi \right\|_{L^2 \Lambda^{r-1}\left(B^+_{\frac 12}\right)} \leq  C  \left( \left\| u \right\|_{H^1 \Lambda^{r-1}(U)} +  \left\| \di f \right\|_{H^1 \Lambda^{r-1}(U)} \right),
\end{equation*}
for some constant $C := C(d , \lambda , \Phi) < \infty$. The proof of Step 5 is complete.

\medskip

\textit{Step 6.} Boundary regularity II. The purpose of this step is to prove that the form $u_\Phi$ belongs to the space $H^2 \Lambda^{r-1}\left( B^+_{1/2} \right)$. To this end, we see that, thanks to the previous step, there only remains to prove that $\partial_d \partial_d u_\Phi$ belongs to the space $L^2 \Lambda^{r-1}\left( B^+_{1/2} \right)$. This is what is proved in this step, along with the estimate
\begin{equation} \label{mainstep456}
\left\|\partial_d \partial_d u_\Phi \right\|_{ L^2 \Lambda^{r-1}\left( B^+_{1/2} \right)} \leq C  \left( \left\| u \right\|_{H^1 \Lambda^{r-1}(U)} +  \left\| \di f \right\|_{H^1 \Lambda^{r-1}(U)} \right),
\end{equation}
for some constant $C:= C(d ,\lambda, \Phi ) < \infty$. The main ingredient is the uniform ellipticity of the operator
$
\di \ahom \di  +  (-1)^r  \star \di \delta.
$
Since $u_\Phi =  \left( \Phi \right)^* u$, one sees that this differential form is a solution of the following equation
\begin{equation*}
\Phi^* \left( \di \ahom \di  +  (-1)^r  \star \di \delta \right) \left(\Phi^{-1} \right)^* u_\Phi =  \left( \Phi \right)^* \di \ahom \di  f \quad \mbox{in}~ B^+.
\end{equation*}
This second order differential operator can be written in the form
\begin{equation*}
\Phi^* \left( \di \ahom \di  +  (-1)^r  \star \di \delta \right) \left(\Phi^{-1} \right)^* u = \sum_{j,k = 1 }^{d}   A_{j,k}\partial_{j}  \partial_{k} u + \sum_{j=1}^{d}  A_j  \partial_{j} u + Au,
\end{equation*}
where the coefficients $A_{j,k},A_j  $ and $ A$ are smooth functions from the half-ball $B^+$ to the space of matrices of size $\binom{d}{r-1} \times \binom{d}{r-1}$ (or equivalently the space of endomorphisms of $\Lambda^{r-1} \left( \Rd \right)$). Since this operator is self-adjoint, one knows that the matrices $A_{i,j}$ are symmetric. The strategy to prove the estimate~\eqref{mainstep456} is to show that this operator is strongly elliptic, i.e.,
\begin{equation}\label{defstrongell}
 \sum_{j,k = 1 }^{d} \left( \eta^\intercal  A_{j,k}(x) \eta \right) \xi_j \xi_k \geq c |\eta|^2 |\xi|^2 \hspace{10mm} \forall x \in U, \forall \eta \in \R^{\binom{d}{r-1}}, \forall \xi \in \Rd.
\end{equation}
To prove the strong ellipticity, it is enough, by~\cite[Theorem 4.6]{mclean2000strongly}, to prove that, for each smooth form $w \in C_c^\infty \Lambda^{r-1} \left( B^+ \right)$,
\begin{equation} \label{H1coercive}
\int_{B^+} \Phi^* \left( \di \ahom \di  +  (-1)^r  \star \di \delta \right) \left(\Phi^{-1} \right)^* w \wedge w \geq c \left\| w  \right\|_{H^1\Lambda^{r-1} \left( B^+ \right) }^2 - C   \left\| w  \right\|_{L^2\Lambda^{r-1} \left( B^+ \right) }^2 .
\end{equation}
This is a consequence of the following computation
\begin{align*}
\lefteqn{\int_{B^+} \Phi^* \left( \di \ahom \di  +  (-1)^r  \star \di \delta \right) \left(\Phi^{-1} \right)^* w \wedge w} \qquad & \\ &  = \int_{V}  \left( \di \ahom \di  +  (-1)^r  \star \di \delta \right) \left(\Phi^{-1} \right)^* w \wedge \left(\Phi^{-1} \right)^* w \\
																				& = \int_{V}    \ahom \di  \left(\Phi^{-1} \right)^* w \wedge \di \left(\Phi^{-1} \right)^*  w  + \int_{V}     \delta \left(\Phi^{-1} \right)^* w \wedge \star \delta \left(\Phi^{-1} \right)^* w \\
																				& \geq \lambda \left\|  \di \left(\Phi^{-1} \right)^*  w  \right\|_{L^2 \Lambda^{r} \left( V \right)}^2 +  \left\|  \delta \left(\Phi^{-1} \right)^*  w  \right\|_{L^2 \Lambda^{r-2} \left( V \right)}^2.
\end{align*}
Since the form $w$ belongs to the space $C_c^\infty \Lambda^{r-1} \left( B^+ \right)$, we deduce that the form $\left(\Phi^{-1} \right)^* w$ belongs to the space $C_c^\infty \Lambda^{r-1} \left( V \right)$ and thus, by the Gaffney-Friedrich inequality,
\begin{equation*}
\left\|  \di \left(\Phi^{-1} \right)^*  w  \right\|_{L^2 \Lambda^{r} \left( V \right)}^2 + \left\|  \delta \left(\Phi^{-1} \right)^*  w  \right\|_{L^2 \Lambda^{r-2} \left( V \right)}^2 \geq c \left\|  \left(\Phi^{-1} \right)^* w \right\|_{H^1 \Lambda^{r} \left( V \right)}^2 - C  \left\|  \left(\Phi^{-1} \right)^* w \right\|_{L^2 \Lambda^{r} \left( V \right)}^2 .
\end{equation*}
We then note that
$
 \left\|  w \right\|_{H^1 \Lambda^{r} \left( V \right)}  \leq  C \left\|  \left(\Phi^{-1} \right)^* w \right\|_{H^1 \Lambda^{r} \left( B^+ \right)} 
$
and
$
 \left\|  \left(\Phi^{-1} \right)^* w \right\|_{L^2 \Lambda^{r} \left( B^+ \right)}  \leq  C \left\| w \right\|_{L^2 \Lambda^{r} \left( V \right)} 
$
for some constant $C:= C(d, \Phi) < \infty.$ This implies the estimate~\eqref{H1coercive}.
Now that one knows that the operator is strongly elliptic, one obtains, by \cite[Lemma 4.17]{mclean2000strongly}, that the coefficient $A_{d,d}$ has a uniformly bounded inverse. As a consequence, one has
\begin{align*}
 \left\| \partial_d \partial_d u_{\Phi} \right\|_{L^2 \Lambda^{r} \left( B^+ \right)} &\leq  \left\| A_{d,d} \partial_{d}  \partial_{d} u_{\Phi} \right\|_{L^2 \Lambda^{r} \left( B^+ \right)} \\
															& \leq \sum_{j=1}^{d}  \sum_{k = 1 }^{d-1}  \left\| A_{j,k}\partial_{j}  \partial_{k} u\right\| _{L^2 \Lambda^{r} \left( B^+ \right)} + \sum_{j=1}^{d}  \left\| A_j  \partial_{j} u\right\|_{L^2 \Lambda^{r} \left( B^+ \right)} \\& \qquad + \left\|  Au\right\|_{L^2 \Lambda^{r} \left( B^+ \right)} + \left\| \Phi^* \di \ahom \di f \right\|_{L^2 \Lambda^{d - r + 1}\left(B^+ \right)} .
\end{align*}
Using the main result of Step 3, this gives, for some constant $C:= C(d ,\lambda , \Phi) <\infty$,
\begin{equation*}
\left\| \partial_d \partial_d u_{\Phi} \right\|_{L^2 \Lambda^{r} \left( B^+ \right)} \leq C  \left( \left\| u \right\|_{H^1 \Lambda^{r-1}(U)} +  \left\| \di f \right\|_{H^1 \Lambda^{r}(U)} \right),
\end{equation*}
and the proof of Step 6 is complete.

\medskip
\textit{Step 7.} The main results of Steps 5 and 6 show that the function $u_{\Phi}$ belongs to the space $H^2 \Lambda^{r-1} \left(B^+ \right)$ together with the estimate, for some constant $C:= C(d ,\lambda , \Phi) <\infty$,
\begin{equation*}
\left\|  u_{\Phi} \right\|_{H^2 \Lambda^{r} \left( B^+ \right)} \leq C  \left( \left\| u \right\|_{H^1 \Lambda^{r-1}(U)} +  \left\| \di f \right\|_{H^1 \Lambda^{r}(U)} \right).
\end{equation*}
This implies 
\begin{equation} \label{almostglobest}
\left\|  u \right\|_{H^2 \Lambda^{r} \left( V  \cap U \right)} \leq C \left\|  u \right\|_{H^1 \Lambda^{r-1} \left(  U\right)}  + C \left\| \di f \right\|_{H^1 \Lambda^{r} \left( V \cap U \right)}.
\end{equation}
Since $\partial U$ is compact, we can cover $\partial U$ with finitely many open sets $V_1 , \ldots, V_N$. We sum the resulting estimates, along with the interior estimate proved in Step 3, and obtain $u \in H^2 \Lambda^{r} \left(  U \right)$ with the estimate, for some constant $C := C(d , \lambda , U) < \infty$,
\begin{equation*}
\left\| u \right\|_{H^2 \Lambda^{r} \left(U \right)} \leq C \left\|  u \right\|_{H^1 \Lambda^{r-1} \left(  U \right)}  + C \left\| \di f \right\|_{H^1 \Lambda^{r-1}(U)}.
\end{equation*}
We then simplify the right-hand side. Since we assumed $u \in \mathcal{H}^{r-1}_D(U)^\perp$, one has, by the Gaffney-Friedrich inequality stated in Proposition~\ref{GFCDsmooth},
\begin{equation*}
\left\| \nabla u \right\|_{L^2 \Lambda^{r-1} \left(  U \right)}  \leq C \left\| \di u \right\|_{L^2 \Lambda^{r} \left(  U \right)} + C \left\| \delta u \right\|_{L^2 \Lambda^{r-2} \left(  U \right)}.
\end{equation*}
This inequality can be further refined, thanks to the version of the Poincar\'e inequality stated in~\eqref{finalpoincareineqothharm}, into
\begin{equation*}
\left\|  u \right\|_{H^1 \Lambda^{r-1} \left(  U \right)}  \leq C \left\| \di u \right\|_{L^2 \Lambda^{r} \left(  U \right)} + C \left\| \delta u \right\|_{L^2 \Lambda^{r-2} \left(  U \right)}.
\end{equation*}
By the equality~\eqref{varregb2} and the ellipticity assumption~\eqref{ellipticityassumption}, one has
\begin{equation*}
\left\| \di u \right\|_{L^2 \Lambda^{r} \left(  U \right)}^2 + \left\| \delta u \right\|_{L^2 \Lambda^{r-2} \left(  U \right)}^2 \leq C \left\|  \di f  \right\|_{H^1 \Lambda^{r} \left(  U \right)}  \left\| u \right\|_{L^2 \Lambda^{d-r+1} \left(  U \right)}.
\end{equation*}
Combining the two previous displays with the estimate~\eqref{almostglobest} shows
\begin{equation*}
\left\|  u \right\|_{H^2 \Lambda^{r-1} \left(  U \right)} \leq C  \left\|  \di f  \right\|_{H^1 \Lambda^{r} \left(  U \right) } .
\end{equation*}
The proof of Step 7 is complete.

\medskip

\textit{Step 8.} The conclusion. Note that, if $u$ is a solution of the equation~\eqref{strongellpb}, then $f+u$ is a solution of the equation~\eqref{globalestpb}. Note also that, since two solutions of the equation~\eqref{globalestpb} differ by a form of~$C^{r-1}_{\di , 0}(U)$, they have the same exterior derivative. From these remarks and the previous estimate, one obtains the estimate~\eqref{globalestres}. The proof is complete.
\end{proof}

\small
\bibliographystyle{abbrv}
\bibliography{holes}

\end{document}